\documentclass[honours]{UNSWthesis}
\usepackage{amsfonts}
\usepackage{amssymb}
\usepackage{amsthm}
\usepackage{latexsym}
\usepackage{amsmath}
\usepackage{graphicx,pstricks}
\newcommand{\R}{\ensuremath{\mathbb{R}}}
\newcommand{\C}{\ensuremath{\mathbb{C}}}

\newcommand{\Z}{\ensuremath{\mathbb{Z}}}
\newcommand{\vecsxy}{\ensuremath{\mathbf{x}, \mathbf{y}}}
\newcommand{\veca}{\ensuremath{\mathbf{a}}}
\newcommand{\vecb}{\ensuremath{\mathbf{b}}}
\newcommand{\vece}{\ensuremath{\mathbf{e}}}
\newcommand{\vecu}{\ensuremath{\mathbf{u}}}
\newcommand{\vecv}{\ensuremath{\mathbf{v}}}
\newcommand{\vecw}{\ensuremath{\mathbf{w}}}
\newcommand{\vecx}{\ensuremath{\mathbf{x}}}
\newcommand{\vecy}{\ensuremath{\mathbf{y}}}
\newcommand{\vecz}{\ensuremath{\mathbf{z}}}
\newcommand{\vecalpha}{\ensuremath{\boldsymbol{\alpha}}}
\newcommand{\vecl}{\ensuremath{\boldsymbol{\lambda}}}
\newcommand{\veczero}{\ensuremath{\mathbf{0}}}
\newcommand{\xcrossy}{\ensuremath{\vecx \otimes \vecy}}

\newcommand{\ycrossz}{\ensuremath{\vecy \otimes \vecz}}

\newcommand{\sltwor}{\ensuremath{SL(2,\mathbb{R})}}
\newcommand{\sltwoz}{\ensuremath{SL(2,\mathbb{Z})}}
\newcommand{\slthreez}{\ensuremath{SL(3,\mathbb{Z})}}
\newcommand{\slthreer}{\ensuremath{SL(3,\mathbb{R})}}

\newcommand{\sldz}{\ensuremath{SL(d,\mathbb{Z})}}
\newcommand{\sotwor}{\ensuremath{SO(2,\mathbb{R})}}
\newcommand{\sonr}{\ensuremath{SO(n,\mathbb{R})}}
\newcommand{\slnr}{\ensuremath{SL(n,\R)}}
\newcommand{\slnz}{\ensuremath{SL(n,\Z)}}
\newcommand{\pslnr}{\ensuremath{PSL(n,\R)}}

\newcommand{\lfc}{\ensuremath{LF(\hat{\mathbb{C}})}}
\newcommand{\bce}{\footnotesize B.C.E.\normalsize}
\newcommand{\ce}{\footnotesize C.E.\normalsize}
\newcommand{\startm}{\begin{pmatrix}}
\newcommand{\finishm}{\end{pmatrix}}
\newcommand{\Rsphere}{\ensuremath{\hat{\mathbb{C}}}}
\newcommand{\lnorm}{\ensuremath{|\!|\!|}}
\newcommand{\matrixu}{\ensuremath{\startm 1 & 1 \\ 0 & 1 \finishm}}
\newcommand{\matrixv}{\ensuremath{\startm 0 & 1 \\ -1 & 0 \finishm}}
\newcommand{\Pn}{\ensuremath{P^{(n)}}}
\newcommand{\Ptwo}{\ensuremath{P^{(2)}}}
\newcommand{\Pd}{\ensuremath{P^{(d)}}}
\newcommand{\mand}{\hspace{5mm}\mbox{and}\hspace{5mm}}
\newcommand{\length}{\operatorname{length}}
\newcommand{\re}{\operatorname{Re}}
\newcommand{\im}{\operatorname{Im}}
\newcommand{\degree}{\operatorname{deg}}
\newcommand{\Span}{\operatorname{span}}
\theoremstyle{plain}
\newtheorem{theorem}{Theorem}[section]

\newtheorem{propn}[theorem]{Proposition}

\newtheorem{lemma}[theorem]{Lemma}
\newtheorem{corollary}[theorem]{Corollary}

\title{Hyperbolic Geometry  and  \\[1mm] Distance Functions on Discrete Groups}

\authornameonly{Anne Caroline Mary Thomas}

\author{\Authornameonly\\{\bigskip}Supervisors: Professor Michael Cowling and
Dr James Franklin}

\copyrightfalse
\figurespagefalse
\tablespagefalse

\begin{document}

\beforepreface

\prefacesection{Acknowledgements}

{\bigskip}Professor Michael Cowling has been a mentor for a number of
years.  He suggested the specific direction taken in this work, always answered
my questions clearly, and spent much time away from our meetings preparing
material.  Professor Cowling's thorough editing of the entire thesis is much
appreciated.

{\bigskip\noindent}Dr James Franklin's sincere commitment to this project, and his
willingness to challenge my thinking, have been extremely
valuable.  As well as supervising the historical part of this
work, Dr Franklin read the mathematical chapters, and helped with
the preparation of my seminar.  His suggestions led to important
improvements.

{\bigskip\noindent}I would like to thank Dr Ian Doust, our honours coordinator, for
always being available to provide both clear guidance on the preparation of this
project, and reassurance.  The School of Mathematics, as a
whole, has been an exciting and supportive environment in which to
study.

{\bigskip\noindent}I could not have completed this thesis without help from
my friends. In particular, I would like to thank David for
discussions which clarified several mathematical issues, and
Claire, Edwina, Lisa, Shaun and Stephanie for, at the appropriate
times, giving me a break from my work, and leaving me alone to get
on with it.

{\bigskip\noindent}My family have been patient and supportive throughout
my honours year.  I am especially grateful for the way they have
cared for me over the last few months.

{\bigskip\bigskip\bigskip\noindent}Anne Thomas, 12 June 2002.

\prefacesection{Introduction}

This thesis reflects my interest in the history of mathematics,
and in areas of mathematics which combine geometry and algebra.
The historical and mathematical sections of this work can be read
independently.  One of the main themes of the history, though, is
how in the nineteenth century a new understanding of mathematical
space developed, in which groups became vital to geometry. The
idea of using algebraic concepts to investigate geometric
structures is present throughout the work.

Chapter 1 is a short history of non-Euclidean geometry. This
chapter is a synthesis of my readings of mainly secondary sources.
Historians of mathematics have traditionally adopted an
`internalist' approach, explaining changes to the content of
mathematics only in terms of factors such as the imperative to
generalise, or the desire to remove inconsistencies. An
internalist narrative is certainly necessary for a satisfactory
history of mathematics, but it is not sufficient. First, the
history of mathematics is part of broader intellectual history.
Euclidean geometry held great philosophical prestige, and
non-Euclidean geometry challenged fundamental assumptions about
the nature of space and the truth value of mathematics. Second,
there are social factors in the history of mathematics.
Mathematicians themselves are part of a community, the
mathematical community.  The structure of this community in the
nineteenth century helps to explain the reception and
dissemination of non-Euclidean geometry. The sources consulted for this
chapter each adopted one or more of the approaches to the history
of mathematics outlined here, but none of them seemed to discuss
all the relevant issues.

Chapter 2 begins the strictly mathematical part of this thesis.
One aim of this chapter is to provide a deeper understanding of
some of the mathematics discussed in Chapter~1.  To this end, each
of the main models of hyperbolic geometry is presented, and important
results for each model proved.  The second aim of Chapter~2
is to describe thoroughly the way in which the action of the group
$PSL(2,\mathbb{Z})$ induces a tesselation of the upper half-plane.
Most of the material in Chapter 2 is selected and adapted from Chapters
3--6 of Ratcliffe~\cite{rat1:fhm}.

Chapter 3 poses a question about the upper half-plane, and then
provides the theory needed to frame the question precisely and
answer it in a variety of settings. Suppose $z$ is a point in the
upper half-plane, and $\gamma z$ the result of the action of an
element $\gamma \in PSL(2,\mathbb{Z})$ on the point $z$.  The
geodesic segment joining $z$ and $\gamma z$ has a finite length,
and crosses a finite number of tiles in the tesselation of the
upper half-plane.  The question is whether or not there is any
relationship between the length of the geodesic segment, and the
number of tiles it crosses.  We begin by defining a symmetric
space.  Then, if $X$ is a symmetric space, and $\Gamma$ a discrete
group of isometries of $X$, we define two distance functions on
the group~$\Gamma$.  The first is the geometric distance function,
and is induced by the distance function on the space $X$. The
second distance function on $\Gamma$ is the word distance
function.  The word distance function is induced by the generators
of $\Gamma$, and corresponds to the number of tiles in the
tesselation crossed by a geodesic segment.  We now ask whether or
not these two distance functions are equivalent.  We prove that,
when the tiles of the induced tesselation are compact, the two
distance functions are equivalent.  Then, in order to generalise
the action of $PSL(2,\mathbb{Z})$ on the upper half-plane, we
describe a symmetric space of matrices on which the group
$PSL(n,\mathbb{Z})$ acts.

Chapter 4 is devoted to the proof of Theorem~\ref{bigone}. This
theorem states that, when $n = 2$, the geometric and word distance
functions on $PSL(n,\mathbb{Z})$ are not equivalent, but that for
all $n \geq 3$, these two distances on $PSL(n,\mathbb{Z})$ are
equivalent.

The content of Chapter 3 and Chapter 4 is based upon two papers by Lubotzky,
Mozes and Raghunathan,~\cite{lmr1:cseg} and~\cite{lmr2:wrm}.  These papers are
not easy to read.  There is sometimes ambiguity about the hypotheses on the
space $X$ and the discrete group $\Gamma$ acting on $X$.  We have, here, assumed $X$
to be a symmetric space.  This assumption ensures that geometric distance
functions and word distance functions exist, and that structures such as
geodesics and fundamental domains behave nicely.  The authors do not prove the
`well-known' result that the geometric and word distance functions are
well-defined up to Lipschitz equivalence, so we fill this gap.  They do not
discuss the interesting geometric interpretation of the two distance functions,
in terms of a geodesic crossing tiles in the space.  Also, these papers
inaccurately refer to what we have called the geometric distance function as a
Riemannian metric. The more recent article,~\cite{lmr2:wrm}, was used only for
the proof of Theorem~\ref{compactfd}.  By assuming that $X$ is a symmetric
space rather than a `coarse path metric space', we have simplified the proof a
little, as we may choose the sequence of points $\{p_i\}_{i=0}^m$ to lie on the
geodesic.  Besides, it is not obvious that a `coarse path metric space' even
contains a point with a trivial stabiliser subgroup in $\Gamma$.  The proofs in
Chapter~4 are based on those in the earlier paper,~\cite{lmr1:cseg}, which
concerns Lipschitz equivalence for the group $PSL(d,\Z)$.  This paper does not
describe the symmetric space on which $PSL(d,\mathbb{Z})$ acts, so we have
added this material.  Also, we have, throughout Chapter~4, rewritten statements
of lemmas, propositions and so on, in order to clarify what each of the many
constants depends upon.   In the paper, for the case $d = 2$, the claim that
the word length of $u^n$ grows linearly in $n$ is not proved.  For the
cases $d \geq 3$, there is no explicit statement of Theorem~\ref{normthm}, and
no indication of how to combine the results of Corollary~\ref{product2} and
Corollary~\ref{word_log} to prove the result of Theorem~\ref{normthm}.  In the
proof of Proposition~\ref{Eij_is_U1},
results~\ref{syndetic}--\ref{lengthbounds} are stated without proof.  In
Proposition~\ref{product2}, the construction suggested by the authors does not
actually satisfy the statement of the proposition, and there is no proof of the
claims of Lemmas~\ref{K1lemma} and~\ref{K3lemma}. Finally, the way in which the
result of Theorem~\ref{normthm} implies that the word distance function is
bounded above by the geometric distance function is not proved.

\afterpreface

%
%


\chapter{A History of Non-Euclidean Geometry}\label{history}


\begin{quote} \emph{I entreat you, leave the doctrine of parallel lines alone;
you should fear it like a sensual passion; it will deprive you of health,
leisure and peace---it will destroy all joy in your life.}
\begin{flushright}
from a letter by Farkas Bolyai to his son J\'anos.\footnote{Quoted in B. A.
Rosenfeld. \emph{A History of Non-Euclidean Geometry: Evolution of the Concept of
a Geometric Space}. Springer-Verlag, 1988, p.~108.}
\end{flushright}
\end{quote}

\begin{quote}
\emph{I have now resolved to publish a work on parallels $\ldots$
the goal is not yet reached, but I have made such wonderful
discoveries that I have been almost overwhelmed by them $\ldots$ I
have created a new universe from nothing.}
\begin{flushright}
from a letter by J\'anos Bolyai to his father Farkas.\footnote{Quoted in Jeremy
Gray. \emph{Ideas of Space: Euclidean, Non-Euclidean and Relativistic}.
Clarendon Press, 1989, p.~107.}
\end{flushright}
\end{quote}

\section{Introduction}

This chapter is a short history of non-Euclidean geometry.  The intellectual
history of developments internal to mathematics is presented, together with
discussion of relevant social and philosophical trends within  the mathematical
community and in wider contexts.  Chapter 2 will provide a rigorous and modern
treatment of many of the mathematical terms and results referred to here.  In
order to properly acknowledge sources, the referencing style adopted for
Chapter 1 is different from that in the rest of this thesis.

The discovery and acceptance of non-Euclidean geometry forms a crucial strand
of the history of mathematics.  For thousands of years, Euclidean geometry
held great prestige.  However, there were small doubts about the parallel axiom.
After numerous attempts to validate this axiom, in the nineteenth century there
was a shift to advancing alternative geometries. The coming-together of
developments in many areas of mathematics aided the acceptance of
non-Euclidean geometry by the mathematical community.  In this process
mathematicians were forced to reconsider and alter their ideas about the nature
of mathematics and its relationship with the real world.

\section{Euclid's \emph{Elements}}

Although Egyptian and Babylonian mathematicians had posed and solved geometric
problems, geometry only acquired logical structure with the
Greeks.\footnote{Carl B. Boyer. \emph{A History of Mathematics}. John Wiley \&
Sons, Inc., 1991, p.~47.}  Greek geometers were the first to write proofs as
deductive sequences of statements.\footnote{Roberto Torretti. \emph{Philosophy
of Geometry from Riemann to Poincar\'e}. D. Reidel Publishing Company, 1984,
p.~2.} Euclid's thirteen-volume \emph{Stoichia (Elements)}, which appeared in
about 300~\bce, was an exposition of the fundamentals of classical Greek
mathematics. The first six books covered plane geometry, and the last three
solid geometry.

The \emph{Elements} opens with a list of 23 definitions, followed by ten
axioms.\footnote{The ten axioms are sometimes called five postulates and five
common notions; surviving manuscripts of the \emph{Elements} are not
consistent. See Boyer, \emph{History of Mathematics}, p.~105.} Propositions and their proofs fill the
remainder of the work. This logical structure reflected the efforts by
mathematicians of the classical period to establish geometry as a deductive system
dependent on as few assumptions as possible.\footnote{Torretti,
\emph{Philosophy of Geometry}, p.~5.}

Euclid's fifth axiom, which is also known as the parallel axiom, states:
\begin{quote} If a straight line falling on two straight lines makes the
interior angles on the same side less than two right angles, the two straight
lines, if produced indefinitely, meet on that side on which the angles are less
than the two right angles.\footnote{Quoted in Boyer, \emph{History of
Mathematics}, p.~106.} \end{quote} To understand this, suppose that the two
straight lines are $l$ and $m$, and that the transversal falling on them is line
$t$.  The axiom is saying that if the interior angles on the same side of $t$
sum to less than $\pi$, the three lines $l$, $m$ and $t$ form a triangle.  This
triangle has two vertices on the line $t$. Its third vertex lies in the
half-plane bounded by the line $t$, and on the same side of $t$ as the interior
angles which sum to less than $\pi$.

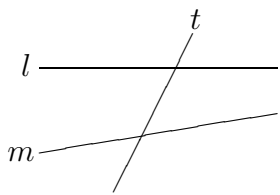
\begin{picture}(90,80)
\put(160,13){\line(6,1){90}} \put(160,45){\line(1,0){90}}
\put(218,58){\line(-1,-2){30}} \put(153,42){$l$} \put(148,10){$m$}
\put(217,60){$t$}
\end{picture}

\medskip

\noindent The parallel axiom does not say anything about the possibility that
the pairs of interior angles on both sides of the transversal sum to $\pi$.
With the word parallel defined as meaning `never intersecting', Euclid proved,
without using the fifth axiom, that if a transversal falling on two straight
lines makes the interior angles on the same side add to $\pi$, then the two
lines are parallel. To prove the converse of this statement, Euclid did need to
use the parallel axiom. The parallel axiom was also essential to the proof of
the key theorem that the angle sum of a triangle is $\pi$.

The \emph{Elements} made little use of the idea of motion in definitions or
methods.\footnote{Rosenfeld, \emph{History of Non-Euclidean Geometry}, pp.
110--112.}  In the sixth century~\bce, Pythagoreans had defined a line as the
trace of a moving point, and a surface as the trace of a moving line. But Euclid
largely conformed to the views of Aristotle (384--322~\bce), who rejected motion
in geometry because he considered mathematical objects to be abstractions of
physical objects. So, in the \emph{Elements}, a line is a ``breadthless length''
rather than the successive positions of a moving point.\footnote{Gray,
\emph{Ideas of Space}, p.~28.}  Euclid did use superposition of geometric
figures in the proofs of Propositions 4 and 8, but chose to prove Proposition
26, where he could also have used superposition, with a longer method instead.
This suggests that he was unhappy with the method of
superposition.\footnote{Morris Kline. \emph{Mathematical Thought from Ancient
to Modern Times}. Oxford University Press, 1972, p.~60.}

Euclidean geometry contained many ambiguities and hidden assumptions.  The
opening definitions, such as ``a surface is that which has length and breadth
only'', served no logical purpose, since Euclid provided no prior set of
undefined concepts in terms of which definitions could be
framed.\footnote{Boyer, \emph{History of Mathematics}, p.~105.}  Euclid's
second and third axioms, \begin{quote} 2. [It is possible] to produce a finite
straight line continuously in a straight line, \newline 3. [It is possible] to
describe a circle with any centre and radius,\footnote{Quoted in Kline,
\emph{Mathematical Thought}, p.~59.} \end{quote} embody the assumption that
geometric space is infinite.  Greek geometers also assumed that space was
homogeneous, so that any construction could be performed at any point with the
same results.\footnote{Gray, \emph{Ideas of Space}, p.~26.} The congruence of
figures depends on this property of space.  The existence of parallel lines was
also taken for granted.\footnote{Ibid, p.~29.}

Despite these logical weaknesses, the Greeks and many later thinkers
considered  the \emph{Elements} to be the model for philosophical inquiry.
After the fall of Rome, the survival and development of Greek
mathematics, including Euclidean geometry, depended on Arabic-speaking
mathematicians in Islamic Africa and Asia.  In the twelfth century, Arabic
translations of the \emph{Elements} were translated in turn into Latin.  The
\emph{Elements} were first printed in 1482, and versions in various European
vernaculars appeared in the sixteenth century.  Euclid's \emph{Elements}
achieved
huge influence in Europe, with more than one thousand editions
published.\footnote{Boyer, \emph{History of Mathematics}, p.~119.} Until about
1800, Euclidean geometry enjoyed secure philosophical and mathematical status.

\section{Doubts about the parallel axiom}

From early on there were, though, questions raised about the parallel axiom.
A lot of basic geometry, such as the angle sum of a triangle, depended on it,
so its soundness was very important. Yet the parallel axiom appeared much less
self-evident than the others.  Euclid's formulation of the axiom seemed
long-winded and complicated, and it was not clear that assuming the existence
of infinite straight lines was reasonable.\footnote{Kline, \emph{Mathematical
Thought}, p.~863.} Mathematicians did not seriously doubt the truthfulness of
the axiom until the nineteenth century; they just wanted to make sure that it
was formally unassailable. Proclus (410--485~\ce), for example, wrote that the
parallel axiom  \begin{quote} \dots ought to be struck from the axioms altogether.
For it is a theorem --- one that invites many questions.\footnote{Proclus. \emph{A
Commentary on the First Book of Euclid's Elements}. Princeton University Press,
1969, p.~150.}
\end{quote}

Many mathematicians tried to settle uncertainty about the parallel axiom by
proving it from the other nine axioms.  They were all unsuccessful, as
they relied on implicit and unsupported assumptions, some of them actually
equivalent to the parallel axiom. In his `proof', Ptolemy (d.~168~\ce) made
tacit assumptions about parallel lines. Proclus found the flaws in Ptolemy's
work, but then based his own efforts on a rather dubious Aristotelian axiom
about the finiteness of the universe, which was inconsistent with Euclidean
assumptions of the infinitude of space.\footnote{Kline, \emph{Mathematical
Thought}, pp.~863--864.} Thabit ibn Qurra (836--901) wrote two treatises trying to
prove the parallel axiom.\footnote{Rosenfeld, \emph{History of Non-Euclidean
Geometry}, pp.~49--56.} In the first, he assumed that if two lines diverge on one
side of a transversal, they converge on the other side.  In the second, he used
kinematic arguments, asserting that  \begin{quote} If any solid is imagined to
move as a whole in one direction with one simple and straight movement, then
every point in it will have a straight movement and will thus draw a straight
line on which it will pass.\footnote{Quoted in ibid., p.~52.} \end{quote} Thus,
the curve at a fixed distance from a straight line is itself straight.  Ibn
al-Haytham (965--1041) gave a false proof of this last statement, then used it to
prove the parallel axiom.  The first substantial European effort on parallels
was that of John Wallis (1616--1703), who realised that many of Euclid's proofs
depended on unstated assumptions.\footnote{Gray, \emph{Ideas of Space}, p.~57.}
He derived the parallel axiom by assuming the existence of triangles similar to
a given triangle. Other European attempts included those of Giordano Vitale
(1633--1711), and A.~M.~Legendre (1752--1833).

A different approach was to attempt to replace or reformulate the parallel
axiom. Ptolemy, Proclus, and the Arabian editor of Euclid,
Naz\^{\i}r-Edd\^{\i}n (1201--1274), as well as the Europeans Wallis, Joseph
Fenn, John Playfair (1748--1819) and Legendre, all attempted substitute axioms.
The version generally given in modern textbooks is due to Playfair: ``Through a
given point $P$ not on a line $l$ there exists only one line in the plane of
$P$ and $l$ which does not meet $l$''.\footnote{Quoted in Kline,
\emph{Mathematical Thought}, p.~865.} All these efforts, though, were
unsatisfactory, relying on assertions about infinity which were no more
self-evident than those in Euclid's version.\footnote{Kline, \emph{Mathematical
Thought}, p.~866.}

A third strategy was to investigate the consequences of denying the parallel
axiom. Aristotle connected negating the axiom with triangles having an angle
sum greater than $\pi$.\footnote{Gray, \emph{Ideas of Space}, pp.~33--34.}  The
resulting geometry is actually spherical geometry.  From a modern point of
view, there are no parallels  in spherical geometry, because all `lines'
(great circles) intersect. However, despite its antiquity, and its many
applications in navigation, spherical geometry  was never perceived as a
challenge to Euclidean geometry.\footnote{Rosenfeld, \emph{History of
Non-Euclidean Geometry}, pp.~1--33.} This was probably because spherical
geometry was not considered to be a geometry of the plane.\footnote{Gray,
\emph{Ideas of Space}, p.~71.} For example, in the first work in the area, the
\emph{Sphaerica} of Theodosius (\emph{c.}~20~\bce), propositions were formulated
mostly in terms of planes intersecting the sphere, rather than in terms of the
sphere's intrinsic geometry.

Many mathematicians who considered negating the parallel axiom found that their
conclusions were counter-intuitive. Common sense and everyday experience
favoured Euclidean geometry. Some abandoned their efforts. For instance, Omar
Khayyam (1048?--1122) entertained the possibility of a quadrilateral with angle
sum different from $2\pi$, but straight away dismissed the idea as
inconsistent.

Girolamo Saccheri (1667--1733), though, investigated in some depth the geometry
which resulted from denying the parallel axiom.  He was the first to try to
prove that a contradiction would follow from supposing the parallel axiom
false.  Saccheri began by establishing that $\angle ACD = \angle BDC$ in the
following quadrilateral.

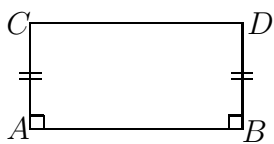
\begin{picture}(90,60)
\put(160,5){\framebox(80,40)} \put(160,10){\line(1,0){5}}
\put(165,10){\line(0,-1){5}} \put(235,5){\line(0,1){5}}
\put(235,10){\line(1,0){5}} \put(156,24){\line(1,0){8}}
\put(156,26){\line(1,0){8}} \put(236,24){\line(1,0){8}}
\put(236,26){\line(1,0){8}} \put(151,1){$A$} \put(240,0){$B$}
\put(151,41){$C$} \put(242,41){$D$}
\end{picture}

\medskip

\noindent He then considered three hypotheses about angles $ACD$ and $BDC$:
they could both be obtuse, or both right angles, or both
acute.\footnote{Rosenfeld, \emph{History of Non-Euclidean Geometry},
pp.~98--99.} He showed that these three hypotheses implied that the angle sum
of a triangle was respectively more than, equal to, or less than $\pi$.
Implicitly assuming that straight lines could be extended indefinitely, and so
ruling out spherical geometry, Saccheri then rejected the obtuse angle
hypothesis, by showing that it implied the parallel axiom.  This was a
contradiction, since the parallel axiom implies that $ACD$ and $BDC$ are right
angles. He then considered the acute angle hypothesis. Assuming that it held,
Saccheri proved several theorems, and found that two straight lines might
diverge in one direction and come asymptotically close together in the other,
so that they had a common point at infinity. Although he found this idea
repugnant, he only rejected the acute angle hypothesis after deriving a
contradiction from an error he made in computing arc length.\footnote{Ibid.,
p.~99.}

Because there had been no rigorous invalidation of the results of negating the
parallel axiom, a few mathematicians began to speculate that the alternative
geometry might hold somewhere outside everyday experience.

J. H. Lambert (1728--1777) considered the same hypotheses as Saccheri, and like
Saccheri obtained from the obtuse angle hypothesis a contradiction. In
Euclidean geometry, he observed, there is an absolute measure of angles, but
length varies according to the unit of length selected. Under the acute angle
hypothesis Lambert found he could actually define an absolute unit of length,
by associating one line segment with one angle.  He also saw that the area of a
triangle was proportional to its angle sum. Counter-intuitive though such
consequences seemed, Lambert was unable to reject the acute angle hypothesis.
He supposed that it might hold on some ``imaginary sphere'', a concept he never
explained clearly.\footnote{Ibid., p.~101.}

Ferdinand Karl Schweikart (1780--1859) and his nephew Franz Adolf Taurinus
(1794--1874) were influenced by Saccheri and Lambert's work.  Assuming the angle
sum of a triangle was not two right angles, they obtained what they called
``astral geometry'', since it might occur in the region of the stars.
Schweikart and Taurinus were the first to use analysis in non-Euclidean
geometry, and this enabled progress past the classical approach of Saccheri and
Lambert.\footnote{Gray, \emph{Ideas of Space}, p.~97.}  Their analytic results
included trigonometric formulae, and the area of a triangle with a vertex at
infinity, in terms of a constant $K$.  They derived their trigonometric formulae
from those of spherical trigonometry, by replacing ordinary trigonometric
functions with hyperbolic trigonometric functions. The radius of the sphere
concerned was $K$.

Carl Friedrich Gau\ss\ (1777--1855) did a little work on non-Euclidean geometry
at the same time as Schweikart and Taurinus.\footnote{Rosenfeld, \emph{History
of Non-Euclidean Geometry}, pp.~214--215.}  He never published anything in the
area. According to his letters and rough notes, Gau\ss\ realised that rejecting
the parallel axiom implied the existence of an absolute measure of length, and
doubted whether Euclidean geometry could be proved.

\section{Systems of non-Euclidean geometry}

Nikolai Ivanovich Lobachevsky (1792--1856) and J\'anos Bolyai (1802--1860) were
the first people to publish, without any reservations, research on a system of
non-Euclidean geometry.\footnote{Torretti, \emph{Philosophy of Geometry},
p.~53.} They worked independently.\footnote{For this reason I will refer to
their geometry, in which the angle sum of a triangle is less than~$\pi$, as
hyperbolic geometry rather than using one or both of their names.}  Lobachevsky
initially set out his ideas on what he called ``imaginary geometry'' in a
presentation at the University of Kazan in 1826. (The city of Kazan lies on the
Volga River in central Russia.) Lobachevsky published papers in Kazan journals
in 1829 and 1835--37, and in 1837 a third paper appeared in \emph{Journal
f\"{u}r Mathematik}. Lobachevsky also published books on his theories, in 1840
and 1855.  As for J\'anos Bolyai, until 1821 he  agreed with his father Farkas
that the parallel axiom must be true. Over the next two years, though, J\'anos
investigated the alternative geometry, arriving at many of the same results as
Lobachevsky. In 1832 J\'anos published his work on what he called ``absolute
geometry'', in an appendix to his father's book on the theory of parallels.

In their systems, Lobachevsky and Bolyai rejected the parallel axiom but
regarded all  Euclidean propositions proved without using this axiom as still
being valid.\footnote{Torretti, \emph{Philosophy of Geometry}, p.~40.} In
Lobachevsky's account,\footnote{discussed in Kline, \emph{Mathematical
Thought}, pp.~874--875.} if $C$ is a point at a perpendicular distance
$a$ from the line $AB$, there exists an angle $\pi(a)$ such that all lines
through $C$ which make an angle less than $\pi(a)$ with the perpendicular $CD$
will intersect $AB$; all other lines through $C$ do not intersect $AB$. In the
Euclidean sense of the word parallel, there are infinitely many parallels
through $C$.

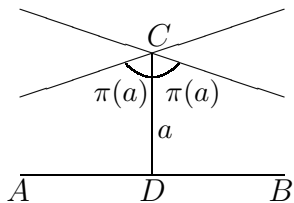
\begin{picture}(90,75)
\put(150,9){\line(1,0){100}} \put(200,9){\line(0,1){46}}
\put(200,55){\line(3,1){50}} \put(200,55){\line(3,-1){50}}
\put(200,55){\line(-3,-1){50}} \put(200,55){\line(-3,1){50}}
\qbezier(190,51)(200,41)(210,51) \put(145,-1){$A$} \put(195,-1){$D$}
\put(244,-1){$B$} \put(198,58){$C$} \put(178,37){\small{$\pi(a)$}}
\put(205,37){\small{$\pi(a)$}} \put(202,23){\small{$a$}}
\end{picture}

\noindent Other results of Lobachevksky and Bolyai included trigonometry, arc
length $ds$, the area of a circle, and theorems on the area of plane regions
and volumes of solids.

The founders of non-Euclidean geometry were not formalists, and  did not
consider their work to be an investigation of an axiomatic system with only one
difference from that of Euclid.\footnote{Torretti, \emph{Philosophy of
Geometry}, p.~61.} This is shown by their willingness to consider the
possibility that their geometry was the geometry of real physical space.  When
Lobachevsky wrote `straight  line', he meant straight in the ordinary sense of
the word, showing that he regarded his geometry as (at least potentially)
real.\footnote{Ibid., pp.~62--63.}  Indeed, the discovery of these different geometries
made  it seem possible that Euclidean geometry might be proved inexact by
experiment, at some point.\footnote{Ibid., p.~254.} Bolyai commented that if
non-Euclidean geometry were real, a measurement would determine the size of the
constant $K$, the radius of the imaginary sphere.  Since Euclidean geometry had
been used so successfully in physics, if non-Euclidean geometry were indeed the
geometry of the real world, then any differences would only be detectable on a very
large scale.\footnote{Ibid., p.~63.} Lobachevsky thus suggested using the
parallax of stars to decide which geometry held. His measurements were
inconclusive.   According to some sources, Gau\ss\ measured distances between
mountain peaks for the purpose of deciding between geometries.\footnote{For
example, Kline, \emph{Mathematical Thought}, pp.~872--873.} It seems more likely
that Gau\ss\ realised that the distances involved were too small, and he
actually made these measurements for his research in geodesy.\footnote{Gray,
\emph{Ideas of Space}, p.~123.}  This is certainly suggested by his writing
about the constant $K$: ``in the light of our astronomical experience, the
constant must be enormously larger than the radius of the
earth''.\footnote{Quoted in Torretti, \emph{Philosophy of Geometry}, p.~63.}

\section{Reception of non-Euclidean geometry}

Until the 1860s, the European mathematical community did not regard
investigations into non-Euclidean geometry as respectable, and non-Euclidean
geometry was separate from the rest of mathematics. Only about 1870 was the
work of Lobachevsky and Bolyai recognised. There were social reasons for this
reluctant reception.  As well, there were  mathematical issues which had to be
resolved, and a system of alternative geometry raised many philosophical
controversies.

\subsection{Social factors}

During the nineteenth century mathematics became professionalised, and
important mathematical communities developed at centres including Paris, Berlin
and G\"ottingen.  Bolyai, as a Hungarian army officer, was socially and
geographically a long way from these mathematical hubs. Lobachevsky's 1829
paper was published in a very obscure journal, got bad reviews and was regarded
as incomprehensible. His article in the mainstream \emph{Journal f\"ur
Mathematik} was virtually impossible to read, and depended on results in the
1829 paper. Gau\ss, at G\"ottingen, approved of Lobachevsky's 1829 paper and
Bolyai's 1832 article, but only praised them in letters to friends.

Because of Gau\ss's prestige, when his correspondence was published in the
early 1860s, these comments made a strong impression on the European
mathematical community, and led to intensive discussion of new geometrical
ideas. In 1865, Arthur Cayley (1821--1895) published a commentary on
Lobachevsky, which did much to disseminate his work. In the years 1866 and
1867, French, German, Italian and Russian translations and publications of
Lobachevsky's work appeared, so that by 1870 Lobachevsky and Bolyai's research
was known to geometers all across Europe. From the 1890s, there were
university courses and textbooks on non-Euclidean geometry, showing that it was
well-accepted by the mathematical community.

\subsection{Mathematical developments}

Internal developments in mathematics were needed before non-Euclidean geometry
could be recognised, because in the 1830s there were sound mathematical reasons
for holding reservations about the new geometry. Neither Lobachevsky nor Bolyai
proved that their geometry was consistent, that is, that it did not contain any
contradictions. Lobachevsky was convinced that he had established consistency,
since he had obtained his trigonometric formulae from those of spherical
trigonometry by multiplying the sides of the triangle by the complex number
$i$. His argument was in fact insufficient.  It only showed that his
trigonometric formulae
followed from the assumptions he had made in setting up hyperbolic
geometry.\footnote{Rosenfeld, \emph{History of Non-Euclidean Geometry}, p.~227.}

Models of hyperbolic geometry were critical to overcoming these mathematical
reservations. The models found by the mathematicians Beltrami, Klein and
Poincar\'e were structured collections of objects satisfying a set of
mathematical statements, given a suitable, if unusual, interpretation of the
key terms such as `straight line'.\footnote{Torretti, \emph{Philosophy of
Geometry}, p.~132.} The (relative) consistency of non-Euclidean geometry was
proved using a model, and models aided the acceptance of the new geometry in
other ways.  They provided a method of visualising concepts contrary to
intuition, such as the existence of infinitely many parallels through a point.
Models also made it possible to apply non-Euclidean geometry to solving problems
in other areas of mathematics.  This demonstrated its usefulness.

In the 1830s, Ferdinand Minding (1806--85) had investigated surfaces of constant
negative curvature, such as the pseudosphere, which is obtained by rotating a
tractrix about a vertical axis. For such surfaces, Minding obtained the same
trigonometric formulae as in hyperbolic geometry, and published them in 1840.

Minding's work was largely forgotten until Eugenio Beltrami (1835--1900) noticed
that Minding's trigonometric formulae were identical to Lobachevsky's. He thus
realised that non-Euclidean plane geometry could be regarded as the geometry of
a surface of constant negative curvature.  In his model of 1868, Beltrami
identified the hyperbolic plane with the interior of a fixed circle in the
Euclidean plane. Hyperbolic straight lines were identified with open chords of
this circle, and points at infinity with the circle itself.  By constructing
this first model of the hyperbolic plane, Beltrami established the relative
consistency of non-Euclidean geometry.  This was because his model was defined
wholly in terms of Euclidean geometry, and he proved, using differential
geometry, that it satisfied all the axioms of hyperbolic geometry. Thus, the new
geometry was consistent if Euclidean geometry was
consistent.\footnote{John G. Ratcliffe. \emph{Foundations of Hyperbolic
Manifolds}. Springer-Verlag, 1994, p.~7.}

Felix Klein (1849--1925) interpreted Beltrami's model in terms of projective
geometry in 1871.  He modelled hyperbolic geometry as metric projective
geometry in the interior of some conic. Beltrami's model was the special case
where the conic was a circle.

In 1882, Henri Poincar\'e (1854--1912) showed that linear fractional
transformations with real coefficients preserve the complex upper half-plane.
Poincar\'e's model is this upper half-plane with arc length defined as $ds =
\sqrt{dx^2 + dy^2}/y$, and geodesics as vertical lines or semicircles orthogonal
to the real axis. Distance is then arc length along the geodesic joining two
points. Poincar\'e's model allows groups of linear fractional transformations
to be represented as discrete groups of orientation-preserving isometries of
the hyperbolic plane. The model can be mapped to the unit disc, and in later
works Poincar\'{e} extended his model to three dimensions, and discussed a
two-sheeted hyperboloid model as well.

Using these models, Poincar\'{e} found many applications for hyperbolic
geometry, which helped the new geometry to seem useful and
familiar.\footnote{Scott Walter. The non-Euclidean style of Minkowskian
relativity. In Jeremy Gray, editor, \emph{The Symbolic Universe: Geometry and
Physics 1890--1930}, pp. 91--127. Oxford University Press, 1999, p.~92.}
Indeed, Poincar\'e's work showed that hyperbolic geometry was already part of
mainstream mathematics.\footnote{John Stillwell. Introduction to Poincar\'e's
`Theory of Fuchsian groups', `Memoir on Kleinian groups', `On the applications
of non-Euclidean geometry to the theory of quadratic forms'. In John Stillwell,
editor, \emph{Sources of Hyperbolic Geometry}, pp.~113--122. American Mathematical Society,
1996, p. 113}  For example, he used hyperbolic
geometry to prove the existence of linear fractional transformations.
He then studied the group of such transformations, using the
tesselation of the upper half-plane pictured here.
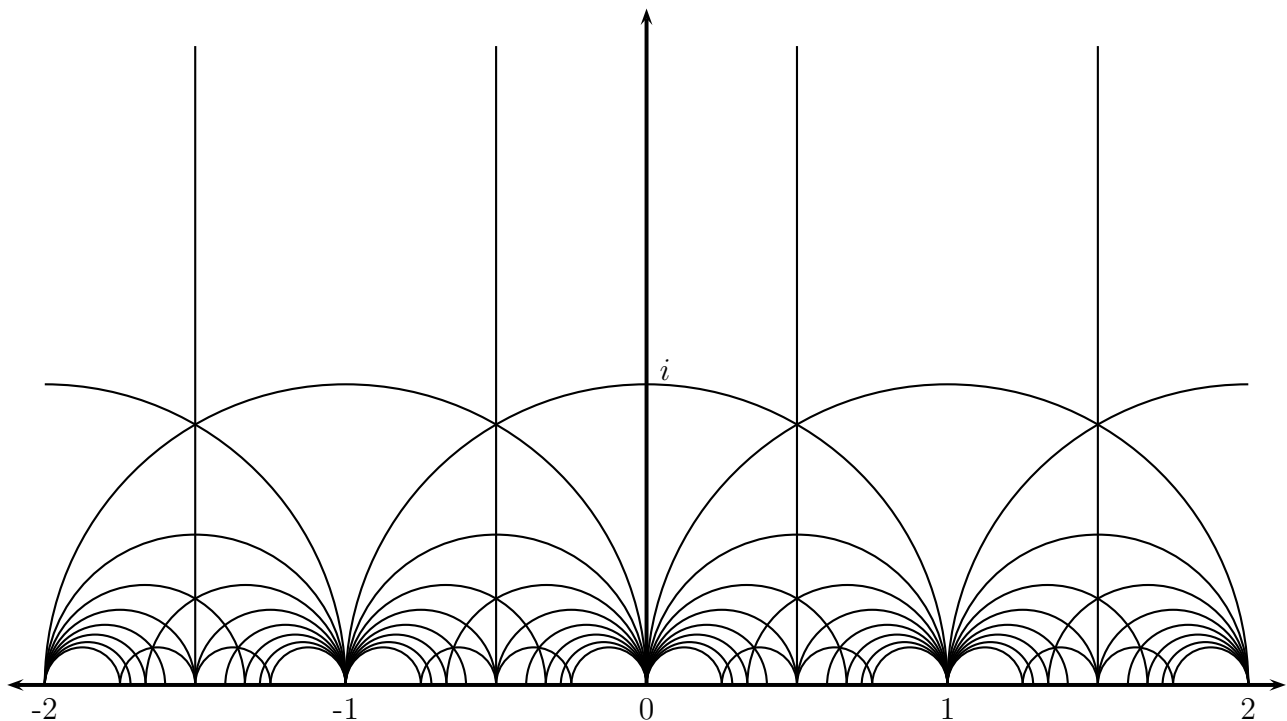
\label{tess}
\begin{center}
\begin{pspicture}(0,-0.4)(16,9)
\psline[linewidth=1.5pt]{<->}(-0.5,0)(16.5,0)
\psline[linewidth=1.5pt]{->}(8,0)(8,9)
\psline(2,0)(2,8.5)
\psline(6,0)(6,8.5)
\psline(10,0)(10,8.5)
\psline(14,0)(14,8.5)
\uput[d](0,0){-2}
\uput[d](4,0){-1}
\uput[d](8,0){0}
\uput[d](12,0){1}
\uput[d](16,0){2}
\uput[r](8,4.2){$i$}
\psarc(0,0){4}{0}{90}
\psarc(4,0){4}{0}{180}
\psarc(8,0){4}{0}{180}
\psarc(12,0){4}{0}{180}
\psarc(16,0){4}{90}{180}
\psarc(2,0){2}{0}{180}
\psarc(6,0){2}{0}{180}
\psarc(10,0){2}{0}{180}
\psarc(14,0){2}{0}{180}
\psarc(1,0){1}{0}{180}
\psarc(3,0){1}{0}{180}
\psarc(5,0){1}{0}{180}
\psarc(7,0){1}{0}{180}
\psarc(9,0){1}{0}{180}
\psarc(11,0){1}{0}{180}
\psarc(13,0){1}{0}{180}
\psarc(15,0){1}{0}{180}
\psarc(1.33,0){1.33}{0}{180}
\psarc(5.33,0){1.33}{0}{180}
\psarc(9.33,0){1.33}{0}{180}
\psarc(13.33,0){1.33}{0}{180}
\psarc(2.67,0){1.33}{0}{180}
\psarc(6.67,0){1.33}{0}{180}
\psarc(10.67,0){1.33}{0}{180}
\psarc(14.67,0){1.33}{0}{180}
\psarc(0.8,0){0.8}{0}{180}
\psarc(4.8,0){0.8}{0}{180}
\psarc(8.8,0){0.8}{0}{180}
\psarc(12.8,0){0.8}{0}{180}
\psarc(3.2,0){0.8}{0}{180}
\psarc(7.2,0){0.8}{0}{180}
\psarc(11.2,0){0.8}{0}{180}
\psarc(15.2,0){0.8}{0}{180}
\psarc(0.67,0){0.67}{0}{180}
\psarc(4.67,0){0.67}{0}{180}
\psarc(8.67,0){0.67}{0}{180}
\psarc(12.67,0){0.67}{0}{180}
\psarc(3.33,0){0.67}{0}{180}
\psarc(7.33,0){0.67}{0}{180}
\psarc(11.33,0){0.67}{0}{180}
\psarc(15.33,0){0.67}{0}{180}
\psarc(0.57,0){0.57}{0}{180}
\psarc(4.57,0){0.57}{0}{180}
\psarc(8.57,0){0.57}{0}{180}
\psarc(12.57,0){0.57}{0}{180}
\psarc(3.43,0){0.57}{0}{180}
\psarc(7.43,0){0.57}{0}{180}
\psarc(11.43,0){0.57}{0}{180}
\psarc(15.43,0){0.57}{0}{180}
\psarc(0.5,0){0.5}{0}{180}
\psarc(1.5,0){0.5}{0}{180}
\psarc(2.5,0){0.5}{0}{180}
\psarc(3.5,0){0.5}{0}{180}
\psarc(4.5,0){0.5}{0}{180}
\psarc(5.5,0){0.5}{0}{180}
\psarc(6.5,0){0.5}{0}{180}
\psarc(7.5,0){0.5}{0}{180}
\psarc(8.5,0){0.5}{0}{180}
\psarc(9.5,0){0.5}{0}{180}
\psarc(10.5,0){0.5}{0}{180}
\psarc(11.5,0){0.5}{0}{180}
\psarc(12.5,0){0.5}{0}{180}
\psarc(13.5,0){0.5}{0}{180}
\psarc(14.5,0){0.5}{0}{180}
\psarc(15.5,0){0.5}{0}{180}
\end{pspicture}
\end{center}

\noindent Other applications included complex analysis, differential equations
and number theory.\footnote{John Stillwell, `Introduction', pp. 113--114.}

The models and applications of hyperbolic geometry which were so important to
its acceptance could not have developed without advances in many different
areas of mathematics in the nineteenth century. These areas included
differential geometry, group theory, geometric transformations and projective
geometry.

Gau\ss's theory of curved surfaces (1827) was essential to the eventual
recognition of hyperbolic geometry.\footnote{John Milnor. Hyperbolic Geometry:
the first 150 years. In \emph{Bulletin of the American Mathematical Society},
6:9--24, 1982, p.~10.} The language established by the theory of
curvature could be used to describe or analyse surfaces in terms of their
intrinsic properties, that is, without consideration of the space in which they
were embedded.\footnote{Torretti, \emph{Philosophy of Geometry}, p.~78.}  Geodesics
in surfaces could be defined by analogy with straight lines in the
plane.\footnote{Gray, \emph{Ideas of Space}, pp.~136--137.}  These techniques
were employed in modelling hyperbolic geometry, as all the models were surfaces
of constant negative curvature with geodesics defined intrinsically.

Riemannian geometry was also crucial to the reception of hyperbolic geometry.
At G\"ottingen in 1854, G. F. B. Riemann (1826--1866) gave a lecture entitled
``\"Uber die Hypothesen, welche der Geometrie zu Grunde liegen'' (``On the
Hypotheses which lie at  the Foundations of Geometry''). It was published in
1867. In this lecture, Riemann described a very general notion, that of an
$n$-dimensional manifold.   A manifold is characterised by a surface, and a
metric chosen from among infinitely many possible alternatives.  The theory of manifolds
included Euclidean geometry as a special case, and hyperbolic geometry as an
infinite family of cases where the surface has constant negative
curvature.\footnote{Torretti, \emph{Philosophy of Geometry}, p.~40.} By showing
that infinitely many geometries were possible, and allowing the design of
geometries very different from Euclid's, Riemann's work helped get hyperbolic geometry
admitted to mathematical respectability.\footnote{Gray, \emph{Ideas of Space},
pp.~141--145.}

Many of the applications of hyperbolic geometry used group theory. The idea of
a group began with concrete cases, such as  permutations of the roots of
algebraic equations, which were studied by J. H. Lagrange (1736--1813) and Paolo
Ruffini (1765--1822) in the eighteenth century.   In 1826, Niels Henrik Abel
(1802--1829) used such groups to prove that it was impossible to obtain a
solution in radicals of the general equation of degree greater than or equal to
5.  Augustin Louis Cauchy (1789--1857) wrote many papers on the theory of groups
of substitutions.  The term `group' was first used in an article by Evariste
Galois (1811--1832), who used it only for groups of substitutions. Cayley, in
1854, was the first to formulate the concept of an abstract group, by giving an
axiomatic definition of a group.  Klein brought group theory into geometry, by
noticing that geometric transformations can form groups under composition.

Although Euclid had rejected the use of motion in geometry, other
mathematicians employed many different kinds of geometric
transformations.\footnote{Rosenfeld, \emph{History of Non-Euclidean Geometry},
p.~112--151.} Archimedes (b.~287?~\bce) used reflections, and Apollonius
(262?--190?~\bce) used dilations and inversions. Greek and Arabian astronomers
applied stereographic projection to construct astrolabes and maps. The
geometric transformations employed by Isaac Newton (1642--1727) included
central projection and projective transformations. Alexis Claude Clairaut
(1713--1765) defined general affine transformations in words in 1733, and
Leonhard Euler (1707--1783) gave general formulae for affine and similarity
transformations in 1748, and classified plane motions. Euler also used
conformal mappings of the plane and linear fractional transformations, which
define transformations generated by similarities and inversions in circles.  A.
F. M\"obius (1790--1860) introduced the general concept of a geometric
transformation as a one-to-one correspondence between figures.

In his \emph{Erlangen Program} of 1872, Klein described geometry as the study
of those properties of figures which are preserved by a particular group of
geometric transformations. Geometries with non-trivial symmetries, such as
projective geometries and metric spaces with constant curvature, could now be
characterised by classifying groups of transformations.\footnote{John D.
Norton. Geometries in collision: Einstein, Klein and Riemann. In Gray, editor,
\emph{The Symbolic Universe}, pp.~128--144, p.~129.}  Two geometries were
equivalent if their characteristic groups were isomorphic.  For example,
Euclidean geometry in the plane is the study of those properties of figures,
such as angles and the ratios between lengths, which remain invariant under the
group of  similarities of the plane, which includes all translations, dilations
and rotations.  Klein's \emph{Erlangen Program} was at
first little noticed, but achieved international influence by the end of the
nineteenth century.\footnote{Boyer, \emph{History of Mathematics}, p.~549.}  As
a result of this program, the study of some geometries became the study of
transformation groups.\footnote{Jeremy Gray. Geometry --- formalisms and
intuitions. In Gray, editor, \emph{The Symbolic Universe}, pp.~58--83, p.~62.} Sophus Lie (1842--1899), for example,  used this idea to
characterise all three-dimensional geometries by determining all subgroups of
motions in three-dimensional space.
Hyperbolic geometry was now just one of many alternatives which could be
investigated using algebraic techniques, and was a source of applications of
group theory.

Projective geometry, with its roots in antiquity, was largely developed in the
nineteenth century. Depending on the projective metric chosen, the requirements
of spherical, Euclidean or hyperbolic geometry can be met.  So, as Klein
discovered, projective metric spaces can model hyperbolic
geometry.\footnote{Torretti, \emph{Philosophy of Geometry}, p.~110.}

\subsection{Philosophical disputes}

The philosophical controversy about non-Euclidean geometry was extensive. Since
the Greeks, Euclidean geometry had been considered the correct representation
of space, which accorded with intuition and possessed admirable logical
structure.  The discovery of non-Euclidean geometry led to debates on issues
including the  relationship between mathematics and the real world, and the
source of  geometric notions.  Changes to the ideas held in these areas were
necessary for non-Euclidean geometry to be accepted, while non-Euclidean
geometry itself inspired investigations into the foundations of mathematics.

One of the basic issues raised by non-Euclidean geometry was the connection
between geometry and physical space.  In classical physics,  Euclidean geometry
was held to be the science of space, a view which implied that this geometry
really existed.\footnote{Torretti, \emph{Philosophy of Geometry}, p.~25.}  The
ontological status of geometry was also linked to  debates over the source of
geometric knowledge.  If our geometric notions are inborn, \emph{a priori},
then geometry really exists. Further, our intuition must be obeyed, and so the
counter-intuitive nature of many non-Euclidean results is a sound argument
against them.  If, on the other hand, geometry is an abstraction of the
physical world, then Euclidean geometry is just one of many possible human
constructs.  It may well be the idealisation which is favoured by experience,
but this does not imply that it really exists.

Disputes about the source of mathematical knowledge went back to the
Greeks.\footnote{Rosenfeld, \emph{History of Non-Euclidean Geometry}, pp.~186--187.} In Plato's philosophy, mathematical knowledge was \emph{a priori}.
Geometric concepts were particular cases of the Forms, those eternal perfect
objects which are the template for physical things, and to which we have access
before we are born.  To prove these ideas, Plato describes in \emph{Meno} a
conversation between Socrates and a slave-boy.  After some leading questions,
the boy proves a result about the square of the hypotenuse of a right-angled
triangle.  In the dialogue, Socrates concludes \begin{quote} \dots now these
opinions have been newly aroused in this boy as if in a dream, but if someone
asks him these same things many times and in many ways, you can be sure that in
the end he will come to have exact knowledge of these things as well as anyone
else does \dots he will come to have knowledge without having been taught by
anyone, but only having been asked questions, having recovered this knowledge
himself, from himself.\footnote{Plato. \emph{Meno}. Aris \& Phillips
Publishers, 1985, p.~79.} \end{quote} The opposing view, put by Aristotle, was
that geometric concepts are abstractions of real world objects: a mathematician
``effects an abstraction, for in thought it is possible to separate figures
from motions''.\footnote{Quoted in Rosenfeld, \emph{History of Non-Euclidean
Geometry}, p.~186.}    Whether they regarded it as inborn knowledge or an
abstraction, the Greeks considered Euclidean geometry to correctly represent
the physical world.

The  idea of space as an empty receptacle which is then occupied by matter was
not familiar to Greek philosophers.  The classical Greek words \emph{topos}
(place) and \emph{kenon} (void) do not correspond to the word space, since
place is the place of a body, which is determined by its relationship with
other nearby bodies, while a void is always filled and is
finite.\footnote{Torretti, \emph{Philosophy of Geometry}, pp.~25--26.}  The
modern concept, though, was current in Europe by the fifteenth century.
Natural philosophers such as Giordano Bruno (1548--1600) held that space was
infinite and existed independent of matter.  Bruno wrote \begin{quote} Space
is a continuous three-dimensional natural quantity, in which the magnitude of
bodies is contained, which is prior by nature to all bodies and subsists
without them but indifferently receives them all, and is free from the
conditions of action and passion, unmixable, impenetrable, unshapeable,
non-locatable, outside all bodies yet encompassing and incomprehensibly
containing them all.\footnote{Quoted in ibid., p.~28.} \end{quote}

Once this concept of space was part of philosophical discourse, the dispute
about the source of geometry embraced the origin of our ideas of space as well as
of more specific geometrical notions.  Immanuel Kant (1724--1804) is important
for his association with the doctrine that our concept of space is \emph{a
priori}.  His major work \emph{Kritik der reinen Vernunft (Critique of Pure
Reason)} (1781) presented his theories of time and space.  Here, Kant wrote
that \begin{quote} 1. Space is not a conception which has been derived from
outward experiences.  For, in order that certain sensations may relate to
something without me \dots the representation of space must already exist as a
foundation. \dots \newline2. Space then is a necessary representation \emph{a
priori}, which serves for the foundation of all external
intuitions.\footnote{Quoted in Rosenfeld, \emph{History of Non-Euclidean
Geometry}, p.~187.}  \end{quote}  As for geometry, this was discussed by Kant in the
\emph{Kritik}, and in a chapter on pure mathematics in the \emph{Prolegomena}
(1783), a work written as a survey of the main points in the \emph{Kritik}.
Kant argued that geometry was the true science of space, that is, of space the
\emph{a priori} representation, and that geometry was revealed by pure
intuition in full agreement with Euclid's \emph{Elements}.\footnote{Torretti,
\emph{Philosophy of Geometry}, p.~33.}   Because of these doctrines, followers
of Kantian orthodoxy in the nineteenth century opposed non-Euclidean geometry,
dismissing it as an intellectual exercise unrelated to the real
world.\footnote{Ibid.}

Much of classical physics depended on Euclidean geometry.\footnote{Kline,
\emph{Mathematical Thought}, pp.~880--881.} Johann Kepler (1571--1630) regarded
geometry as the foundation of both celestial and terrestrial physics, writing
\begin{quote} We see that the motions [of the planets] occur in time and place
and that the force [that binds them to the sun] emanates from its source and
diffuses through the spaces of the world.  All these are geometrical things.
Must not that force be subject also to other geometrical
things?\footnote{Quoted in Torretti, \emph{Philosophy of Geometry}, p.~23.}
\end{quote} The coordinate geometry of Ren\'e Descartes (1596--1650) was crucial
to the status of  geometry in physics.\footnote{Torretti, \emph{Philosophy of
Geometry}, p.~34.}  With the Cartesian coordinate system, space could  be
equated with a structured set of points and, importantly, motion could be
analysed.

Gottfried Wilhelm Leibniz (1646--1716) and Isaac Newton are considered the
founders of mathematical physics.  They had differing views of space. Leibniz
(like Descartes) considered it impossible for space to be empty.  In
\emph{Nouveaux essais sur l'entendement humain (New Essays on Human
Understanding)}, Leibniz wrote that ``it is necessary rather to conceive space
as full of a matter originally fluid, susceptible of all the
divisions''.\footnote{Quoted in Rosenfeld, \emph{History of Non-Euclidean
Geometry}, p.~184.} Newton, however, thought that space was independent from
matter.  In the \emph{Principia} (1687), he defined absolute space as:
\begin{quote} \emph{Absolute space}, in its own nature, without relation to
anything external, remains always similar and immovable.\footnote{Quoted in
ibid., p.~185.} \end{quote} Absolute space was essential to  the results of the
\emph{Principia}, and it was Newton's ideas on the nature of space which
prevailed as theoretical mechanics was developed  in the eighteenth century.
The implicit assumptions about space made by Newton and later mechanists
included that it was continuous, infinite, three-dimensional and homogeneous,
and that it satisfied Euclidean theorems.\footnote{Torretti, \emph{Philosophy
of Geometry}, p.~25.}

Because physics depended on Euclidean geometry, undermining the parallel axiom
was a real physical problem.\footnote{Kline, \emph{Mathematical Thought},
pp.~880--881.}  There were sound reasons to think that physical space was
Euclidean.  How then could an alternative geometry exist?  The word `existence'
here was taken to mean true physical existence, and this is certainly the
interpretation of `existence' held by Bolyai and Lobachevsky.  Apart from the
parallel axiom, they never questioned Euclidean geometry, which Bolyai
described as ``the absolutely true science of space''.\footnote{Quoted in
Torretti, \emph{Philosophy of Geometry}, p.~62.}  Concrete notions of existence
definitely hampered the acceptance of non-Euclidean geometry: contemporaries
could discount Lobachevsky and Bolyai's work because neither author succeeded
in establishing the physical existence of their geometry.\footnote{Gray,
\emph{Ideas of Space}, p.~170.}

The development of Riemannian geometry was very important in changing
mathematicians' ideas about the nature of mathematical existence. Riemann began
his 1854 lecture on the foundations of geometry by observing that all previous
presentations of geometry had presupposed the nature of space.  He then
described space, meaning real physical space, as just one instance of a general
concept which he called a \emph{mehrfach ausgedehnte Gr\"osse} (multiply
extended quantity).\footnote{Torretti, \emph{Philosophy of Geometry}, p.~83.}
Riemann showed that such an extended quantity (in modern terminology, a
manifold) may have arbitrarily many ``metric relations''; therefore ``space
constitutes only a special case of a threefold extended
quantity''.\footnote{Quoted in ibid., p.~83.}   The implication of Riemann's geometry
for the philosophy of space was that, since there can be only one geometry of
real space among infinitely many mathematical possibilities, the science of space could not
be determined by mathematical processes alone. As Riemann said, ``those
properties which distinguish space from other conceivable threefold extended
quantities can be gathered only from experience''.\footnote{Quoted in ibid., p.~84.}
Riemann acknowledged that Euclidean geometry was empirically successful within
``the limits of observation'',\footnote{Quoted in ibid., p.~104.} but his work
meant that mathematical existence could no longer be equated with physical
existence.  This was a very significant change in mathematical discourse.

Because each possible manifold had equal mathematical validity, Euclidean and
non-Euclidean geometries could now exist side by side.  Indeed, Beltrami's
model showed that the consistency of Euclidean geometry implied the consistency
of hyperbolic geometry.  The models of hyperbolic geometry, such as Beltrami's,
also represented a move towards semantic neutrality.  With words such as
`straight' being interpreted in ways differing from their everyday meaning,
physical existence became irrelevant.  Poincar\'e, in the spirit of the
\emph{Erlangen Program}, characterised geometries using groups, and then
explained that alternative geometries can exist because the existence of one
group is not incompatible with the existence of another.\footnote{Arthur Miller.
Einstein, Poincar\'e, and the testability of geometry. In Gray, editor,
\emph{The Symbolic Universe}, pp.~47--57, p.~54.}

The development of projective geometry further aided the acceptance of
hyperbolic geometry.  Since projective geometry was so counter-intuitive, it
was quickly characterised by axioms, a project which was completed in 1882 by
Moritz Pasch (1843--1930).\footnote{Torretti, \emph{Philosophy of Geometry},
p.~190.}   Objections to non-Euclidean geometry based on intuition had little
weight once mathematics accepted systems as strange as projective geometry.

By the late nineteenth century, mathematicians made no claims that their work
represented the physical truth.  For example, in his popularising work of 1912,
\emph{La science et l'hypoth\`{e}se}, Poincar\'e wrote that the question of
whether Euclidean geometry was true had no meaning. He dismissed the argument
that geometric ideas are \emph{a priori} by pointing out that alternative
systems, such as projective geometry, are conceivable.\footnote{Henri
Poincar\'e. \emph{Science and Hypothesis}. Dover Publications Inc., 1952,
p.~48.}  Also, he argued, geometry could not be a body of experimental truths,
since if it were experimental it would not be exact, and would be subject to
continuous revision.\footnote{Ibid., pp.~49--50.}  Thus the choice of which
geometry to use in  representing physical space was a matter of convention,
with experience guiding but not determining our choice of the most convenient
system.\footnote{Ibid., p.~50.}  Poincar\'e concluded: ``One geometry cannot be
more true than another; it can only be more convenient''.\footnote{Ibid.}

This attenuation of mathematical truth-claims meant that the mathematical
notion of existence was very much weaker than that of physics. In fact, over
the course of the nineteenth century, it became for the first time possible,
then necessary, to distinguish between the subjects of mathematics and physics,
as they  developed quite distinct institutions and aims.\footnote{Jeremy Gray.
Introduction. In Gray, editor, \emph{The Symbolic Universe}, pp.~1--21, p.~3.}  The
leading physicist Ludwig Boltzmann (1844--1906), for instance, considered
mathematics to be unrelated to the real world, and to have only a service role
in physics.\footnote{Ibid., p.~4.}  The writings of Ernst Mach (1838--1916) were
very influential in physics.  His notion of existence was stronger than that
prevailing in mathematics, and he wrote of nineteenth-century geometry:
\begin{quote}  Analogues of the geometry we are familiar with, are constructed
on broader and more general assumptions for any number of dimensions, with no
pretension to being regarded as more than intellectual scientific experiments
and with no idea of being applied to reality. \dots Seldom have thinkers become so
steeped in reverie, or so far estranged from reality, as to imagine for our
space a number of dimensions exceeding the three of the given space of sense,
or to conceive of representing that space by any geometry that appreciably
departs from the Euclidean.  Gauss, Lobachevsky, Bolyai and Riemann were
perfectly clear on this point, and certainly cannot be held responsible for the
grotesque fictions subsequently stated in this field.\footnote{Quoted in Rosenfeld,
\emph{History of Non-Euclidean Geometry}, p.~203.} \end{quote} It is interesting,
then, that in twentieth-century physics, space is Riemannian and not
Euclidean.\footnote{Gray, `Geometry', p.~80.} Einstein's theories of special
and general relativity depend on geometry quite different from the
three-dimensional Euclidean geometry of classical physics.\footnote{Miller,
`Einstein, Poincar\'e', p.~55.}  Indeed, from 1908 to 1916, the laws of
special relativity were reformulated and reinterpreted in terms of hyperbolic
geometry.\footnote{Walter, `Non-Euclidean style', p.~91.}  Even so,
non-relativistic physics
still employs only one geometry to represent physical space, while
mathematicians consider a whole variety of geometries.

Non-Euclidean geometry stimulated discussion of the foundations of
mathematics.  Works responding to the new geometries began to appear in the
1860s.  Rigorous axiomatisation was their aim, and they embodied
the new, weaker notion of existence, and a lesser role for intuition.
Lie, for example, published two papers on the foundations of
geometry in 1890. In these, he asked what properties were necessary and
sufficient to define Euclidean geometry and the non-Euclidean geometries of
constant non-zero curvature.  Thus Lie considered the problem of space as an
issue of pure mathematics only, showing the divide between mathematics and
ideas about physical space.\footnote{Torretti, \emph{Philosophy of Geometry},
p.~154.}

The investigations into foundations culminated in 1899 with the publication of
\emph{Grundlagen der Geometrie} by David Hilbert (1862--1943). This was a
complete axiomatisation of geometry, which rectified the logical inadequacies
of Euclid's \emph{Elements}.  For Hilbert, a geometry was defined by a set of
unproved axioms, any of which might be negated to obtain a different theory.
The axioms contained undefined primitive terms, which could be modelled as
referring to particular objects such as points, lines and planes, restricted
only by the implications of the axioms.  Such a level of formalism further
weakened claims as to the truth of geometry.\footnote{Gray, `Geometry',
p.~69.}  At the end of the nineteenth century, Euclidean and hyperbolic
geometry had become two abstract axiomatic systems, differing in only one
axiom, and equally consistent.  This philosophy was very different from that of
Bolyai and Lobachevsky, but their work contributed to it.\footnote{Torretti,
\emph{Philosophy of Geometry}, p.~61.}

\section{Conclusion}

This history shows how the development and reception of non-Euclidean geometry
were associated with fundamental changes in the way mathematics is done and
thought about.  The recognition and study of abstract mathematical spaces, and
the realisation of the connection between geometry and groups, are stages in
this history which are of particular relevance for subsequent chapters.
Indeed, the style of exposition adopted for the rest of this thesis, with
explicit definitions, acknowledgement of assumptions and so forth, illustrates
well the abstraction and logical precision of modern mathematics.  The shift in
conceptions of geometry which resulted in this modern approach was stimulated
by the discovery of non-Euclidean geometry, and also aided the acceptance of
the new geometry by mathematicians.


\chapter{Hyperbolic Geometry}\label{hypgeom}


\section{Introduction}

In this chapter we survey hyperbolic geometry.  We provide a modern treatment
of much of the geometry discussed in Chapter 1, and lay the foundations for
Chapters 3 and 4.  After a section of preliminary definitions and results, we
consider four models of hyperbolic geometry.  Hyperbolic geometry is actually a
continuum of geometries, each distinguished by its Gaussian curvature.  For
convenience, the models of hyperbolic geometry presented here all have constant
curvature of $-1$. The first model discussed is called the hyperboloid model,
because the hyperbolic plane in this model is one sheet of a two-sheeted
hyperboloid. In general, the hyperboloid model is the sphere of radius $i$ in a
non-Euclidean space called Lorentzian $n$-space.  Lorentzian 4-space is the
model of space-time in special relativity.  Next is a brief treatment of
Beltrami's model, which was reinterpreted by Klein, and is also known as the
projective disc model. The model introduced by Poincar\'e was the upper
half-plane, and he showed how this model could be mapped to the unit disc.
Poincar\'e's models have now been generalised to higher dimensions, and are
known respectively as the upper half-space model, and the conformal ball
model.  We will prove that, in the upper half-space and the conformal ball,
hyperbolic isometries may be identified with compositions of finitely many
reflections in planes and spheres. In the upper half-plane, we go on to
consider how tesselation by triangles is related to the action of the discrete
group \sltwoz.  Most of the material in this chapter is selected and adapted
from  the treatment of hyperbolic geometry in Chapters 3--6
of~\cite{rat1:fhm}.  The proofs of relevant exercises in~\cite{rat1:fhm} are
provided.  Some additional results which will be needed in later chapters are
established.

\section{Geometric preliminaries}

This section comprises useful definitions and results relating to Euclidean
$n$-space and to general metric spaces.  The geometric transformations
discussed include reflections in planes and spheres.  Compositions of finitely
many such reflections are called M\"obius transformations, and M\"obius
transformations turn out to be very important for understanding hyperbolic
geometry.

\subsection{Euclidean $n$-space}

We define Euclidean $n$-space, $E^n$, to be the metric space consisting of
$\R^n$ together with the distance function \[ d_E(\vecx,\vecy) = |\vecx -
\vecy|. \] Here, the right-hand side is the Euclidean norm, $|\vecx| = (\vecx
\cdot \vecx)^{1/2}$.  The Euclidean inner product is the usual dot product
given by \[ \vecx \cdot \vecy = x_1 y_1 + x_2 y_2 + \cdots + x_n y_n. \] We are
referring to $d_E$ as a distance function rather than as a metric in order to
avoid confusion with Riemannian metrics in Chapter 3. The Cauchy--Schwartz
inequality for $\vecx, \vecy \in E^n$, \[ |\vecx\cdot\vecy| \leq
|\vecx||\vecy| \] with equality if and only if \vecx\ and \vecy\ are linearly
dependent, allows us to define $\theta(\vecx,\vecy)$, the Euclidean angle
between non-zero vectors \vecx\ and \vecy, by \[ \vecx \cdot \vecy =
|\vecx||\vecy|\cos\theta(\vecx,\vecy). \]

\subsection{Geometric transformations and group actions}

An isometry from a metric space $(X,d_X)$ to a metric space $(Y,d_Y)$ is a
bijection $\phi: X \rightarrow Y$ such that \[ d_X(x_1,x_2) =
d_Y(\phi(x_1),\phi(x_2))  \] for all $x_1,x_2 \in X$.  The set of isometries
from a metric space $X$ to itself, denoted by $I(X)$, forms a group under
composition. Isometries of $E^n$ are known as Euclidean isometries.

A function $\phi:\mathbb{R}^n \rightarrow \mathbb{R}^n$ is called an orthogonal
transformation if \[ \phi(\vecx)\cdot\phi(\vecy) = \vecx\cdot\vecy \] for all
$\vecx,\vecy \in \mathbb{R}^n$.  Let $\{\mathbf{e}_1,\vece_2,\ldots,\mathbf{e}_n\}$ be
the standard basis of $\mathbb{R}^n$.  Then, it can be proved that a map $\phi:
\mathbb{R}^n \rightarrow \mathbb{R}^n$ is an orthogonal transformation if and
only if $\phi$ is linear and $\{\phi(\mathbf{e}_1),\phi(\vece_2), \ldots,\phi(\mathbf{e}_n)\}$
is an orthonormal basis of $\mathbb{R}^n$.  It follows that every orthogonal
transformation is a Euclidean isometry, and that orthogonal transformations may
be identified with orthogonal $n \times n$ matrices.  The set of all orthogonal
$n \times n$ matrices, denoted by $O(n)$, forms a group under matrix
multiplication.

A group $G$ is said to act on a set $X$ if there exists a function from
$G\times X$ to $X$, written $(g,x) \mapsto gx$, such that for all $g,h \in G$,
and all $x \in X$, we have \[ 1x = x \hspace{1cm} \mbox{and} \hspace{1cm} g(hx)
= (gh)x. \] For example, the group $I(X)$ acts on $X$, and the group $O(n)$
acts on the set of $m$-dimensional vector subspaces of $\mathbb{R}^n$, for $1
\leq m \leq n$.  The action of $G$ on $X$ is said to be transitive if for each
$x, y \in X$ there is a $g \in G$ such that $gx = y$.

\begin{propn} \label{orthogtransitive} For each dimension $m$, the action of
the group $O(n)$ on the set of $m$-dimensional vector subspaces of $\mathbb{R}^n$ is
transitive. \end{propn}

\begin{proof} Let $V$ be any $m$-dimensional subspace of $\mathbb{R}^n$.
Identify $\R^m$ with the span of the set $\{\vece_1, \vece_2, \ldots, \vece_m \}$ in
$\R^n$. Since $O(n)$ is a group, it suffices to show there is an $A \in
O(n)$ such that $A(\mathbb{R}^m) = V$.

Choose an orthonormal basis $\{ \vecw_1, \vecw_2, \ldots, \vecw_m \}$ for $V$. Then,
using the Gram--Schmidt process, extend it to an orthonormal basis $\{ \vecw_1,
\vecw_2,
\ldots, \vecw_n \}$ for $\mathbb{R}^n$. Let $A$ be the $n \times n$ matrix
which has $\vecw_1,\vecw_2, \ldots,\vecw_n$ as its columns.  Then $A$ is orthogonal, and
$A(\mathbb{R}^m) = V$. \end{proof}

An $m$-plane of $E^n$ is defined to be a coset $\veca + V$, where $\veca \in
E^n$ and $V$ is an $m$-dimensional subspace of $E^n$.

\begin{corollary} \label{transplanes}For each dimension $m$, the group $I(E^n)$
acts transitively on the set of $m$-planes of $E^n$. \end{corollary}

\begin{proof} Let $\veca + V$ and $\vecb + W$ be $m$-planes of $E^n$.  By
Proposition~\ref{orthogtransitive}, there is an $A \in O(n)$ such that $A(V) =
W$.  Define the map $\phi: E^n \rightarrow E^n$ by \[ \phi(\vecx) = (\vecb -
A\veca) + A\vecx. \] Then $\phi(\veca + V) = \vecb + W$, and \begin{align*}
|\phi(\vecx) - \phi(\vecy)|  =  |A\vecx - A\vecy| =  |\vecx - \vecy|,
\end{align*} since $A$ is an isometry.  Hence $\phi$ is an isometry.
\end{proof}

The following proposition shows that every Euclidean isometry is the
composition of an orthogonal transformation and a translation.

\begin{propn} \label{formisom} The function $\phi: E^n \rightarrow E^n$ is an
isometry if and only if $\phi$ is of the form
\[
\phi(\vecx) = \veca + A\vecx,
\]
where $A$ is an orthogonal matrix and $\veca = \phi(\veczero)$.
\end{propn}

\begin{proof} Suppose $\phi$ is an isometry.  Define a map $A$ by $A(\vecx) =
\phi(\vecx) - \phi(\veczero)$.  Then $A(\veczero) = \veczero$, and, since $\phi$
preserves Euclidean norms,  \[ |A(\vecx) - A(\vecy)| = |\phi(\vecx) -
\phi(\vecy)| = |\vecx - \vecy| \] for all $\vecx, \vecy \in E^n$.  Therefore,
if $\vecx \in E^n$, \[ |A(\vecx)| = |A(\vecx) - A(\veczero)| = |\vecx -
\veczero| = |\vecx|, \] so $A$ preserves Euclidean norms.  It follows that $A$
is orthogonal, since \begin{align*} 2 (A\vecx \cdot A\vecy) & =  |A\vecx|^2 +
|A\vecy|^2 - |A\vecx - A\vecy|^2 \\ & =  |\vecx|^2 + |\vecy|^2 - |\vecx -
\vecy|^2 \\ & =  2(\vecx\cdot \vecy). \end{align*} So we have an orthogonal
matrix $A$ such that $\phi(\vecx) = \phi(\veczero) + A\vecx$.

Conversely, suppose $\phi$ is of the form $\phi(\vecx) = \phi(\veczero) + A\vecx$.  Then
$\phi$ is the composition of an orthogonal transformation and a translation,
both of which are isometries. Therefore $\phi$ is an isometry. \end{proof}

A function $\phi:X \rightarrow Y$ between metric spaces $(X,d_X)$ and $(Y,d_Y)$ is called a
similarity if it is a bijection, and there is a scale factor $k>0$ such that \[
d_X(x_1, x_2) = k d_Y(\phi(x_1),\phi(x_2))  \] for all $x_1,x_2 \in X$.  The
set of similarities from a metric space $X$ to itself, denoted by $S(X)$, forms
a group under composition. The isometries of $X$ are a subgroup of $S(X)$.
Under transformation by elements of $S(E^n)$, all the theorems of Euclidean
geometry remain true. Thus the similarities of $E^n$ are its characteristic
group, in the sense used by Klein in his \emph{Erlangen Program}.  The
following characterisation of similarities follows from
Proposition~\ref{formisom}.

\begin{propn} \label{formsim}  The function $\phi: E^n \rightarrow E^n$ is a
similarity if and only if $\phi$ is of the form \[ \phi(\vecx) = \veca +
kA\vecx, \] where $A$ is an orthogonal matrix, $k$ is a positive constant and
$\veca = \phi(\veczero)$. \end{propn}

\subsection{Geodesics}

Let $(X,d)$ be a metric space, and let $\alpha:[a,b]\rightarrow X$ be a
continuous injection, where $a < b$ in \R.  We say $\alpha$ is a geodesic curve
if, for all $s,t \in [a,b]$, \[ d(\alpha(s),\alpha(t)) = |s-t|. \] For points
$x$ and $y$ in $X$, we denote by $[x,y]$ the image in $X$ of a geodesic curve
$\alpha:[a,b]\rightarrow X$ such that $\alpha(a) = x $ and $\alpha(b) = y$.  We
call $[x,y]$ a geodesic segment. Then $[x,y]\cup[y,z]$ is a geodesic segment
joining $x$ to $z$ if and only if \[ d(x,y) + d(y,z) = d(x,z). \] A geodesic
line is defined to be a function $\lambda:\R \rightarrow X$ which is locally
distance-preserving.  This means that for any $x \in \lambda(\R)$, with $x =
\lambda(t)$, there is a set $[a,b] \subseteq \R$ containing $t$  such that
$\lambda$ restricted to $[a,b]$ is a geodesic curve. A geodesic is the image in
$X$ of a geodesic line $\lambda: \R\rightarrow X$.

It turns out that, if the metric space $(X,d)$ satisfies the axioms of
Euclidean or hyperbolic geometry, then the shortest path from $x$ to $y$ in $X$ is
along the unique geodesic segment joining $x$ and $y$, and the length of this
shortest path is $d(x,y)$.

\subsection{Reflections} \label{reflections}

Let \vecu\ be a unit vector in $E^n$, and let $t$ be a real number. The hyperplane
of $E^n$ with unit normal vector \vecu\ passing through the point $t\vecu$ is
the set \[ P(\vecu,t) = \{ \vecx \in E^n : \vecu \cdot \vecx = t \}. \] We will
sometimes refer to hyperplanes as just planes.  The reflection $\rho$ of $E^n$ in
the plane $P(\vecu, t)$ is given by \[ \rho(\vecx) = \vecx + 2(t -
\vecu\cdot\vecx)\vecu. \] Then $\rho(\vecx) = \vecx$ if and only if $\vecx \in
P(\vecu,t)$, and $\rho^2(\vecx) = \vecx$ for all $\vecx \in E^n$. We now show
that every Euclidean isometry is a composition of finitely many reflections in
hyperplanes.

\begin{lemma} \label{rhoorthog} Let $\vecx$, $\vecy$ be in $E^n$ and $\rho$ be
the reflection in the plane $P(\vecu,t)$.  Then \[ |\rho(\vecx) - \rho(\vecy)|
= |\vecx - \vecy|. \] That is, $\rho$ is a Euclidean isometry. \end{lemma}

\begin{propn} \label{isomrefl}Every isometry of $E^n$ is a composition of at
most $n+1$ reflections in hyperplanes. \end{propn}

\begin{proof} Let $\phi:E^n \rightarrow E^n$ be an isometry.  We construct
reflections $\rho_0,\rho_1,\ldots,\rho_n$ such that $\phi = \rho_0\rho_1 \cdots \rho_n$.

Set $\vecv_0 = \phi(\veczero)$.  If $\vecv_0 = \veczero$ let $\rho_0$ be the
identity, and otherwise let $\rho_0$ be the reflection in the plane
$P\left(\vecv_0/|\vecv_0|,|\vecv_0|/2\right)$. Then \[
\rho_0(\vecv_0) = \vecv_0 + 2\left(\frac{|\vecv_0|}{2} -
\frac{\vecv_0\cdot\vecv_0}{|\vecv_0|} \right)\frac{\vecv_0}{|\vecv_0|} =
\veczero, \] and so $\rho_0\phi(\veczero) = \veczero$.  Let $\phi_0 = \rho_0
\phi$.  Then the composition $\phi_0$ is an isometry, and $\phi_0(\vecx) = \rho_0\phi(\veczero)
+ \rho_0\phi(\vecx)$, so by Proposition~\ref{formisom}, $\phi_0$
is orthogonal.

For $k = 1,2,\ldots,n$, assume inductively that $\phi_{k-1}$ is an orthogonal
transformation of $E^n$ which fixes $\veczero$ and the basis vectors
$\mathbf{e}_1,\vece_2,\ldots,\mathbf{e}_{k-1}$.  Let $\vecv_k$ be the vector
$\phi_{k-1}(\mathbf{e}_k) - \mathbf{e}_k$.  If $\vecv_k = \veczero$ let
$\rho_k$ be the identity, and otherwise let $\rho_k$ be the reflection in the
plane $P\left(\vecv_k/|\vecv_k|,0\right)$. Then, by  using the assumption that
$\phi_{k-1}$ is orthogonal to expand out $\rho_k\phi_{k-1}(\vece_k)$, we find
that $\rho_k\phi_{k-1}$ fixes $\vece_k$. Also, for $1 \leq j  \leq k-1$,
\begin{align*} \vecv_k \cdot \vece_j & =  (\phi_{k-1}(\vece_k) -
\vece_k)\cdot\vece_j \\ & =  \phi_{k-1}(\vece_k)\cdot\vece_j \\ & =
\phi_{k-1}(\vece_k)\cdot\phi_{k-1}(\vece_j) \\ & =  \vece_k \cdot\vece_j \\ &
=  0. \end{align*} Therefore, for $1 \leq j \leq k-1$, the vector $\vece_j$ is
in the plane $P\left(\vecv_k/|\vecv_k|,0\right)$, hence is fixed by the
reflection $\rho_k$.  Consequently, $\phi_k = \rho_k\phi_{k-1}$ fixes  the
basis vectors $\vece_1,\vece_2,\ldots,\vece_k$. Also, since $\veczero$ is in
the plane $P\left(\vecv_k/|\vecv_k|,0\right)$, $\rho_k$ fixes $\veczero$. By
induction, then, the maps $\rho_0,\rho_1,\ldots,\rho_n$ are each either the
identity or a reflection, and the composition $\rho_n\cdots\rho_1\rho_0\phi$
fixes $\veczero,\vece_1,\vece_2,\ldots,\vece_n$. Using
Proposition~\ref{formisom} again, we infer that $\rho_n\cdots\rho_1\rho_0\phi$
is an orthogonal, hence linear, map which fixes all the basis vectors, and so
is the identity.  Since each reflection is its own inverse, this means that
$\rho_0\rho_1\cdots\rho_n = \phi$. \end{proof}

We now discuss reflections in spheres.  These maps are also known as
inversions. For $\mathbf{a} \in E^n$ and $r> 0$, we denote by $S(\veca,r)$ the
sphere with centre $\mathbf{a}$ and radius $r$. The reflection $\sigma$ of the
set  $E^n\setminus\{ \veca\}$ in the sphere $S(\veca,r)$ is given by \[
\sigma(\vecx) = \veca + r^2\frac{\vecx - \veca}{|\vecx - \veca|^2}. \] Then
$\sigma(\vecx) = \vecx$ if and only if $\vecx \in S(\veca,r)$, and
$\sigma^2(\vecx) = \vecx$ for all $\vecx \not = \veca$. Abusing notation, we
will say that $\sigma$ is the reflection of ${E}^n$ in $S(\veca,r)$.   Although
$\sigma$ is not an isometry, we do have the following identity.

\begin{lemma}\label{phixphiy} Let $\sigma$ be the reflection of $E^n$ in the
sphere $S(\veca,r)$. For all $\vecx, \vecy \not = \veca$, \[ |\sigma(\vecx)
- \sigma(\vecy)| = \frac{r^2|\vecx - \vecy|}{|\vecx - \veca||\vecy - \veca|}.
\] \end{lemma}

We will sometimes denote by $S^{n-1}$ the sphere $S(\veczero,1)$ in $E^n$.

Reflections in planes and spheres are continuous maps with respect to the usual
metric topology of $\R^n$, which is that induced by the distance function
$d_E$.  They are also differentiable functions.  The following result shows
that reflections in hyperplanes and spheres preserve angles. Let $\Omega$ be an
open subset of $E^n$ and let $\phi:\Omega \rightarrow E^n$ be a differentiable
function.  We say $\phi$ is conformal if $\phi$ preserves angles between
differentiable curves in $\Omega$. An important class of conformal maps is the
set of orthogonal matrices.  These are conformal because they preserve the
Euclidean inner product.

\begin{propn} \label{conformal} Every reflection of $E^n$ in a hyperplane or
sphere is conformal. \end{propn}

\begin{proof} Let $\rho$ be the reflection of $E^n$ in the plane
$P(\vecu,t)$. As $\rho(\vecx) = \vecx + 2(t - \vecu\cdot\vecx)\vecu$, the
$j^{\rm th}$ component of $\rho(\vecx)$ is $x_j + 2(t -
(u_1x_1+u_2x_2 + \cdots+u_nx_n))u_j$.  Thus \[ D\rho(\vecx) = (\delta_{ij} - 2u_iu_j) =
I - 2U, \] where $D\rho(\vecx)$  is the matrix of partial derivatives of
$\rho(\vecx)$, and $U$ is the matrix $(u_iu_j)$. As $D\rho(\vecx)$ is
independent of $t$, we may assume without loss of generality that $t=0$.  Then
$\rho(\vecx) = (I-2U)\vecx$, and \begin{align*} \rho(\vecx)\cdot\rho(\vecy)& =
(\vecx-2(\vecu\cdot\vecx)\vecu)\cdot(\vecy-2(\vecu\cdot\vecy)\vecu)\\ & =
\vecx\cdot\vecy - 4(\vecu\cdot\vecx)(\vecu\cdot\vecy) +
4(\vecu\cdot\vecx)(\vecu\cdot\vecy)(\vecu\cdot\vecu) \\ & =  \vecx\cdot\vecy,
\end{align*} so $\rho$ is an orthogonal transformation. Therefore $I-2U$ is an
orthogonal matrix, and so the reflection $\rho$ is conformal.

For reflections in spheres, we consider first the reflection $\sigma_r$ in the
sphere $S(\veczero,r)$.  We have $\sigma_r(\vecx) =
r^2\vecx/|\vecx|^2$, so the $j^{\rm th}$ component of
$\sigma_r(\vecx)$ is $r^2x_j/(x_1^2 + x_2^2 + \cdots + x_n^2)$. Hence, \[
D\sigma_r(\vecx) = r^2\left(\frac{\delta_{ij}|\vecx|^2 -
2x_ix_j}{|\vecx|^4}\right) = \frac{r^2}{|\vecx|^2}(I-2U), \] where $U$ is the
matrix $(x_ix_j/|\vecx|^2)$.  From above, $I - 2U$ is orthogonal, as
$\vecx/|\vecx|$ is a unit vector. Thus $\sigma_r$ is conformal.

Now let $\sigma$ be the reflection in the sphere $S(\veca,r)$ and let $\tau$ be
translation by $\veca$.  Then $\sigma = \tau\sigma_r\tau^{-1}$.  Since $\tau$
and $\sigma_r$ are conformal, this means $\sigma$ is conformal.   \end{proof}

\subsection{M\"obius transformations}

We now consider the one-point compactification of $E^n$, $E^n \cup \{ \infty
\}$, which we denote by $\hat{E}^n$. The reflections $\rho$ and $\sigma$ above
may be extended to be maps from $\hat{E}^n$ to $\hat{E}^n$, by setting
$\rho(\infty) = \infty$, $\sigma(\veca) = \infty$ and $\sigma(\infty) = \veca$.
A sphere of $\hat{E}^n$ is defined to be either a Euclidean sphere
$S(\veca,r)$, or an extended plane $P(\vecu,t) \cup \{ \infty \}$. For
simplicity of notation, we will also use $P(\vecu,t)$ to denote extended
planes. In the case $n=2$, we may identify $\hat{E}^2$ with the Riemann sphere
$\hat{\mathbb{C}}$.

A M\"obius transformation of $\hat{E}^n$ is a finite composition of reflections
in spheres of $\hat{E}^n$.  The set of all M\"obius transformations of
$\hat{E}^n$, denoted by $M(\hat{E}^n)$, forms a group under composition.  By
Proposition~\ref{isomrefl}, any Euclidean isometry may be written as a finite
composition of reflections, so we may by extending these reflections regard the
group of Euclidean isometries to be a subgroup of $M(\hat{E}^n)$. {}From the
properties of reflections, M\"obius transformations are continuous,
differentiable and conformal.

We now discuss the action of $M(\hat{E}^n)$ on spheres.  A M\"obius
transformation of particular importance is the reflection of $\hat{E}^n$ in the
sphere $S(\vece_n,\sqrt{2})$. Identify $\hat{E}^{n-1}$ with $\hat{E}^{n-1}
\times \{ 0 \}$ in $\hat{E}^n$.  The stereographic projection $\pi$ of
$\hat{E}^{n-1}$ onto the sphere $S^{n-1}$ is defined to be the projection of
$\vecx \in \hat{E}^{n-1}$ towards (or away from) the vector $\vece_n$ until it
meets the sphere $S^{n-1}$.  Let $\sigma$ be the reflection of $\hat{E}^n$ in
the sphere $S(\vece_n,\sqrt{2})$.  Then, by comparing the explicit formulae for
$\pi$ and $\sigma$, it can be shown that the restriction of $\sigma$ to
$\hat{E}^{n-1}$ is stereographic projection $\pi:\hat{E}^{n-1} \rightarrow
S^{n-1}$.

\begin{lemma} \label{mobactionlemma} Let $\sigma$ be the reflection of
$\hat{E}^n$ in the sphere $S(\veca,r)$, and let $\sigma_1$ be the reflection of
$\hat{E}^n$ in the sphere $S(\veczero,1)$. Define a map $\tau:\hat{E}^n
\rightarrow \hat{E}^n$ by $\tau(\vecx) = \veca + r\vecx$.  Then $\sigma =
\tau\sigma_1\tau^{-1}$. \end{lemma}

\begin{propn} The group $M(\hat{E}^n)$ acts transitively on the set of spheres
of $\hat{E}^n$. \end{propn}

\begin{proof} We first need to show that $M(\hat{E}^n)$ maps spheres to
spheres. Let $\phi$ be a M\"obius transformation and let $\Sigma$ be a sphere
of $\hat{E}^n$.  As $\phi$ is a composition of reflections, we may assume that
$\phi$ is a reflection.  Suppose first that $\phi$ is a reflection in an
extended plane.  Since $\phi$ is an isometry, by Corollary~\ref{transplanes},
$\phi$ maps hyperplanes to hyperplanes, and by the definition of a sphere
$\phi$ maps spheres to spheres.  Now suppose that $\phi$ is the reflection in
the Euclidean sphere $S(\veca,r)$, and let $\tau$ be the same map as in
Lemma~\ref{mobactionlemma}.  By Proposition~\ref{formsim}, $\tau$ is a
Euclidean similarity, and so $\tau$ maps spheres to spheres. We may thus, by
Lemma~\ref{mobactionlemma}, assume without loss of generality that $\phi$ is
the reflection in the sphere $S(\veczero,1)$, that is, $\phi(\vecx) =
\vecx/|\vecx|^2$.

We may characterise the spheres of $\hat{E}^n$ by a vector $(a_0,a_1,\ldots,a_{n+1})$
in $\R^{n+2}$, called a coefficient vector, which is unique up to multiplication by a non-zero scalar.  The
equations for the sphere $S(\veca,r)$ and the extended plane $P(\veca,t)$ in
$\hat{E}^n$ are respectively \[ |\vecx|^2 - 2\veca\cdot\vecx + |\veca|^2 - r^2
= 0 \hspace{5mm}\mbox{and}\hspace{5mm} -2\veca\cdot\vecx + 2t = 0. \] These may
be written in the common form \begin{equation}\label{coeffsphere} a_0|\vecx|^2
- 2\veca\cdot\vecx + a_{n+1} = 0, \end{equation} with $|\veca|^2 >a_0a_{n+1}$.
Conversely, any vector $(a_0,a_1,\ldots,a_{n+1})\in \R^{n+2}$ such that $|\veca|^2>
a_0a_{n+1}$ determines a sphere of $\hat{E}^n$ satisfying
equation~(\ref{coeffsphere}).

Now let $(a_0,a_1,\ldots,a_{n+1})$ be a coefficient vector for $\Sigma$. Then
$\vecx \in \Sigma$ satisfies equation~(\ref{coeffsphere}), and so if $\vecy =
\phi(\vecx)$, by the formula for $\phi$, the vector $\vecy$ satisfies \[ a_0 -
2\veca\cdot\vecy + a_{n+1}|\vecy|^2 = 0. \] This is the equation of another
sphere, say $\Sigma'$.  Hence $\phi$ maps $\Sigma$ into $\Sigma'$.  Since
$\phi$ is its own inverse, the same argument shows that $\phi$ maps $\Sigma'$
into $\Sigma$.  Thus, since $\phi$ is a bijection, $\phi(\Sigma) = \Sigma'$. We
have shown that $M(\hat{E}^n)$ acts on the set of spheres of $\hat{E}^n$.

To show that this action is transitive, let $\Sigma$ be a sphere of
$\hat{E}^n$. It suffices to prove there is a M\"obius transformation $\phi$
such that $\phi(\Sigma) = \hat{E}^{n-1}$. By Corollary~\ref{transplanes}, the
group of Euclidean isometries acts transitively on the set of hyperplanes of
$E^n$. Since Euclidean isometries are M\"obius transformations, this completes
the proof if $\Sigma$ is a hyperplane. We may thus assume that $\Sigma$ is a
Euclidean sphere. Now, the group of Euclidean similarities acts transitively on
the set of spheres of $E^n$. Let $\psi$ be the similarity which maps $\Sigma$
to $S^{n-1}$. Then by Proposition~\ref{formsim}, $\psi$ has the form
$\psi(\vecx) = \psi(\veczero) + kA\vecx$, for some orthogonal matrix $A$ and
some $k > 0$.  The map $\vecx \mapsto k\vecx$ is the composition of the
reflection in $S(\veczero,1)$ followed by the reflection in
$S(\veczero,\sqrt{k})$, and so is a M\"obius transformation. Since $A$ is
orthogonal, $\vecx \mapsto A\vecx$ is an isometry, and so is also a M\"obius
transformation. Hence $\psi$ is the composition of an isometry (translation) and
a M\"obius transformation. Therefore $\psi$ is a M\"obius transformation as
well. We may now assume that $\Sigma = S^{n-1}$. Let $\sigma$ be the reflection
in the sphere $S(\vece_n,\sqrt{2})$.  Then $\sigma$ is stereographic
projection, so we have $\sigma(S^{n-1}) =\hat{E}^{n-1}$, and the proof is
complete. \end{proof}

The proof of the following important lemma is omitted for reasons of brevity.
See~\cite{rat1:fhm}.

\begin{lemma} \label{fixB} Let $\phi$ be a M\"obius transformation which fixes
every point of a sphere $\Sigma$ of $\hat{E}^n$.  Then $\phi$ is either the
identity or the reflection in $\Sigma$. \end{lemma}

\section{The hyperboloid model} \label{hyperboloid}

We first establish useful definitions and results in Lorentzian $n$-space.
Then, we define the hyperboloid model and show that it is a metric space.  A
discussion of its geodesics and trigonometry completes this section.

\subsection{Lorentzian $n$-space}

Let \vecx\ and \vecy\ be vectors in $\R^n$, where $n > 1$, and consider the symmetric
bilinear form \[ \vecx \circ \vecy = -x_1 y_1 + x_2 y_2 + \cdots + x_n y_n. \]
This form is known as the Lorentzian inner product, although it is not positive
definite. The space of $\R^n$ together with $\circ$ is called Lorentzian
$n$-space.  The Lorentzian norm derived from the Lorentzian inner
product is $\| \vecx \| = (\vecx \circ \vecx)^{1/2}$. Then, $\| \vecx \|$ may be
positive, zero, or positive imaginary. If $\| \vecx \|$ is imaginary, we denote
its modulus by $\lnorm \vecx \lnorm$.

Using the Lorentzian norm, the vectors in $\R^n$ may be partitioned into space-,
light- and time-like sets (these labels derive from the use of Lorentzian
4-space as a model for special relativity). Space-like vectors are defined to
be those with positive Lorentzian norm, and light-like vectors those with
Lorentzian norm zero. A vector \vecx\ is time-like if $\| \vecx \|$ is
imaginary, and is positive or negative time-like as $x_1 > 0$  or  $x_1 < 0$.
In $\R^3$, the light-like vectors are the cone $x_1^2 = x_2^2 + x_3^2$, the
space-like vectors are the exterior of this cone, and the time-like vectors are
the two disjoint convex sets forming the cone's interior.

We now concentrate on time-like vectors and prove an analogue of the
Cauchy--Schwartz inequality. Two vectors $\vecx$ and $\vecy$ are said to be
Lorentz orthogonal if $\vecx \circ \vecy = 0$. If vectors $\vecx$ and $\vecy$
with non-zero Lorentzian norm are Lorentz orthogonal then they are linearly
independent.

\begin{propn} \label{orthog} Let $\vecx$ and $\vecy$ be non-zero Lorentz
orthogonal vectors in $\R^n$.  If $\vecx$ is time-like then \vecy\ is
space-like. \end{propn}

\begin{proof} Since \vecx\ is time-like, $x_1^2 > x_2^2 + \cdots + x_n^2$.
Therefore, \[ 1 > \left( \sum_{i=2}^n x_i^2 \right)x_1^{-2}. \] As $\vecx \circ
\vecy = 0$, we have $x_1y_1 = x_2y_2 + \cdots + x_ny_n$.  Then, using the
Cauchy--Schwartz inequality for the Euclidean inner product, \begin{align*} \|
\vecy \|^2 & =  -y_1^2 + y_2^2 + \cdots + y_n^2 \\ & =  -\left[ \left(
\sum_{i=2}^n x_iy_i \right)x_1^{-1} \right]^2 + \sum_{i=2}^n y_i^2 \\ & \geq
-\left( \sum_{i=2}^n x_i^2 \right)\left( \sum_{i=2}^n y_i^2 \right)x_1^{-2} +
\sum_{i=2}^n y_i^2 \\ & =  \left( \sum_{i=2}^n y_i^2 \right)\left[ 1 - \left(
\sum_{i=2}^n x_i^2 \right)x_1^{-2}\right] \\ & \geq 0. \end{align*}
{}From the second last line, if $\| \vecy \|^2 = 0$ then $\sum_{i=2}^n y_i^2 =
0$, implying $y_i = 0$ for $2 \leq i \leq n$. Then $y_1 =  (
\sum_{i=2}^n x_iy_i )x_1^{-1} = 0$, and so $\vecy = 0$.  But $\vecy$ is
non-zero.  Thus, $\| \vecy \| > 0$, that is, \vecy\ is space-like. \end{proof}

An $n \times n$ matrix $A$ such that $A\vecx \circ A\vecy = \vecx \circ \vecy$
for all vectors \vecx\ and \vecy\ in $\R^n$ is said to be Lorentzian.  The set
of all Lorentzian matrices, denoted by $O(1,n-1)$, forms a group under matrix
multiplication. The subset of matrices $A \in O(1,n-1)$ which transform
positive time-like vectors into positive time-like vectors is in fact a
subgroup of index two of $O(1,n-1)$, and is denoted by $PO(1,n-1)$. A vector
subspace of $\R^n$ is said to be time-like if it contains a time-like vector.

\begin{propn} \label{transitive} For each dimension $m$, the action of
$PO(1,n-1)$ on the set of $m$-dimensional time-like vector subspaces of $\R^n$
is transitive. \end{propn}

\begin{proof}  The method is a modification of the Gram--Schmidt process.  Let $V$ be an
$m$-dimensional time-like vector subspace of $\R^n$.  Choose a  basis $\{
\vecu_1,\vecu_2, \ldots, \vecu_m\}$ for $V$ so that $\vecu_1$ is time-like.
Then let $\vecw_1 = \vecu_1/\lnorm \vecu_1 \lnorm$, hence $\vecw_1 \circ
\vecw_1 = -1$.  Let $\vecv_2 = \vecu_k + (\vecu_k \circ \vecw_1)\vecw_1$, and
$\vecw_2 = \vecv_2/\| \vecv_2\|$. For $3 \leq k \leq n$, let the vector $\vecv_k$
be \[ \vecv_k = \vecu_k + (\vecu_k\circ\vecw_1)\vecw_1 -
\sum_{i=2}^{k-1}(\vecu_k\circ\vecw_i)\vecw_i, \] and let $\vecw_k =
\vecv_k/\| \vecv_k\|$.  Then for $2 \leq k \leq n$, we have $\vecv_k
\circ \vecw_1 = 0$, and so by Proposition~\ref{orthog}, the vector $\vecv_k$ is
space-like.  Let $A$ be the matrix with columns $\vecw_1,\vecw_2,\ldots,\vecw_n$, so
that $A(\R^m) = V$.  The following calculation shows that $A$ preserves the
Lorentzian inner product and is thus an element of $PO(1,n-1)$: \begin{align*}
A\vecx \circ A\vecy & =  A(x_1\vece_1 + \cdots + x_n\vece_n)\circ A(y_1\vece_1
+ \cdots + y_n\vece_n)\\ & =  (x_1A\vece_1 + \cdots + x_nA\vece_n)\circ
(y_1A\vece_1 + \cdots + y_nA\vece_n) \\ & =  (x_1\vecw_1 + x_2\vecw_2 + \cdots
+ x_n\vecw_n)\circ (y_1\vecw_1 + y_2\vecw_2 + \cdots + y_n\vecw_n) \\ & =
-x_1y_1 + x_2y_2 + \cdots + x_ny_n \\ & =  \vecx \circ \vecy. \end{align*}
\end{proof}

\begin{propn} \label{CS} Suppose \vecx and \vecy\ are positive time-like
vectors in $\R^n$.  Then $\vecx\circ\vecy\leq\|\vecx \|\|\vecy\|$, with
equality if and only if \vecx\ and \vecy\ are linearly dependent. \end{propn}

\begin{proof} By Proposition~\ref{transitive}, there is a Lorentzian matrix $A$
so that $A\vecx = t\mathbf{e}_1$ for some $t > 0$. We may thus replace \vecx\
and \vecy\ by $A\vecx$ and $A\vecy$ without affecting the inner product or
norms. So we may assume, without loss of generality, that $\vecx = x_1
\mathbf{e}_1$. Then, \begin{align*} \| \vecx \|^2 \| \vecy \|^2 & =
-x_1^2(-y_1^2 + y_2^2 + \cdots + y_n^2) \\ & =  x_1^2y_1^2 - x_1^2(y_2^2 +
\cdots + y_n^2) \\ & \leq x_1^2 y_1^2 \\ & =  (\vecx \circ \vecy)^2,
\end{align*} with equality if and only if $y_2^2 + \cdots + y_n^2 = 0$, that
is, $\vecy = y_1 \mathbf{e}_1$.  Now, $\| \vecx \|$ and $\| \vecy \|$ are both
positive  imaginary, so $\| \vecx \| \| \vecy \| < 0$. Also, \[ \vecx \circ
\vecy = -x_1y_1 < 0. \] So $ \vecx \circ \vecy \leq \| \vecx \| \| \vecy \|$,
with equality if and only if \vecx\ and \vecy\ are linearly dependent.
\end{proof}

The following corollary uses this inequality to define the Lorentzian
time-like angle between positive time-like vectors $\vecx$ and $\vecy$.  The
Lorentzian angle is denoted $\eta(\vecx,\vecy)$, and will be used to define a
distance function for the hyperboloid model of hyperbolic $n$-space.

\begin{corollary}\label{cosheta} If \vecx\ and \vecy\ are positive time-like
vectors in $\R^n$, then there is a unique non-negative real number
$\eta(\vecx,\vecy)$ such that \[ \vecx \circ \vecy = \| \vecx\| \| \vecy\|
\cosh \eta(\vecx,\vecy). \] Moreover, $\eta(\vecx,\vecy) = 0$ if and only if
$\vecx$ and $\vecy$ are positive scalar multiples of each other.
\end{corollary}

To conclude our discussion of Lorentzian $n$-space, we define a Lorentzian
cross product between vectors in $\R^3$.  Let $\vecx$ and $\vecy$ be in $\R^3$,
and let $J$ be the matrix \[ J = \startm -1 & 0 & 0 \\ 0 & 1 & 0 \\ 0 & 0 & 1
\finishm. \] The Lorentzian cross product of \vecx\ and \vecy\ is defined to be
$ \vecx \otimes \vecy = J(\vecx \times \vecy)$, where~$\times$ is the usual
cross product in $\R^3$. It can be calculated that $\vecx \otimes \vecy =
J\vecy \times J\vecx$, and also that $\vecx \circ \vecy = \vecx \cdot J\vecy$.
These identities can then be used to prove the following facts about the
Lorentzian cross product.

\begin{lemma} \label{crossprodfacts} If \vecw, \vecx, \vecy\ and \vecz\ are
vectors in $\R^3$, then \begin{enumerate} \item $\vecx \circ (\vecx \otimes
\vecy) = \vecy \circ (\vecx \otimes \vecy)= 0$, \item $\vecx \otimes \vecy = -
\vecy \otimes \vecx$, \item $(\vecx \otimes \vecy) \circ \vecz = \det
\begin{pmatrix} x_1 & x_2 & x_3 \\ y_1 & y_2 & y_3 \\ z_1 & z_2 & z_3
\end{pmatrix}$, \item $\vecx \circ (\vecy \otimes \vecz) = (\vecx \otimes
\vecy) \circ \vecz$, \item $\vecx \otimes (\vecy \otimes \vecz) = (\vecx \circ
\vecy)\vecz - (\vecz \circ \vecx)\vecy$, \item  $(\vecx \otimes
\vecy)\circ(\vecz \otimes \vecw) = \det\begin{pmatrix} \vecx \circ \vecw &
\vecx \circ \vecz \\ \vecy \circ \vecw & \vecy \circ \vecz \end{pmatrix}$.
\end{enumerate} \end{lemma}

The following corollaries define the Lorentzian angle $\eta(\vecx,\vecy)$ between
pairs of space-like vectors \vecx\ and \vecy, and provide some useful identities
involving Lorentzian angles and trigonometric or hyperbolic trigonometric
functions.

\begin{corollary} \label{timecrossisspace} Let \vecx\ and \vecy\ be linearly
independent positive time-like vectors in $\R^3$. Then $\vecx \otimes \vecy$ is
space-like, and $\| \xcrossy \| = - \| \vecx \| \| \vecy \| \sinh \eta (\vecx,
\vecy)$. \end{corollary}

\begin{proof} We have \begin{align*} \| \xcrossy \|^2 & =  (\xcrossy) \circ
(\xcrossy) \\ & =  (\vecx \circ \vecy)^2 - \| \vecx \|^2 \| \vecy \|^2 \\ & =
\| \vecx \|^2 \| \vecy \|^2\cosh^2 \eta(\vecx,\vecy) - \| \vecx \|^2 \| \vecy
\|^2 \\ & =  \| \vecx \|^2 \| \vecy \|^2 \sinh^2 \eta(\vecx, \vecy).
\end{align*} Since $\eta(\vecx,\vecy) > 0$, and $\| \vecx \|$ and $\| \vecy \|$
are positive imaginary, on taking square roots of both sides we get \[ \|
\xcrossy \| = - \| \vecx \| \| \vecy \| \sinh \eta (\vecx, \vecy). \] The
right-hand side is real and positive, so \xcrossy\ is space-like. \end{proof}

\begin{corollary} \label{coseta} Let \vecx\ and \vecy\ be space-like vectors in
$\R^3$. If \xcrossy\ is time-like then $|\vecx \circ \vecy| < \| \vecx \| \|
\vecy \|$, and so there is a unique real number $\eta(\vecx,\vecy) \in (0,\pi)$
such that \[ \vecx \circ \vecy = \| \vecx \| \| \vecy \| \cos
\eta(\vecx,\vecy). \] \end{corollary}

\begin{proof} Follows immediately from the equality $\| \xcrossy \|^2 = (\vecx
\circ \vecy)^2 - \| \vecx \|^2 \| \vecy \|^2$. \end{proof}

\begin{corollary}\label{sineta} Let \vecx\ and \vecy\ be space-like vectors in
$\R^3$.  If \xcrossy\ is time-like, then \[ \lnorm \xcrossy \lnorm = \| \vecx
\| \| \vecy \| \sin \eta (\vecx,\vecy). \] \end{corollary}

\begin{proof} By Corollary~\ref{coseta}, since \xcrossy\ is time-like, $\vecx
\circ \vecy = \| \vecx \| \| \vecy \| \cos \eta(\vecx,\vecy)$.  Then by
Lemma~\ref{crossprodfacts}, \begin{align*} \| \xcrossy \|^2 & = (\vecx \circ
\vecy)^2 - \| \vecx \|^2 \| \vecy \|^2 \\ & =  \| \vecx \|^2 \| \vecy
\|^2\cos^2 \eta(\vecx,\vecy) - \| \vecx \|^2 \| \vecy \|^2 \\ & =  -\| \vecx
\|^2 \| \vecy \|^2 \sin^2 \eta(\vecx,\vecy), \end{align*} and the conclusion
follows from \xcrossy\ being time-like. \end{proof}

\subsection{The hyperboloid model}

We define $H^n$, the hyperboloid model of hyperbolic $n$-space, to be the set
of positive time-like vectors of norm $i$ in Lorentzian $(n+1)$-space: \[ H^n =
\{ \vecx \in \R^{n+1} : \| \vecx \|^2 = -1 \mbox{ and } x_1 > 0 \}. \] We are
disregarding the negative time-like vectors of norm $i$ so as to obtain a
connected set. The hyperbolic plane $H^2$ is thus one sheet of the two-sheeted
hyperboloid $x_1^2 = x_2^2 + x_3^2 + 1$. Let \vecx\ and \vecy\ be vectors in
$H^n$ and let $\eta(\vecx,\vecy)$ be the Lorentzian time-like angle between
them. The hyperbolic distance between \vecx\ and \vecy\ is defined to be \[
d_H(\vecx,\vecy) = \eta(\vecx,\vecy). \] Since $\vecx \circ \vecy = \| \vecx \|
\| \vecy \| \cosh \eta(\vecx,\vecy)$, and the norms of \vecx\ and \vecy\ are
equal to $i$, we have the explicit formula \[ d_H(\vecx,\vecy) =
\cosh^{-1}(-\vecx \circ \vecy). \]

\begin{propn} The function $d_H$ is a distance function. \end{propn}

\begin{proof} Since $\vecx \circ \vecy = \vecy \circ \vecx$, $d_H$ is
symmetric.  By the definition of $\eta(\vecx,\vecy)$ in
Corollary~\ref{cosheta}, $d_H$ is non-negative. If $\vecx = \vecy$, then
$\cosh^{-1}(-\vecx \circ \vecy) = \cosh^{-1}(1) = 0$, so $d_H(\vecx,\vecy) =
0$.  Conversely, if $d_H(\vecx,\vecy) = 0$, then by the definition of
$\eta(\vecx,\vecy)$, the vectors \vecx\ and \vecy\ must be positive scalar multiples of
each other.  But \vecx\ and \vecy\ have the same norm, which implies $\vecx =
\vecy$.

For the triangle inequality, we need to show that for all $\vecx$, \vecy\ and
\vecz\ in $H^n$, \[ d_H(\vecx,\vecz) \leq d_H(\vecx,\vecy) + d_H(\vecy,\vecz).
\] By Proposition~\ref{transitive}, there exists a Lorentzian matrix $A$ such
that $A\vecx$, $A\vecy$ and $A\vecz$ are in the subspace of $\R^{n+1}$ spanned
by $\mathbf{e}_1$, $\mathbf{e}_2$ and $\mathbf{e}_3$.  So we may, without loss
of generality, assume that $n=2$. This allows us to use results about the
Lorentzian cross product. By Corollary~\ref{timecrossisspace}, \xcrossy\ and
\ycrossz\ are space-like, and \[ \| \xcrossy \| = \sinh\eta(\vecx,\vecy),
\hspace{5mm} \| \ycrossz \| = \sinh \eta(\vecy,\vecz). \] We have, also,
\[ (\xcrossy) \otimes (\ycrossz)  =  ((\xcrossy)\circ \vecy)\vecz
- (\vecz\circ(\xcrossy))\vecy  =  - (\vecz\circ(\xcrossy))\vecy,
\] so the vectors \vecy\ and $(\xcrossy) \otimes (\ycrossz)$ are
linearly dependent.  Thus, $(\xcrossy) \otimes (\ycrossz)$ is either zero
or time-like, so we have by Corollary~\ref{coseta} that \[
|(\xcrossy)\circ(\ycrossz)| \leq \|\xcrossy\| \|\ycrossz\|. \] Combining all
these facts, we obtain \begin{align*} \cosh(\eta(\vecx,\vecy) +
\eta(\vecy,\vecz)) & =  \cosh\eta(\vecx,\vecy)\cosh\eta(\vecy,\vecz) +
\sinh\eta(\vecx,\vecy)\sinh\eta(\vecy,\vecz) \\ & =  (-\vecx \circ
\vecy)(-\vecy \circ \vecz) + \| \xcrossy \| \| \ycrossz \| \\ & \geq (\vecx
\circ \vecy)(\vecy \circ \vecz) +  ( \xcrossy )\circ( \ycrossz )\\ & =  (\vecx
\circ \vecy)(\vecy \circ \vecz) + ((\vecx \circ \vecz)(\vecy \circ \vecy) -
(\vecx \circ \vecy)(\vecy \circ \vecz)) \\ & =  -\vecx \circ \vecz \\ & =
\cosh\eta(\vecx,\vecz). \end{align*} Since $\cosh$ is monotonic increasing, we
infer that $\eta(\vecx,\vecz) \leq \eta(\vecx,\vecy) + \eta(\vecy,\vecz)$, as
required. \end{proof}

The topology induced on $H^n$ by the distance function $d_H$ is the same as the
subspace topology of $H^n$, where $H^n$ is regarded as a subspace of $\R^{n+1}$
with its usual metric topology.

\subsubsection{Hyperbolic geodesics}

We now define hyperbolic lines, and show that the hyperbolic lines of $H^n$
are its geodesics.  A hyperbolic line of $H^n$ is the non-empty intersection of
$H^n$ with a two-dimensional vector subspace of $\R^{n+1}$. If \vecx\ and
\vecy\ are distinct vectors in $H^n$ then $L(\vecx,\vecy) = H^n \cap
\Span\{ \vecsxy \}$ is the unique hyperbolic line containing \vecx\ and
\vecy. Observe that $L(\vecx,\vecy)$ is one branch of a hyperbola.  We first
establish a sufficient condition for points of $H^n$ to lie on the same
hyperbolic line.

\begin{lemma} \label{collinear} If \vecx, \vecy\ and \vecz\ are points of $H^n$
and \[ \eta(\vecx,\vecy) + \eta(\vecy,\vecz) = \eta(\vecx,\vecy), \] then \vecx,
\vecy\ and \vecz\ are hyperbolically collinear, that is, there is a hyperbolic
line of $H^n$ containing \vecx, \vecy\ and \vecz. \end{lemma}

\begin{proof} Since $\Span\{\vecx,\vecy,\vecz\}$ has dimension at most 3,
we may assume that $n = 2$ and use Lorentzian cross product results.  From the
proof that $d_H$ is a distance function, $ \eta(\vecx,\vecy) +
\eta(\vecy,\vecz) = \eta(\vecx,\vecy)$ if and only if
\begin{equation}\label{collineareqn}(\xcrossy)\circ(\ycrossz) = \|\xcrossy\|
\|\ycrossz\|.\end{equation} Also from that proof, the vector
$(\xcrossy)\otimes(\ycrossz)$ is either zero or time-like. But if it were
time-like, equation~(\ref{collineareqn}) would contradict
Corollary~\ref{coseta}.  So $(\xcrossy)\otimes(\ycrossz)$ is $\veczero$.  Now,
since \[ (\xcrossy)\otimes(\ycrossz) = -((\xcrossy)\circ\vecz)\vecy \] and
\vecy\ is time-like, we have that $(\xcrossy)\circ\vecz = 0$. Therefore, by
Lemma~\ref{crossprodfacts}, the vectors \vecx, \vecy\ and \vecz\ are linearly
dependent.  So \vecx, \vecy\ and \vecz\ lie in a two-dimensional subspace of
$\R^{n+1}$, and are thus hyperbolically collinear.\end{proof}

The next proposition relates a geodesic curve $\vecalpha$ to a parametrisation
using Lorentz orthonormal vectors, and to a differential equation.  Two vectors
\vecx\ and \vecy\ in $\R^{n+1}$ are said to be Lorentz orthonormal if $\| \vecx
\|^2 = -1$, $\vecx \circ \vecy = 0$, and $\| \vecy \|^2 = 1$. This definition
means that \vecx\ must be time-like and \vecy\ space-like.  Note that
$\mathbf{e}_1$ and $\mathbf{e}_2$ are Lorentz orthonormal.

\begin{propn} \label{geodtfae} Let $\vecalpha:[a,b] \rightarrow H^n$ be a
curve. Then the following are equivalent:\begin{enumerate} \item The curve
$\vecalpha$ is a geodesic curve. \item There are Lorentz orthonormal vectors
\vecx, \vecy\ in $\R^{n+1}$ such that \[ \vecalpha(t) = (\cosh(t-a))\vecx +
(\sinh(t-a))\vecy. \] \item The curve $\vecalpha$ satisfies the differential
equation $\vecalpha'' - \vecalpha = \veczero$. \end{enumerate} \end{propn}

\begin{proof} If $A$ is a Lorentzian matrix then $(A\vecalpha)' = A\vecalpha'$,
so $\vecalpha$ satisfies the differential equation if and only if $A\vecalpha$
does.  We may thus transform $\vecalpha$ by a Lorentzian matrix.

Suppose $\vecalpha$ is a geodesic curve, and $t \in [a,b]$. We prove the second
statement holds.  Now, \begin{align*} \eta(\vecalpha(a),\vecalpha(b)) & =  b -
a \\ & =  (t - a) + (b - t) \\ & =  \eta(\vecalpha(a),\vecalpha(t)) +
\eta(\vecalpha(t),\vecalpha(b)). \end{align*} By Lemma~\ref{collinear},
$\vecalpha(a)$, $\vecalpha(t)$ and $\vecalpha(b)$ are hyperbolically
collinear.  So $\vecalpha([a,b])$ is contained in a hyperbolic line  of $H^n$.
Thus we may assume that $n = 1$, and so consider $H^1$ as the branch of the
hyperbola $x_1^2 = 1 + x_2^2$ on which $x_1 > 0$.  By
Proposition~\ref{transitive}, we may apply a Lorentz transformation to map
$\vecalpha(a)$ to $\mathbf{e}_1$, and so we may assume that $\vecalpha(a) =
\mathbf{e}_1$. Then, \[ \mathbf{e}_1 \cdot \vecalpha(t)  =
(-\vecalpha(a))\circ \vecalpha(t)  =  \cosh\eta(\vecalpha(a),\vecalpha(t)) =
\cosh(t-a). \] {}From the equation $x_1^2 = 1 + x_2^2$, we have $\mathbf{e}_2
\cdot \vecalpha(t) = \pm\sinh(t-a)$.  Since $\vecalpha$ is continuous, either $\mathbf{e}_2 \cdot \vecalpha(t) = \sinh(t-a)$ for all $t$, or
$\mathbf{e}_2 \cdot \vecalpha(t) = -\sinh(t-a)$ for all $t$. In the second
case, we may reflect in the $x_2$-axis, and so we may assume that $\vecalpha$
has the form \[ \vecalpha(t) = (\cosh(t-a))\mathbf{e}_1 + (\sinh(t-a))\mathbf{e}_2.
\]

Now suppose there are Lorentz orthonormal vectors \vecx\ and \vecy\ such that
\[ \vecalpha(t) = (\cosh(t-a))\vecx + (\sinh(t-a))\vecy. \] Let $s$ and $t$ be such
that $a \leq s \leq t \leq b$.  Then \begin{align*} \cosh
\eta(\vecalpha(s),\vecalpha(t)) & =  - \vecalpha(s) \circ \vecalpha(t) \\ & =
-(-\cosh(s-a)\cosh(t-a) + \sinh(s-a)\sinh(t-a)) \\ & =  \cosh((t-a)-(s-a)) \\ &
=  \cosh(t-s), \end{align*} so $\eta(\vecalpha(s),\vecalpha(t)) = t-s$.
Therefore $\vecalpha$ is a geodesic curve.

If the second statement holds then just differentiate to find that $\vecalpha''
- \vecalpha = \veczero$.

Finally, suppose that $\vecalpha'' - \vecalpha = \veczero$. Then \begin{equation}
\label{threetotwo} \vecalpha(t) = (\cosh(t-a))\vecalpha(a) +
(\sinh(t-a))\vecalpha'(a) \end{equation} is a solution of this differential
equation which satisfies the initial value conditions at $t = a$. By
uniqueness of such solutions, $\vecalpha(t)$ must have the
form~(\ref{threetotwo}). We now just need to prove that $\vecalpha(a)$ and
$\vecalpha'(a)$ are Lorentz orthonormal. By differentiating the equation
$\vecalpha(t) \circ \vecalpha(t) = -1$ we obtain $\vecalpha(t) \circ
\vecalpha'(t)= 0$ for all $t$. In particular, $\vecalpha(a) \circ \vecalpha'(a)
= 0$. Also, \[ \| \vecalpha(t) \|^2  =  \vecalpha(t) \circ \vecalpha(t)  =
\cosh^2(t-a)\|\vecalpha(a)\|^2 + \sinh^2(t-a)\| \vecalpha'(a)\|^2.
\] Since $\vecalpha$ maps into $H^2$, $\|\vecalpha(a)\|^2 = \|
\vecalpha(t) \|^2 = -1$, and so $\| \vecalpha'(a)\|^2 = 1$.  This completes the
proof. \end{proof}

The next proposition uses these results to characterise the geodesic lines of
$H^n$.

\begin{propn} \label{geodeqn} A function $\vecl: \R \rightarrow H^n$ is a
geodesic line if and only if there are Lorentz orthonormal vectors \vecx\ and
\vecy\ such that \[ \vecl(t) = (\cosh t )\vecx + (\sinh t)\vecy. \]
\end{propn}

\begin{proof} Suppose such vectors \vecx\ and \vecy\ exist.  Then $\vecl$
satisfies the differential equation $\vecl'' - \vecl = \veczero$.  By
Proposition~\ref{geodtfae}, the restriction of $\vecl$ to any interval $[a,b]$
is a geodesic curve, and so $\vecl$ is a geodesic line.

For the converse, suppose $\vecl$ is a geodesic line.  Then $\vecl''  - \vecl =
\veczero$ so, as in the proof of Proposition~\ref{geodtfae}, \[ \vecl(t) = (\cosh
t)\vecl(0) + (\sinh t)\vecl'(0). \] The vectors $\vecl(0)$ and $\vecl'(0)$ are
Lorentz orthonormal by the same argument as in the proof of
Proposition~\ref{geodtfae}. \end{proof}

\begin{corollary} \label{geodlines} The geodesics of $H^n$ are its hyperbolic
lines. \end{corollary}

\begin{proof} Suppose $\vecl$ is a geodesic line of $H^n$.  By
Proposition~\ref{geodeqn}, there exist vectors $\vecx$ and $\vecy$, one of
which is time-like, such that \[ \vecl(t) = (\cosh t )\vecx + (\sinh t )\vecy.
\] Then the image of $\vecl$ is the intersection of $H^n$ and
$\Span\{\vecx,\vecy\}$, which is a hyperbolic line.

Conversely, let $L$ be a hyperbolic line of $H^n$.  As in the proof of
Proposition~\ref{geodtfae}, we may assume that $n = 1$.  Then $L = H^1$.
Define $\vecl: \R \rightarrow H^1$ by \[ \vecl(t) = (\cosh t )\mathbf{e}_1 +
(\sinh t )\mathbf{e}_2. \] Then by Proposition~\ref{geodeqn}, $\vecl$ is a
geodesic line which is onto $H^1 = L$.  So $L$ is a geodesic. \end{proof}

\subsubsection{Hyperbolic trigonometry}

We now restrict our attention to the hyperbolic plane $H^2$, and discuss
hyperbolic trigonometry.  For reasons of space, the treatment here is cursory;
again, see~\cite{rat1:fhm} for the details. In particular, the area of regions
in the hyperbolic plane is not needed at all in later chapters, so we do not
define area.

The definition of a hyperbolic triangle is as follows. Suppose \vecx, \vecy\ and \vecz\ in
$H^2$ are hyperbolically non-collinear. Let $L(\vecsxy)$ be the unique
hyperbolic line of $H^2$ containing \vecx\ and \vecy, and let $H(\vecx, \vecy,
\vecz)$ be the closed half-plane of $H^2$ with \vecz\ in its interior and
$L(\vecx,\vecy)$ as its boundary. The hyperbolic triangle with vertices at
\vecx, \vecy\ and \vecz\ is defined to be \[ T(\vecx,\vecy,\vecz) =
H(\vecx,\vecy,\vecz) \cap H(\vecy,\vecz,\vecx) \cap  H(\vecz,\vecx,\vecy). \]
The sides of this triangle are the geodesic segments $[\vecx,\vecy]$,
$[\vecy,\vecz]$ and $[\vecz,\vecx]$.  To define the angles of this triangle,
let $a = d_H(\vecy,\vecz)$, $b = d_H(\vecz,\vecx)$ and $c = d_H(\vecsxy)$.  Let
$\mathbf{f}:[0,a] \rightarrow H^2$, $\mathbf{g}:[0,b] \rightarrow H^2$ and
$\mathbf{h}:[0,c] \rightarrow
H^2$ be the geodesic curves joining \vecy\ to \vecz, \vecz\ to \vecx, and
\vecx\ to \vecy\ respectively.  Then, the hyperbolic angle $\alpha$ at the
vertex \vecx\ is defined to be the Lorentzian angle between the vectors
$-\mathbf{g}'(b)$ and $\mathbf{h}'(0)$.  Angles $\beta$ and $\gamma$ at vertices \vecy\ and
\vecz\ are defined similarly.  The following sketch shows the relationships
between vectors, sides and angles.

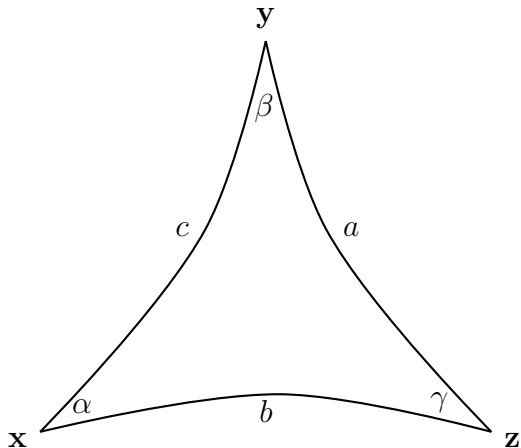
\begin{center}
\begin{pspicture}(0,0)(7,7)
\pscurve(4,6.2)(3.2,3.7)(1,1)
\pscurve(4,6.2)(4.8,3.7)(7,1)
\pscurve(1,1)(4.2,1.5)(7,1)
\uput[l](1,0.9){$\mathbf{x}$}
\uput[u](4,6.2){$\mathbf{y}$}
\uput[r](7,0.9){$\mathbf{z}$}
\uput[l](3.15,3.7){$c$}
\uput[r](4.85,3.7){$a$}
\uput[d](4,1.6){$b$}
\uput[r](1.25,1.34){$\alpha$}
\uput[u](3.98,4.95){$\beta$}
\uput[l](6.6,1.4){$\gamma$}
\end{pspicture}
\end{center}

The trigonometric identities below may be proved using
Lemma~\ref{crossprodfacts} and
Corollaries~\ref{timecrossisspace}--\ref{sineta}.  In each of these identities,
$\alpha$, $\beta$ and $\gamma$ are the angles of a hyperbolic triangle, and
$a$, $b$ and $c$ are the lengths of the opposite sides.

\begin{theorem}[Sine Rule]  \[ \frac{\sinh a}{\sin \alpha} = \frac{\sinh b}{\sin \beta} =
\frac{\sinh c}{\sin \gamma}. \] \end{theorem}

\begin{theorem}[First Cosine Rule]  \[ \cos \gamma = \frac{\cosh a \cosh b - \cosh c}{\sinh a
\sinh b}. \] \end{theorem}

\begin{theorem}[Second Cosine Rule] \[ \cosh c = \frac{\cos \alpha \cos \beta + \cos
\gamma}{\sin \alpha \sin \beta}. \] \end{theorem}

We now generalise the definition of hyperbolic triangle to include the
possibility of vertices at infinity.  Suppose $L_1$ and $L_2$ are hyperbolic
lines, with $L_1$ the intersection of $H^2$ with the plane
$\Span\{\vecx_1,\vecy_1\}$, and $L_2$ the intersection of $H^2$ with the plane
$\Span\{\vecx_2,\vecy_2\}$.  If $\Span\{\vecx_1,\vecy_1\} \cap
\Span\{\vecx_2,\vecy_2\}$ is a one-dimensional subspace of $\R^{n+1}$
consisting of light-like vectors, then $L_1$ and $L_2$ are said to meet at
infinity. Hyperbolic lines which meet at infinity are disjoint, but become
asymptotically close in one direction.  A generalised hyperbolic triangle,
then, is defined in the same way as a hyperbolic triangle, except that the
hyperbolic lines defining adjacent sides may meet at infinity. The angle at a
vertex at infinity is defined to be zero.

The Second Cosine Rule may be extended to a generalised hyperbolic triangle
with one vertex at infinity.  If such a triangle has angles $\alpha$,
$\beta=\frac{\pi}{2}$ and $\gamma=0$, and finite side of length $c$, we obtain
\[ c = \cosh^{-1}\left(\frac{1}{\sin \alpha}\right). \] This shows that there
is an absolute unit of length in hyperbolic geometry, as the length $c$ depends
only on the angle $\alpha$.

\begin{theorem} Let $T$ be a generalised hyperbolic triangle.  If the  angles
of $T$ are $\alpha$, $\beta$ and $\gamma$, then the area of $T$ is $\pi - \alpha
- \beta - \gamma.$ \end{theorem}

\begin{corollary} The angle sum of a generalised hyperbolic triangle is less
than~$\pi$.
\end{corollary}

\section{The projective disc model}

The projective disc model is mapped to $H^n$ by a projection known as gnomonic projection.
Its chief feature is that its hyperbolic lines are the same as its Euclidean
lines.

The open unit disc in $\R^n$ is the set $ D^n = \{ \vecx \in \R^n : |\vecx| <
1\}$. Identify $\R^n$ with $\R^n \times \{ 0 \}$ in $\R^{n+1}$.  Let $\mu$ be
the translation of $D^n$ vertically by $\vece_{n+1}$, followed by radial
projection to $H^n$. The map $\mu$ is known as gnomonic projection, is a
bijection, and has the explicit formula \[ \mu(\vecx) = \frac{\vecx +
\vece_{n+1}}{\lnorm \vecx + \vece_{n+1}\lnorm}. \]

We now define a distance function $d_D$ on $D^n$ by \[ d_D(\vecx,\vecy) =
d_H(\mu(\vecx),\mu(\vecy)). \] The projective disc model of hyperbolic
$n$-space is  the metric space consisting of $D^n$ together with the distance
function $d_D$.

A subset $L$ of $D^n$ is called a hyperbolic line of $D^n$ if $\mu(L)$ is a
hyperbolic line of $H^n$.  Since $\mu: D^n \rightarrow H^n$ is an isometry, and
the geodesics of $H^n$ are its hyperbolic lines, the geodesics of $D^n$ are its
hyperbolic lines. The following characterisation of the hyperbolic lines of
$D^2$ shows that in this model, hyperbolic lines are the same as Euclidean
lines.

\begin{propn} A subset $L$ of $D^2$ is a hyperbolic line of $D^2$ if and only
if $L$ is an open chord of $D^2$. \end{propn}

\begin{proof} Let $L$ be a hyperbolic line of $H^2$.  Then $L$ is the
intersection of $H^2$ with the plane $P(\vecu,0)$, where $\vecu$ is some  unit
vector in $\R^3$. Radial projection maps $L$ onto the line of intersection of
the plane $P(
\vecu,0)$ and the plane $P(\vece_3,1)$.  The translation by the vector $-\vece_3$
which follows ensures that $\mu^{-1}(L)$ is an open chord of $D^2$. This
argument can be reversed for the converse. \end{proof}

\section{The conformal ball model} \label{ball}

The conformal ball model is mapped to $H^n$ by a projection known as
stereographic projection.  We  show that the isometries of this model are
M\"obius transformations, and that the hyperbolic angles of this model are the
same as Euclidean angles.

First, we redefine the Lorentzian inner product on $\R^{n+1}$ to be \[ \vecx
\circ \vecy = x_1 y_1 + x_2 y_2 + \cdots + x_n y_n - x_{n+1}y_{n+1}. \] All
results of Section~\ref{hyperboloid} are still true, once the order of
coordinates in $\R^{n+1}$ is reversed.  Identify $\R^n$ with $\R^n \times \{ 0
\}$ in $\R^{n+1}$.  Let $B^n$ be the open unit ball in $\R^n$ (this is the same
set as $D^n$, but we will soon be defining a different distance function for
$B^n$, and so wish to distinguish the two metric spaces). The stereographic
projection $\zeta$ of $B^n$ onto $H^n$ is the projection of $\vecx \in B^n$
away from the vector $-\mathbf{e}_{n+1}$ until it meets $H^n$ in the unique
point $\zeta(\vecx)$.  The explicit formula for this map is \begin{equation}
\label{sterproj} \zeta(\vecx) = \left( \frac{2x_1}{1 - |\vecx|^2},\ldots,
\frac{2x_n}{1 - |\vecx|^2}, \frac{1 + |\vecx|^2}{1 - |\vecx|^2}\right).
\end{equation} The projection $\zeta$ is a bijection from $B^n$ to $H^n$.

The conformal ball model of hyperbolic $n$-space is defined to be the set $B^n$
together with the distance function $d_B$ given by \[ d_B(\vecx,\vecy) =
d_H(\zeta(\vecx),\zeta(\vecy)). \] By the explicit formula for $d_H$,
and~(\ref{sterproj}), \begin{align*} \cosh d_H(\zeta(\vecx),\zeta(\vecy))
& =  -\zeta(\vecx) \circ \zeta(\vecy) \\ &=  \frac{-4\vecx \cdot \vecy + (1 +
|\vecx|^2)(1 + |\vecy|^2 )}{(1- |\vecx|^2)(1 - |\vecy|^2)} \\ & = 1 + \frac{2|\vecx - \vecy |^2}{(1-
|\vecx|^2)(1 - |\vecy|^2)}. \end{align*} Thus an explicit formula for $d_B$ is
\[ d_B(\vecx,\vecy) = \cosh^{-1}\left(1 + \frac{2|\vecx - \vecy|^2}{(1 -
|\vecx|^2)(1 - |\vecy|^2)}\right). \] The topology induced on $B^n$ by $d_B$ is
the same as the subspace topology of $B^n$ as a subspace of $\R^n$.

A subset $L$ of $B^n$ is defined to be a hyperbolic line of $B^n$ if $\zeta(L)$
is a hyperbolic line of $H^n$. By Corollary~\ref{geodlines}, and the fact that
$\zeta$ is an isometry, the geodesics of $B^n$ are its hyperbolic lines.  We now
characterise the hyperbolic lines of $B^2$.

\begin{propn} \label{Blines}A subset $L$ is a hyperbolic line of $B^2$ if and
only if $L$ is either an open diameter of $B^2$, or the intersection of $B^2$
with a circle orthogonal to $S^1$. \end{propn}

\begin{proof} Let $L$ be an open diameter of $B^2$.  Then $L$ is the
intersection of $B^2$ with the one-dimensional subspace of $\R^3$
spanned by, say, the vector $\vecv$.  The projection $\zeta$ maps
$L$ onto the hyperbolic line of $H^2$ obtained by intersecting
$H^2$ with the two-dimensional subspace
$\Span\{\vecv,\vece_3\}$. Thus $L$ is a hyperbolic line of
$B^2$.  This argument can be reversed to show that if $\zeta(L)$
is a hyperbolic line of $H^2$ which passes through the point
$\vece_3$, then $L$ is an open diameter of $B^2$.

Now, if $\vecx=(x_1,x_2,x_3) \in H^2$ then $x_1^2 + x_2^2 - x_3^2
= -1$.  The inverse of the projection $\zeta$ is thus the map
\[
\zeta^{-1}(x_1,x_2,x_3) = \frac{1}{x_3 + 1}(x_1,x_2).
\]
Let $R$ be a rotation of $\R^3$ about the $x_3$ axis, and identify
$R$ with a rotation of $\R^2$.  Then $\zeta^{-1} (R(\vecx)) =
R(\zeta^{-1}(\vecx))$ for any $\vecx \in H^2$. So we may consider,
without loss of generality, a hyperbolic line of $H^2$ which is of
the form $H^2 \cap P(\vecw,0)$, where $\vecw=
(-\sin\theta,0,\cos\theta)$ for some
$\theta\in(\frac{\pi}{4},\frac{\pi}{2})$.   Let $L =
\zeta^{-1}(H^2 \cap P(\vecw,0))$.  Then $L$ is a hyperbolic line
of $B^2$. Now, $\vecx \in P(\vecw,0)$ if and only if $x_3 = x_1
\tan \theta$. This means that $\vecx \in H^2 \cap P(\vecw,0)$ if
and only if
\[
x_2^2 = x_1^2 \tan^2 \theta - x_1^2 - 1.
\]
Abbreviate $\tan\theta$ to $t$.  Then
\[
\zeta^{-1}(\vecx) = \frac{1}{x_3 + 1}(x_1,x_2) = \frac{1}{x_1 t
+ 1}\left(x_1,\sqrt{x_1^2 t^2 - x_1^2 - 1}\right).
\]
The circle in $\R^2$ with centre $(t,0)$ and radius $\sqrt{t^2 -
1}$ meets $S^1$ and is orthogonal to it. See the diagram in the
proof of Lemma~\ref{orthogB} below. We have
\[
| \zeta^{-1}(\vecx) - (t,0) |^2 = \left( \frac{x_1}{x_1 t + 1} -
t\right)^2 + \left(\frac{\sqrt{x_1^2 t^2 - x_1^2 - 1}}{x_1 t +
1}\right)^2 = \cdots = t^2 - 1.
\]
So $L$, the image of $H^2 \cap P(\vecw,0)$ under $\zeta^{-1}$,
lies in the circle centre $(t,0)$ and radius $\sqrt{t^2 - 1}$.
Since hyperbolic lines of $H^2$ are connected, and $\zeta^{-1}$ is
continuous, $L$ is a connected arc.  As $x_1 \rightarrow \infty$,
the arc $L$ approaches the boundary of $B^2$.  Thus $L$ is the
intersection of $B^2$ with a circle orthogonal to $S^1$.  This
argument can be reversed to prove the converse.\end{proof}

\subsection{M\"obius transformations and isometries}

A map $\phi:\hat{E}^n \rightarrow \hat{E}^n$ is said to leave $B^n$ invariant
if $\phi$ maps $B^n$ bijectively onto itself. (If, in addition, $\phi(\vecx) =
\vecx$ for all $\vecx$ in $B^n$, we say that $\phi$ fixes each point of
$B^n$.)  We define a M\"obius transformation of $B^n$ to be a M\"obius
transformation of $\hat{E}^n$ which leaves $B^n$ invariant. Our aim is to show
that $I(B^n)$, the group of isometries of $B^n$, is isomorphic to $M(B^n)$, the
group of M\"obius transformations of $B^n$.

The following lemma gives sufficient and necessary conditions for a reflection
in a sphere to leave $B^n$ invariant.  Two spheres are said to be orthogonal if
they intersect in $E^n$, and at each point of intersection their normal lines
are orthogonal.

\begin{lemma} \label{orthogB} Let $\sigma$ be the reflection of $\hat{E}^n$ in
the sphere $S(\mathbf{a},r)$.  Then $\sigma$ leaves $B^n$ invariant if and only
if $S(\mathbf{a},r)$ is orthogonal to $S^{n-1}$. \end{lemma}

\begin{proof} By considering the intersection of $S(\veca,r)$, $B^n$ and
$S^{n-1}$ with two-dimensional subspaces of $\hat{E}^n$ which contain the
vector \veca, we may assume that $n=2$.  The circles $S(\veca,r)$ and $S^1$ are
orthogonal if and only if, at their points of intersection, their respective
radii are orthogonal.  Thus, $S(\veca,r)$ is orthogonal to $S^1$ if and only if
$|\veca|^2 = r^2 + 1$.

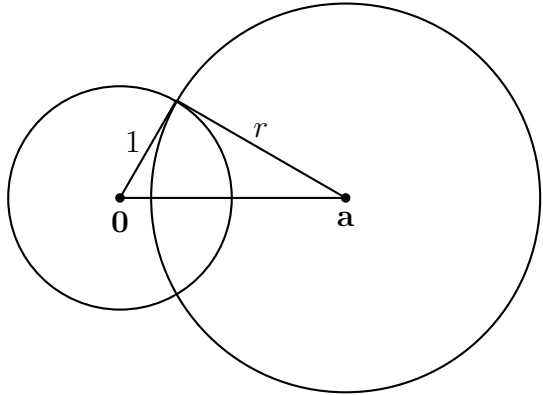
\begin{center}
\begin{pspicture}(0,0)(6,6)
\pscircle(1,3){1.5}
\psdot*(1,3)
\pscircle(4,3){2.6}
\psdot*(4,3)
\psline(1,3)(1.75,4.3)
\psline(1.75,4.3)(4,3)
\psline(4,3)(1,3)
\uput[d](1,3){$\mathbf{0}$}
\uput[d](4,3){$\mathbf{a}$}
\uput[l](1.45,3.75){1}
\uput[r](2.6,3.9){$r$}
\end{pspicture}
\end{center}

Suppose $S(\veca,r)$ is orthogonal to $S^1$. Then for $\vecx \in B^2$,
\begin{align*} |\sigma(\vecx)|^2 & =  \left| \veca + r^2 \frac{\vecx - \veca
}{|\vecx - \veca|^2}\right|^2 \\ & =  \left| \veca + \frac{(|\veca|^2 -
1)}{|\vecx - \veca|^2}(\vecx - \veca)\right|^2 \\ & = \cdots \\& =
\frac{|\vecx - \veca|^2 + (1 - |\vecx|^2)(1 - |\veca|^2)}{|\vecx - \veca|^2}\\
& < 1, \end{align*} since $|\vecx| < 1$ and $|\veca| > 1$.  Thus $\sigma(B^2)$
is contained in $B^2$.  To show that $\sigma$ is onto $B^2$, let $\vecy$ be in
$B^2$.  Then, since $\sigma$ is its own inverse, $\sigma(\vecx) = \vecy$ if and
only if $\vecx = \sigma(\vecy)$. But $\sigma$ maps $B^2$ into $B^2$, so $\vecx
\in B^2$, hence $\sigma$ is onto.  Therefore $\sigma$ leaves $B^2$ invariant.

For the converse, we identify $B^2$ with the open unit disc in $\mathbb{C}$ and
$S^1$ with the unit circle in $\mathbb{C}$. Without loss of generality, we may
consider the centre of the circle $S(\veca,r)$ to lie on the real axis at a
point $a \geq 0$. We prove the contrapositive.  Suppose $S(a,r)$ is not
orthogonal to $S^1$. We show that there exists a point of $B^2$ which is mapped
outside $B^2$ by $\sigma$.  If $a < 1$, let $\delta > 0$ be such that $a +
\delta < 1$. Then $a + \delta \in B^2$, but for sufficiently small $\delta$, \[
\sigma(a + \delta) = a + \frac{r^2}{\delta} > 1. \] If $a \geq 1$ and $a^2 <
r^2 + 1$, then for small enough $\delta > 0$, we have $a^2 + \delta(a + 1) <
r^2 + 1$ and $1 - \delta \in B^2$. However, \[ \sigma(1 - \delta) = a -
\frac{r^2}{a - (1 - \delta)} < -1. \] Finally, if $a > 1$ and $a^2 > 1 + r^2$,
then choose $\delta>0$ so that $a^2 - \delta(a - 1) > 1 + r^2$ and $-1 + \delta
\in B^2$. We have $\sigma(-1 + \delta) > 1$.  So, if $S(a,r)$ is not orthogonal
to $S^1$, the reflection $\sigma$ does not leave $B^2$ invariant. \end{proof}

\begin{lemma} \label{MobdB} If $\phi$ is a M\"obius transformation of
$\hat{E}^n$ leaving $B^n$ invariant, and $\vecx$, $\vecy$ are in $B^n$, then \[
\frac{|\phi(\vecx) - \phi(\vecy)|^2}{(1 - |\phi(\vecx)|^2)(1 -
|\phi(\vecy)|^2)} = \frac{|\vecx - \vecy|^2}{(1 - |\vecx|^2)(1 - |\vecy|^2)}.
\] \end{lemma}

\begin{proof} Since $\phi$ is a composition of reflections, we may assume that
$\phi$ is a reflection.  If $\phi$ is a reflection in a plane then, in order
that $\phi$ leave $B^n$ invariant, $\phi$ must be a reflection in a plane
passing through the origin.  Then $\phi$ preserves Euclidean distances, and
$\phi(\mathbf{0}) = \mathbf{0}$, so $|\phi(\vecx)| = |\vecx|$, $|\phi(\vecy)| =
|\vecy|$ and $|\phi(\vecx) - \phi(\vecy)| = |\vecx - \vecy|$.  The result
follows immediately.

We are left with the possibility that $\phi$ is a reflection in a sphere
$S(\mathbf{a},r)$.  By Lemma~\ref{phixphiy}, \[ \frac{|\phi(\vecx) -
\phi(\vecy)|}{|\vecx - \vecy|} = \frac{r^2}{|\vecx - \veca||\vecy - \veca|}. \]
By Lemma~\ref{orthogB}, $S(\veca,r)$ must be orthogonal to $S^{n-1}$, and from
its proof, this means that $|\veca|^2 = r^2 + 1$.  The formula for the
reflection $\phi$ is \[ \phi(\vecx) = \veca + \frac{r^2}{|\vecx -
\veca|^2}(\vecx - \veca), \] hence \begin{align*} |\phi(\vecx)|^2  - 1 & =
|\veca|^2 - 1 + \frac{2r^2}{|\vecx - \veca|^2}\veca \cdot (\vecx - \veca) +
\frac{r^4}{|\vecx - \veca|^2} \\ & =  \frac{r^2|\vecx - \veca|^2 +
2r^2\veca\cdot(\vecx - \veca) + r^4}{|\vecx - \veca|^2}  \\ & =  \cdots \\ & =
\frac{r^2(|\vecx|^2 - 1)}{|\vecx - \veca|^2}. \end{align*} Putting all this
together, we obtain \[ \frac{|\phi(\vecx) - \phi(\vecy)|^2}{|\vecx - \vecy|^2}  =
\frac{r^4}{|\vecx - \veca|^2|\vecy - \veca|^2} =  \frac{(1 - |\phi(\vecx)|^2
)(1 - |\phi(\vecy)|^2 )}{(1 - |\vecx|^2 )(1 - |\vecy|^2 )}. \] Rearrange this
equation to finish the proof. \end{proof}

\begin{theorem} \label{Mob_isom_B} Every M\"obius transformation of $\hat{E}^n$
which leaves $B^n$ invariant  restricts to an isometry of $B^n$,
and every isometry of $B^n$ extends to a unique M\"obius transformation of
$\hat{E}^n$ which leaves $B^n$ invariant. \end{theorem}

\begin{proof} The claim that M\"obius transformations restrict to isometries
follows immediately from the explicit formula for $d_B$ together with
Lemma~\ref{MobdB}.

Conversely, let $\phi: B^n \rightarrow B^n$ be an isometry. We construct an
extension for $\phi$ as follows.  First, for any $\vecb \in B^n$, we construct
a M\"obius transformation $\tau_\vecb$ such that $\tau_\vecb$ leaves $B^n$
invariant, and $\tau_\vecb (\veczero) = \vecb$.  Next, we consider the map
$\tau_{\phi(\veczero)}^{-1}\phi$, which fixes $\veczero$, and show that this
map is orthogonal.  It follows that $\tau_{\phi(\veczero)}^{-1}\phi$ extends to
a M\"obius transformation of $\hat{E}^n$.  We then prove that this extension is
unique.

Let $S(\veca,r)$ be a sphere of $E^n$ orthogonal to $S^{n-1}$, and let
$\sigma_\veca$ be the reflection in this sphere (since $r^2 = |\veca|^2 - 1$,
the radius $r$ is a function of $\veca$).  Let $\rho_\veca$ be the reflection
in the hyperplane through $\mathbf{0}$ normal to \veca.  Then by
Lemma~\ref{orthogB}, $\sigma_\veca$ and $\rho_\veca$ leave $B^n$ invariant, so
the composition $\rho_\veca \sigma_\veca$ leaves $B^n$ invariant.  A
calculation shows that $\rho_\veca\sigma_\veca(\mathbf{0}) =
-\veca/|\veca|^2$. For $\mathbf{b} \not = \mathbf{0}$ in $B^n$, set
$\mathbf{b}' = -\vecb/|\vecb|^2$.  If $r = (|\mathbf{b}'|^2
- 1)^{1/2}$ then $S(\vecb', r)$ is orthogonal to $S^{n-1}$, so we may define a
M\"obius transformation of $B^n$ by the formula $\tau_\vecb =
\rho_{\vecb'}\sigma_{\vecb'}$.  Define $\tau_\mathbf{0}$ to be the identity
map.  Then for all $\vecb$ in $B^n$, $\tau_\vecb(\mathbf{0}) = \vecb$.

We now define a map $\psi: B^n \rightarrow B^n$ by $\psi =
\tau_{\phi(\mathbf{0})}^{-1}\phi$. We have $\psi(\mathbf{0}) = \mathbf{0}$.
And, since $\psi$ is a composition of an isometry and a M\"obius
transformation, by the first part of this theorem, $\psi$ is an isometry of
$B^n$. Let $\vecx$ and $\vecy$ be in $B^n$.  Since $d_B(\psi(\vecx),\mathbf{0})
= d_B(\vecx,\mathbf{0})$, we have \[ \frac{|\psi(\vecx)|^2}{1 -
|\psi(\vecx)|^2} = \frac{|\vecx|^2}{1 - |\vecx|^2 }, \]  hence $|\psi(\vecx)| =
|\vecx|$.  Also, since $d_B(\psi(\vecx),\psi(\vecy)) = d_B(\vecx,\vecy)$, we
have \[ \frac{|\psi(\vecx) - \psi(\vecy)|^2}{(1 - |\psi(\vecx)|^2)(1 -
|\psi(\vecy)|^2)} = \frac{|\vecx - \vecy|^2}{(1 - |\vecx|^2)(1 - |\vecy|^2)},
\] and so $|\psi(\vecx) - \psi(\vecy)|=|\vecx - \vecy|$.  Thus $\psi$ preserves
Euclidean distances in $B^n$.  It is known that any mapping $\psi$ of the open
ball $B^n$ which preserves Euclidean distances and fixes the point $\veczero$
is the restriction to $B^n$ of an orthogonal transformation, say $A$, of $E^n$.
Hence, $\tau_{\phi(\mathbf{0})}A$ is an extension of $\phi$. Now, $A$ is a
Euclidean isometry, so is also a M\"obius transformation of $\hat{E}^n$.   Thus
$\tau_{\phi(\mathbf{0})}A$ is a M\"obius transformation of $\hat{E}^n$ which
leaves $B^n$ invariant and extends $\phi$.

To show that this extension is unique, suppose that $\upsilon$ is another
M\"obius transformation of $B^n$ which extends $\phi$, and let $\xi$ be the
composition $\upsilon^{-1}\tau_{\phi(\mathbf{0})}A$. By the continuity of $\xi$,
we have $\xi(\vecx) = \vecx$ for all $\vecx \in \overline{B^n}$. So by
Lemma~\ref{fixB}, $\xi$ is the identity.  Therefore $\tau_{\phi(\mathbf{0})}A$
is the unique M\"obius extension of $\phi$. \end{proof}

\begin{corollary} The groups $I(B^n)$ and $M(B^n)$ are isomorphic.
\end{corollary}

We can now explain why the hyperbolic angles of $B^n$ are the same as its
Euclidean angles.  From the geometric definition of the map $\zeta:B^n
\rightarrow H^n$, this projection preserves the Euclidean angle between any two
hyperbolic lines of $B^n$ which intersect at the origin. By considering the tangent space
to $H^n$ at $\zeta(\veczero) = \vece_{n+1}$, it can be proved that the
hyperbolic angle between two hyperbolic lines of $H^n$ intersecting at
$\zeta(\veczero)$ is the same as the Euclidean angle.  So, the
hyperbolic angle between two hyperbolic lines of $B^n$ which intersect at the
origin is the same as the Euclidean angle.  Moreover, the isometries of
$B^n$ are M\"obius transformations, which are conformal and act transitively on
the set of hyperbolic lines of $B^n$. Therefore, the hyperbolic angle between
any two intersecting hyperbolic lines in $B^n$ is the same as the Euclidean
angle between these lines. This is why $B^n$ is called the conformal ball model
of hyperbolic $n$-space.

\section{The upper half-space model}\label{upperhs}

The upper half-space model is mapped to $B^n$ by a M\"obius transformation.  As
in the conformal ball model, we may identify isometries and
M\"obius transformations.  In the two-dimensional case, there are
many interesting relationships between isometries, maps known as
linear fractional transformations, and groups of matrices.

Consider the upper half-space $U^n = \{ \vecx \in \hat{E}^n: x_n > 0\}$. There
is a standard transformation $\eta$ from $U^n$ onto the open unit ball $B^n$.
It is given by $\eta = \sigma \rho$, where $\rho$ is the reflection in the
boundary $\hat{E}^{n-1}$, and $\sigma$ is the reflection in the sphere
$S(\mathbf{e}_n,\sqrt{2})$.  The upper half-space model of hyperbolic space is
the set $U^n$ together with the distance function $d_U$ defined by \[
d_U(\vecsxy) = d_B(\eta(\vecx),\eta(\vecy)). \] The topology induced on $U^n$
by $d_U$ is the subspace topology of $U^n$ as a subspace of $\R^n$. We will
mainly be considering the upper half-plane $U^2$, which may be identified with
the set of complex numbers \[ \{ z \in \mathbb{C} : \im(z) > 0 \}. \]

\begin{lemma} The explicit formula for $d_U$ is \[ d_U(\vecx,\vecy) =
\cosh^{-1}\left(1 + \frac{|\vecx - \vecy|^2}{2x_n y_n}\right). \] \end{lemma}

\begin{proof} From the explicit formula for $d_B$, we have \[ \cosh
d_U(\vecx,\vecy)  =  \cosh d_B(\eta(\vecx),\eta(\vecy))  =  1 +
\frac{2|\sigma\rho(\vecx) - \sigma\rho(\vecy)|^2} {(1 -
|\sigma\rho(\vecx)|^2)(1 - |\sigma\rho(\vecy)|^2)}. \] Then, by
Lemma~\ref{phixphiy} applied to $\sigma$, and $\rho$ being a Euclidean
isometry, we have \[ |\sigma\rho(\vecx) - \sigma\rho(\vecy)|  =
\frac{2|\rho(\vecx) - \rho(\vecy)|} {|\rho(\vecx) - \vece_n||\rho(\vecy) -
\vece_n|} =  \frac{2|\vecx - \vecy|} {|\vecx + \vece_n||\vecy + \vece_n|}.
\] Also, \begin{align*} |\sigma\rho(\vecx)|^2 & =  \left| \vece_n +
\frac{2(\rho(\vecx)-\vece_n)} {|\rho(\vecx) - \vece_n|^2}\right|^2 \\ & =  1 +
\frac{4\vece_n\cdot(\rho(\vecx)-\vece_n)}{|\rho(\vecx) - \vece_n|^2} +
\frac{4}{|\rho(\vecx)- \vece_n|^2} \\ & =  1 + \frac{4[\rho(\vecx)]_n}{|\vecx +
\vece_n|^2}, \end{align*} hence \[ 1 - |\sigma\rho(\vecx)|^2 =
\frac{4x_n}{|\vecx + \vece_n|^2}. \] These results may be combined to
complete the proof. \end{proof}

\begin{corollary} Let $ia$ and $ib$, where $b \geq a > 0$, be points on the
imaginary axis in the upper half-plane $U^2$.  Then the hyperbolic distance
between $ia$ and $ib$ is $\log \frac{b}{a}$. \end{corollary}

A subset $L$ of $U^n$ is, by definition, a hyperbolic line of $U^n$ if
$\eta(L)$ is a hyperbolic line of $B^n$.  Thus the geodesics of $U^n$ are its
hyperbolic lines.  Now, the boundary of $U^n$ is mapped onto the boundary of
$B^n$ by $\eta$ which is  a M\"obius transformation.  So $\eta$ is conformal
and maps spheres to spheres. By Proposition~\ref{Blines}, a hyperbolic line
of $B^2$ is the intersection of $B^2$ with either a Euclidean line or a
Euclidean circle which is orthogonal to its boundary $S^1$. Hence, $L$ is a hyperbolic
line of $U^2$ if and only if $L$ is the intersection of $U^2$ with either a
Euclidean line or a Euclidean circle which is orthogonal to
the real axis.

Another consequence of $\eta$ being conformal is that, since $B^n$ is a
conformal model of hyperbolic $n$-space, $U^n$ is also a conformal model.

\subsection{M\"obius transformations, isometries and linear fractional
transformations}\label{linfrac}

As in the conformal ball model, we may identify the M\"obius transformations
and isometries of $U^n$.  A M\"obius transformation of $U^n$ is a M\"obius
transformation of $\hat{E}^n$ which maps $U^n$ bijectively onto itself.

\begin{theorem} \label{Mob_isom_thm} Every M\"obius transformation of
$\hat{E}^n$ which leaves $U^n$ invariant restricts to an isometry of $U^n$, and every isometry of $U^n$ extends to a unique
M\"obius transformation of $\hat{E}^n$ which leaves $U^n$ invariant. \end{theorem}

\begin{proof} Follows immediately from Theorem~\ref{Mob_isom_B} and the
fact that $\eta$ is an isometry.\end{proof}

\begin{corollary} \label{Mob_isom_U^2} The groups $I(U^n)$ and $M(U^n)$ are
isomorphic. \end{corollary}

We now consider just the two-dimensional case, where $U^2$ is the complex upper
half-plane.  An important class of maps on $\Rsphere$ is the linear fractional
transformations. A linear fractional transformation  is a continuous map $\phi:
\Rsphere \rightarrow \Rsphere$ of the form \[\phi(z) = \frac{az + b}{cz + d},\]
where $a$, $b$, $c$ and  $d$ are in $\mathbb{C}$ and $ad - bc \not = 0$. Let
$LF(\Rsphere)$ be the set of all linear transformations of \Rsphere. Let
$B(\Rsphere)$ be the set of all bijections of \Rsphere.  Then
$B(\Rsphere)$ is a group under composition.  Note that a linear fractional
transformation is a bijection, so $\lfc \subseteq B(\Rsphere)$.

We will need the definitions of several groups of matrices. First,  $GL(n,\C)$
is the group of invertible $n \times n$ complex matrices.  Then, $GL(n,\C)$
contains the  subgroup $SL(n,\C)$, which is the group of $n \times n$ complex
matrices which have determinant 1.  The groups $SL(n,\R)$ and $SL(n,\Z)$ are
defined similarly.  The group $\sonr$ is the group of real orthogonal matrices
with determinant 1.  So $\sonr$ is a subgroup of $\slnr$.

We discuss the relationships between $LF(\Rsphere)$, M\"obius transformations,
groups of matrices and the isometries of $U^2$.   Let $\Xi:
GL(2,\mathbb{C})\rightarrow LF(\Rsphere)$ be the map \[ \left(\Xi \startm a & b
\\ c & d \finishm\right)(z) = \frac{az + b}{cz + d}. \]

\begin{lemma} \label{LFgroup} The set $LF(\Rsphere)$ is a group under
composition. \end{lemma}

\begin{proof} The map $\Xi$ is into the group $B(\Rsphere)$.   For $g, h \in
GL(2,\mathbb{C})$, it can be checked that $\Xi(gh)$, the image of the matrix
product $gh$, is the same map as the composition $\Xi(g)\Xi(h)$.  Thus, $\Xi$
is a group homomorphism.  Therefore, the image of $\Xi$ is a subgroup of
$B(\Rsphere)$.  Since $\Xi$ maps onto $LF(\Rsphere)$, we conclude that
$LF(\Rsphere)$ is a group. \end{proof}

Lemma~\ref{LFgroup} permits us to define an action of the group
$GL(2,\mathbb{C})$ on $\Rsphere$.  Let $g$ be the matrix
$\startm a & b \\ c & d \finishm \in GL(2,\C)$. The action is defined by setting
$gz=(\Xi(g))(z)$, that is, \begin{equation}\label{gltwocaction} gz = \startm a
& b \\ c & d \finishm z = \frac{az + b}{cz + d}. \end{equation}

We now define some quotient groups. Let $\C^*$ be $\C \setminus \{ 0 \}$.  The
group $PGL(n,\mathbb{C})$ is the quotient of the group $GL(n,\mathbb{C})$ by
the normal subgroup $\{\lambda I: \lambda \in \C^* \}$.  The group $PSL(n,
\mathbb{C})$ is then the quotient \[SL(n,\mathbb{C})/(SL(n,\C) \cap \{\lambda
I: \lambda\in \C^* \}).\]The groups $PSL(n,\R)$ and $PSL(n,\Z)$ are defined
similarly.  Then, if $n$ is even, $PSL(n,\R)$ is $\slnr/\{\pm I \}$, and
$PSL(n,\Z)$ is $\slnz/\{\pm I \}$. If $n$ is an odd integer,  we have
$PSL(n,\R) = SL(n,\R)$ and $PSL(n,\Z) = SL(n,\Z)$. The fact that the groups
$PGL(2,\mathbb{C})$ are $PSL(2,\mathbb{C})$ isomorphic is used to prove the
following corollary.

\begin{corollary} \label{LFisPSL} The group $LF(\Rsphere)$ is isomorphic to the
group $PSL(2,\mathbb{C})$. \end{corollary}

\begin{proof} The homomorphism $\Xi : GL(2,\mathbb{C}) \rightarrow
LF(\Rsphere)$ is onto, so \[ LF(\Rsphere) \cong GL(2,\mathbb{C}) / \ker \Xi. \]
Suppose $g = \startm a & b \\ c & d \finishm$ is in the kernel of $\Xi$.  Then,
for all $z \in \Rsphere$, \[ \frac{az + b}{cz + d} = z. \] Rearranging, we
obtain the equation $cz^2 + (d-a)z + b = 0$.  Since this equation holds for all
$z$, we have $c = 0$, $d = a$ and $b = 0$.  Hence, the kernel of $\Xi$ is the
subgroup $\{ \lambda I : \lambda \in \C^* \}$. Therefore $LF(\Rsphere) \cong
PSL(2,\mathbb{C})$. \end{proof}

\begin{propn}\label{LFisMob} Every linear fractional transformation of
\Rsphere\ is a M\"obius transformation of \Rsphere. \end{propn}

\begin{proof} We first show that the group $SL(2,\mathbb{C})$ is generated by
matrices of the form \[ u(m)=\startm 1 & m \\ 0 & 1 \finishm
\hspace{5mm}\mbox{and}\hspace{5mm} v = \matrixv, \] where $m$ is a complex
number.  For complex numbers $m$, $n$ and $p$, consider the product \[
u(m)\,v\,u(n)\,v\,u(p)  =  \startm mn - 1 & mnp - p - m \\ n & np - 1 \finishm
\in SL(2,\mathbb{C}). \] Let $g = \startm a & b \\ c & d \finishm \in
SL(2,\mathbb{C})$.  Suppose first that $c \not = 0$.  Set $n = c$, and then
choose $m$ and $p$ so that $ mn - 1 = a \hspace{5mm} \mbox{and}\hspace{5mm} np
- 1 = d$.  Since both $g$ and the product $u(m)\,v\,u(n)\,v\,u(p)$  have
determinant 1, we must have $b = mnp - p - m$.  Thus $g =
u(m)\,v\,u(n)\,v\,u(p)$. Now suppose that $c = 0$.  Since $\det(g) = 1$, $a
\not = 0$.  We have \[ v^{-1}\startm a & b \\ c & d \finishm  =  \matrixv
\startm a & b \\ c & d \finishm  =  \startm -c & -d \\ a & b \finishm, \] which
may be written $u(m)\,v\,u(n)\,v\,u(p)$ for some complex numbers $m$, $n$ and $p$,
since $a \not = 0$.  Thus $g = v\,u(m)\,v\,u(n)\,v\,u(p)$. Therefore
$SL(2,\mathbb{C})$ may be generated by the matrix $v$, and matrices of the form
$u(m)$, where $m \in \mathbb{C}$.

Let $\phi_{u(m)} = \Xi(u(m))$ and $\phi_v = \Xi(v)$.  Then $\phi_{u(m)}(z) = z
+ m$ and $\phi_v(z) = -1/z$.  We have $v^2 = -I$, and, from
Corollary~\ref{LFisPSL}, $LF(\Rsphere) \cong PSL(2,\mathbb{C})$. Hence
$LF(\Rsphere)$ is generated by the maps $\phi_{u(m)}$ and $\phi_v$.  So it now
suffices to prove that $\phi_{u(m)}$ and $\phi_v$ are M\"obius transformations.
First, $\phi_{u(m)}$ is a translation, and translations are isometries.
Therefore $\phi_{u(m)}$ is a M\"obius transformation. For $\phi_v$, let
$\sigma$ be the reflection in the unit circle $S^1$, and let $\tau$ be the
reflection in the imaginary axis, that is, $\tau(z) = -\overline{z}$.  Then \[
\tau\sigma(z) = \tau\left(\frac{z}{|z|^2}\right) =
\tau\left(\frac{1}{\overline{z}}\right) = -\frac{1}{z}, \] hence $\phi_v =
\tau\sigma$.  Thus $\phi_v$ is a M\"obius transformation. \end{proof}

\begin{propn} \label{MobequalsLF}Let $\rho:\Rsphere \rightarrow \Rsphere$ be
complex conjugation, that is, $\rho(z) = \overline{z}$.  Then $M(\Rsphere) =
LF(\Rsphere)\cup LF(\Rsphere)\rho$.  \end{propn}

\begin{proof} Since $\rho$ is a reflection (in the real axis), and
Proposition~\ref{LFisMob} shows that $LF(\Rsphere) \subseteq M(\Rsphere)$, we
have the inclusion \[ LF(\Rsphere)\cup LF(\Rsphere)\rho \subseteq M(\Rsphere)
\] immediately.

For the other inclusion, let $\phi$ be a M\"obius transformation of
$\Rsphere$.  Since $\phi$ is a composition of reflections in spheres of
$\Rsphere$, we first consider the cases where $\phi$ is a single reflection.
Suppose that $\phi$ is the reflection in the circle $S(a,r)$.  Then
\[ \phi(z)  =  a + r^2\frac{z-a}{|z-a|^2} =  a +
\frac{r^2}{\overline{z-a}} =  \frac{a\overline{z} + r^2 - a\overline{a}
}{\overline{z} - \overline{a}}  =  \psi\rho(z), \] where $\psi$ is
the map given by \[ \psi(z) = \frac{az + (r^2 - |a|^2)}{z - \overline{a}}. \]
Since \[ a(-\overline{a}) - (r^2 - |a|^2)(1) = -|a|^2 - r^2 + |a|^2 = -r^2
\not = 0, \] the map $\psi$ is a linear fractional transformation.  Hence, $\phi \in
LF(\Rsphere)\rho$.

Now suppose that $\phi$ is the reflection in the line $P(u,t)$. Let $\phi_0$ be
the reflection in the line through the origin $P(u,0)$.  As $P(u,t)$ is the
image of $P(u,0)$ under translation by $tu$, we have
\begin{equation}\label{LFM1} \phi(z) = \phi_0(z - tu) + tu. \end{equation}
Next, the line $P(u,0)$ is the image of the real axis under rotation through an
angle of $\theta + \frac{\pi}{2}$, where $u = e^{i\theta}$.  Since $\rho$ is
reflection in the real axis, \begin{equation}\label{LFM2} \phi_0(z) =
e^{i(\theta + \frac{\pi}{2})}\rho(e^{-i(\theta + \frac{\pi}{2})}z) =
iuiu\overline{z} = -u^2\overline{z}. \end{equation} Combining~(\ref{LFM1})
and~(\ref{LFM2}), we obtain
\[\phi(z) = -u^2\overline{z} + tu^2\overline{u} + tu = \psi\rho(z),\] where
$\psi$ is the linear fractional transformation \[ \psi(z) = \frac{-u^2z +
tu^2\overline{u} + tu }{0z + 1}. \] Thus $\phi \in LF(\Rsphere)\rho$.

We now consider the effect of composing reflections. Let $\phi$ and $\phi'$ be
in $LF(\Rsphere)$.  By expanding out, it can be shown that the composition
$\phi\rho\phi'\rho$ is in $LF(\Rsphere)$.  The composition $\phi\phi'$ is a
linear fractional transformation, as $LF(\Rsphere)$ is a group.
Also, both of the compositions $\phi\rho\phi'$ and $\phi\phi'\rho$ are in
$LF(\Rsphere)\rho$.  Thus every M\"obius transformation is either in
$LF(\Rsphere)$, or in $LF(\Rsphere)\rho$, depending on whether it contains an
even or odd number of reflections respectively. Hence $M(\Rsphere) \subseteq
LF(\Rsphere)\cup LF(\Rsphere)\rho$.  \end{proof}

The M\"obius transformations belonging to $LF(\Rsphere)$ are usually called
orientation preserving.

\begin{lemma} \label{formofphi} A linear fractional transformation $\phi$ of
\Rsphere\ leaves $U^2$ invariant if and only if there exist real numbers $a$,
$b$, $c$, $d$ with $ad - bc = 1$ such that \[ \phi(z) = \frac{az + b}{cz + d}.
\] \end{lemma}

\begin{proof} Suppose there exist real numbers $a$, $b$, $c$, $d$ with $ad - bc
= 1$ such that $\phi(z) = (az + b)/(cz + d)$. Recall that $\phi$ is a bijection.
Let $z \in U^2$.  Then \[\im(\phi(z)) =  \frac{\im(z)}{|cz +
d|^2} > 0.\]  Thus $\phi$ maps into $U^2$. To show that $\phi$ is onto $U^2$,
let $w$ be any point of $U^2$. Then $\phi(z)=w$ if and only if \[ z =
\frac{b-dw}{cw-a}. \] Here, $\im(z) = \im(w)/|cw - a|^2 >
0$, so $\phi$ is onto $U^2$.  Thus $\phi$ leaves $U^2$ invariant.

Conversely, suppose $\phi$ leaves $U^2$ invariant.  Since $\phi$ is a
homeomorphism, $\phi$ maps the boundary of $U^2$, which is $\hat{E}^1 = \R \cup
\{ \infty \}$, bijectively onto itself. We know that for some complex numbers
$a$, $b$, $c$ and $d$ with $ad-bc \not = 0$,  \[ \phi(z) = \frac{az + b}{cz +
d}. \]  We consider two cases, the first being if $b = 0$.  Since $ad-bc \not =
0$, it follows that $a \not = 0$. Then, we may divide $a$, $b$, $c$ and $d$
by $a$ without changing the function $\phi$, so we may assume that $a = 1$. We
have $\phi(1) = 1/(c+d)$ and $\phi(-1) = -1/(-c + d)$ in
$\hat{E}^1$, so $c+d$ and $-c+d$ are real. Thus $c$ and $d$ are real.

The second case is when $b \not = 0$.  By dividing each of $a$, $b$, $c$ and
$d$ by $b$, we may assume that $b = 1$.  Then $\phi(0) = 1/d \in \hat{E}^1$, so
$d \in \R$. If $a = 0$ then $c \not = 0$, and $\phi(1)= 1/(c+d) \in
\hat{E}^1$, so $c + d \in \R$, implying $c\in\R$. If $a \not = 0$, we have that
$\phi(-1/a) = 0$, hence $a \in \R$. Then, \[ \phi(1) = \frac{a +
1}{c+d} \in \hat{E}^1,  \] implying $c \in \R$.

Since $\phi$ maps the upper half-plane onto itself, $\im(\phi(i)) > 0$.  We
find that  $\im(\phi(i)) = (ad-bc)/(c^2 + d^2)$.  Therefore, the determinant
$ad - bc$ must be positive. Let $A = a/\sqrt{ad-bc}$, $B~=~b/\sqrt{ad-bc}$, $C
= c/\sqrt{ad-bc}$ and $D = d/\sqrt{ad-bc}$.  Then $A$, $B$, $C$ and $D$ are
real, $AD - BC = 1$, and $\phi(z) = (Az + B)/(Cz + D)$. \end{proof}

\begin{theorem} Let $-\rho$ be the map $(-\rho)(z) = -\overline{z}$.  Then \[
I(U^2) = \Xi(SL(2,\R)) \cup \Xi(SL(2,\R))(-\rho). \] \end{theorem}

\begin{proof} Let $\phi$ be in $\Xi(SL(2,\R))$. By Proposition~\ref{LFisMob},
$\phi$ is a M\"obius transformation of $\Rsphere$.  By
Proposition~\ref{formofphi}, $\phi$ leaves $U^2$ invariant, so $\phi$ is a
M\"obius transformation of $U^2$. Thus, by Theorem~\ref{Mob_isom_thm}, $\phi$
restricts to an isometry of $U^2$.  Since $-\rho$ is a reflection, which
moreover leaves $U^2$ invariant, $\phi(-\rho)$ is also a M\"obius transformation
of $U^2$. Hence, the composition $\phi(-\rho)$ is an isometry of $U^2$ as well.

For the other inclusion, let $\phi$ be an isometry of $U^2$. By
Theorem~\ref{Mob_isom_thm}, $\phi$ extends to a unique M\"obius transformation
of $U^2$.  So by Proposition~\ref{MobequalsLF}, either $\phi \in LF(\Rsphere)$, or
$\phi \in LF(\Rsphere)\rho$.

If $\phi \in LF(\Rsphere)$, then by Lemma~\ref{formofphi}, since $\phi$ leaves
$U^2$ invariant, $\phi$ is the image of some matrix in \sltwor\ under the map
$\Xi$. If $\phi \in LF(\Rsphere)\rho$, let $\tau$ be the linear fractional
transformation $\tau(z) = (-z + 0)/(0z + 1) = -z$. Then $(-\rho)(z) =
\tau\rho(z)$, and so $LF(\Rsphere)\rho = LF(\Rsphere)\tau\rho =
LF(\Rsphere)(-\rho)$.  Now, the map $-\rho$ leaves $U^2$ invariant, and $\phi$
leaves $U^2$ invariant, so $\phi$ has the form \[ \phi(z) =
\frac{a(-\overline{z}) + b}{c(-\overline{z}) + d}, \] where $a$, $b$, $c$ and
$d$ are real and $ad-bc = 1$.  Hence $\phi \in \Xi(\sltwor)(-\rho)$.
\end{proof}

\subsection{Topological groups and discrete subgroups}

A topological group is a group $G$ which is also a topological space, such
that the multiplication $(g,h) \mapsto gh$ and inversion $g \mapsto g^{-1}$ are
continuous functions in $G$. Many of the groups discussed in this chapter are
topological groups.  Groups of matrices such as $GL(n,\mathbb{C})$ and
$SL(n,\R)$ are topological groups, and their topology is the metric topology
induced by the distance function \[ d(A,B) = |A - B|, \] where the right-hand
side is the matrix norm defined by \[ |A| = \left( \sum_{i,j = 1}^n
|a_{ij}|^2\right)^{1/2}. \] Then,  quotient groups such as $P\sltwor$ are
topological groups as well, equipped with the quotient topology.  The group of
isometries of $B^n$ or $U^n$, and the group of linear fractional
transformations of $\Rsphere$, are also topological groups. They are
topologised with the subspace topology inherited from the spaces $C(B^n)$,
$C(U^n)$ or $C(\Rsphere)$, respectively, of continuous maps with the
compact-open topology.

An isomorphism of topological groups $G$ and $H$ is a map $\phi:G \rightarrow
H$ which is both a group isomorphism and a homeomorphism.  The isomorphism of
$LF(\Rsphere)$ and $PSL(2,\mathbb{C})$ established in Proposition~\ref{LFisPSL}
is an isomorphism of topological groups.

Let $\Gamma$ be a subgroup of a topological group $G$.  Then  $\Gamma$ is said
to be a
discrete subgroup of $G$ if for all $\gamma \in \Gamma$ there exists an open
neighbourhood $\Omega$ of $\gamma$ in $G$ such that $\Omega \cap \Gamma = \{ \gamma \}$.
Recall that the groups \sltwor\ and $P\sltwor$ act on $U^2$ by isometries.  We
now show that $SL(n,\mathbb{Z})$ is a discrete subgroup of $SL(n,\R)$, and that
$PSL(n,\mathbb{Z})$ is a discrete subgroup of $PSL(n,\R)$.

\begin{lemma} \label{discretelemma} Let $G$ be a topological group which is
also a metric space, and let $\Gamma$ be a subgroup of $G$. Suppose every
convergent subsequence in $\Gamma$ is eventually constant.  Then $\Gamma$ is
discrete. \end{lemma}

\begin{proof} Suppose that every convergent subsequence in $\Gamma$ is
eventually constant, but that $\Gamma$ is not discrete.  Then there is a point
$\gamma \in \Gamma$ such that for all $n \geq 1$, the ball
$B(\gamma,1/n)$ contains a point $\gamma_n$ of $\Gamma$ which is
different from $\gamma$.  Then $\gamma_n \rightarrow \gamma$, but the sequence
$\{\gamma_n\}_{n=1}^\infty$ is not eventually constant, which is a contradiction.  We
conclude that $\Gamma$ must be discrete. \end{proof}

\begin{propn} A subgroup $\Gamma$ of $SL(n,\R)$ is discrete if for all $r > 0$
the set $\{ A \in \Gamma : |A| < r\}$ is finite. \end{propn}

\begin{proof} Suppose that $\{ A \in \Gamma : |A| \leq r \}$ is finite for each
$r > 0$.  Let $B_j \rightarrow B$ in $\Gamma$.  As the norm function is
continuous, $|B_j| \rightarrow |B|$.  Thus, there exists an integer $j_0$ such
that for all $j \geq j_0$,   $| |B_j| - |B| | \leq 1$.  If $| |B_j| - |B| |
\leq 1$ then $|B_j| \leq 1 + |B|$.  Fix $r = 1 + |B|$, then the set $\{ A \in
\Gamma : |A| \leq 1 + |B| \}$ is finite.  Hence, the sequence $\{ B_j \}$ is
eventually constant.  Since $SL(n,\R)$ is a metric space, we may apply
Lemma~\ref{discretelemma} to conclude that $\Gamma$ is discrete. \end{proof}

\begin{corollary} The group $SL(n,\mathbb{Z})$ is a discrete subgroup of
$SL(n,\R)$. \end{corollary}

\begin{corollary} The group $PSL(n,\mathbb{Z})$ is a discrete subgroup of
$PSL(n,\R)$. \end{corollary}

We now discuss group actions in greater detail.  Let $G$ be a group acting on a
set $X$ and let $x$ be an element of $X$.  The subset of $X$ given by $Gx = \{
gx : g \in G \}$ is called the $G$-orbit of $x$, or the orbit of $x$
under $G$.   The $G$-orbits partition $X$.

\begin{lemma}\label{discretecompact} If $\Gamma$ is a discrete group of
isometries of a metric space $X$, and $K$ is compact in
$X$, then $K$ contains only finitely many points of each orbit
$\Gamma x$.\end{lemma}

\begin{proof} Let $K$ be compact in $X$.  Suppose $K$  contains infinitely many
points of an orbit $\Gamma x$.  Then $K$ contains a convergent subsequence $\{
\gamma_n x \}$ of $\Gamma x$, and there are infinitely many distinct points
$\gamma_n x$. Since $\Gamma$ is topologised with the compact open topology,
$\gamma_n x \rightarrow \gamma x$ in $K$ if and only if $\gamma_n \rightarrow
\gamma$ in $\Gamma$.  As $\Gamma$ is discrete, the sequence $\{\gamma_n\}$ is
eventually constant, so the sequence $\{ \gamma_n x \}$ is eventually constant,
that is, it contains only finitely many distinct elements.  This is a
contradiction. \end{proof}

The stabiliser subgroup $G_x$ of $x$ in $G$ is the subgroup $\{ g \in G : gx =
x \}$.  As a useful application of the ideas of orbits and stabiliser
subgroups, we have the following decomposition of \sltwoz.

\begin{lemma} \label{sl2decomp} Let $\gamma \in \sltwoz$.  Then $\gamma$ may be
expressed in the form $k_1ak_2$, where $k_1$ and $k_2$ are in \sotwor,
and $a$ is in \sltwor\ with the form $\startm s & 0 \\ 0 & s^{-1} \finishm$
for some $s \geq 1$. \end{lemma}

\begin{proof} We first show that \sotwor\ is the stabiliser subgroup of $z = i$
in \sltwor.  Let $k \in \sotwor$.  Then $k$ has the form \[ k_\theta = \startm
\cos \theta & -\sin \theta \\ \sin \theta & \cos \theta \finishm \] for some
$\theta \in [0,2\pi)$.  It can be calculated that  $k_\theta$  fixes the point
$i$.  Now suppose that $k$ is the matrix $\startm a & b \\ c & d \finishm$ in
$\sltwor$, and that $ki = i$. Then $ai + b  =  di - c$, hence $a = d$ and $b =
-c$.  Since $ad - bc = 1$, we have $a^2 + c^2 = 1$, and so the matrix \[ k =
\startm a & b \\ c & d \finishm = \startm a & -c \\ c & a \finishm \] is an
element of \sotwor.  Therefore, if $k \in \sltwor$, $ki = i$ if and only if $k$
is in $\sotwor$.

We next show that for each $r \geq 0$, the orbit of the point $e^ri$ under
the group \sotwor\ is the Euclidean circle with Cartesian equation \[ x^2 + (y
- \cosh r)^2 = \sinh^2 r. \] We will denote this circle by $C_r$. When $r = 0$,
the point $e^ri$ is just $i$.  Since \sotwor\ is the stabiliser subgroup of
$i$, the orbit is then the single point $i$. For $r > 0$, a tedious calculation
shows that for any $k_\theta \in \sotwor$, the point $k_\theta e^ri$ lies in
the circle $C_r$. We now argue that the orbit is onto the circle $C_r$.   The
action of $\sotwor$  on the point $e^{-r}i$ may be regarded as a continuous
function of $\theta$.  Thus, the image of the interval $[0,\pi/2]$ is a
connected subset of $C_r$.  When $\theta = 0$, we get $k_\theta e^ri =
e^ri$.  When $\theta = \pi/2$, we get $k_\theta e^ri = e^{-r}i$.  For $0 <
\theta < \pi/2$,  the real part of $k_\theta e^{r}i$ is negative.  Hence, the
image of $[0,\pi/2]$ is those points of the circle $C_r$ with non-positive real
part.  Similarly, the image of the interval $[\pi/2,\pi]$ is those points of
$C_r$ with non-negative real part.  Therefore, the orbit of the point $e^ri$
under \sotwor\ is the whole circle $C_r$.

We now apply these results to complete the proof. Let $p$ be the point $\gamma
i$. Then, since $p$ lies in the orbit $C_r$ for some $r \geq 0$, there exists a
$k \in \sotwor$ such that $kp = e^ri$. Let $\alpha$ be the matrix $\startm
e^{-r/2} & 0 \\ 0 & e^{r/2} \finishm$.  Then $\alpha e^{r}i = i$. Let $k' =
\alpha k\gamma$.  Then $k'i = i$, so $k' \in \sotwor$.  Making $\gamma$ the subject and then relabelling,
we obtain $\gamma = k^{-1}\alpha^{-1}k' = k_1 a k_2$. \end{proof}

\subsection{Tesselation by fundamental domains}\label{tesselation}

In this subsection, we provide general definitions of concepts needed to describe
tesselations of metric spaces which are induced by the action of discrete
groups of isometries.  Most of the proofs are for just the
upper half-plane.  However, with a few exceptions,
the proofs for the general case are not much different.  See Chapter 6
of~\cite{rat1:fhm} for a full treatment of the general case.

Let $(X,d)$ be a metric space and let $\Gamma$ be a non-trivial group of
isometries of $X$. A subset $R$ of $X$ is said to be a fundamental region for
the group $\Gamma$ if \begin{enumerate} \item the set $R$ is open in $X$, \item
the members of $\{ \gamma R : \gamma \in \Gamma \}$ are mutually disjoint, and
\item $X = \cup \{ \gamma \overline{R} : \gamma \in \Gamma \}$. \end{enumerate}
A fundamental domain is a connected fundamental region. Now, if $\mathcal{S}$
is a collection of subsets of $X$, then $\mathcal{S}$ is described as locally
finite if for each point $x \in X$, there is an open neighbourhood of $x$ which
meets only finitely many members of $\mathcal{S}$.   We say that a
fundamental domain $D$ is a locally finite fundamental domain if the collection
$\{ \gamma \overline{D} : \gamma \in \Gamma \}$ is a locally finite collection of sets.

A subset $F$ of $X$ is a fundamental set for the group $\Gamma$ if $F$ contains
exactly one point from each $\Gamma$-orbit in $X$.  Note that a fundamental
domain is not a fundamental set.  As the orbits of $\Gamma$ on the space $X$
partition $X$, if $F$ is a fundamental set, then we have $X = \cup \{  \gamma
F: \gamma \in \Gamma\}$.

We consider the metric space $U^2$. One feature of the upper half-plane is that
any isometry of $U^2$ which fixes every point of a non-empty open subset of
$U^2$ is the identity map. This fact is used to prove the following lemma.

\begin{lemma}\label{fundsets} An open subset $R$ of the metric space $U^2$ is a
fundamental region for a group $\Gamma$ of isometries of $U^2$ if there is a
fundamental set $F$ for $\Gamma$ such that $R \subseteq F \subseteq \overline{R}$.
\end{lemma}

\begin{proof} Suppose that there exists a fundamental set $F$ for the group
$\Gamma$ such that $R \subseteq F \subseteq \overline{R}$.  We first show that the
elements of the set $\{ \gamma R : \gamma \in \Gamma \}$ are mutually
disjoint.  Suppose that $\gamma$ and $\delta$ are elements of $\Gamma$ such
that $\gamma R \cap \delta R$ is non-empty.  Then, there are points $z$ and $w$
in $R$ such that $\gamma z = \delta w$.  Hence, $\delta^{-1}\gamma z = w$.  As
$R \subseteq F$, the points $z$ and $w$ are in $F$.  Since $F$ contains only one point from
the orbit $\Gamma z$,  we have $\delta^{-1}\gamma z = z$.  Thus the isometry
$\delta^{-1}\gamma$ fixes each point of the open set $R \cap \gamma^{-1}\delta
R$.  Therefore, $\delta^{-1}\gamma$ is the identity, and so $\delta = \gamma$.

Next, since $F \subseteq\overline{R}$, we have \[ X = \bigcup_{\gamma \in \Gamma}
\gamma F = \bigcup_{\gamma \in \Gamma} \gamma \overline{R}. \] Therefore $R$ is a
fundamental region for $\Gamma$.  \end{proof}

We now describe a general method for constructing a fundamental domain. Suppose
there is a point $p_0$ in $X$ which has a trivial stabiliser subgroup
$\Gamma_{p_0}$.  For each $\gamma \not = 1$ in $\Gamma$, we define the open
half-space $H_\gamma(p_0)$ by \[ H_\gamma(p_0) = \{ x \in X : d(x,p_0) <
d(x,\gamma p_0) \}. \] Then, the Dirichlet domain $D(p_0)$ with centre $p_0$ is
defined to be the intersection of these half-spaces: \[ D(p_0) = \bigcap \{
H_\gamma(p_0) : \gamma \not = 1 \mbox{ in } \Gamma \}. \]

If $\Gamma = P\sltwoz$, we may regard $\Gamma$ as a subgroup of the group of
isometries of $U^2$.   We will, abusing notation, write $\gamma$ for both a
matrix $\gamma \in\sltwoz$, and the coset of this matrix in $P\sltwoz$.  There
are uncountably many points of $U^2$ which have trivial stabilisers in
$\Gamma$.  This is because, if \[ \startm a & b \\ c & d \finishm z = z, \]
then by the quadratic formula $z = ((a-d)\pm \sqrt{(d-a)^2 + 4bc})/2c$, and
$\mathbb{Z}$ is countable.  Some points which do have trivial stabilisers are
all the points $ti$, where $t >  1$.   We now show that in the upper
half-plane, a Dirichlet domain with centre at $p_0$ is a locally finite
fundamental domain.  To show that $D(p_0)$ is open, we will show that its
complement is closed in $U^2$.  In general, a union of closed sets need not be
closed. However, if a collection $\mathcal{S}$ of closed subsets is locally
finite, then the union of the members of $\mathcal{S}$ is closed.

\begin{propn} Let $\Gamma = P\sltwoz$ and let $p_0$ be a point of
$U^2$ such that $\Gamma_{p_0}$ is trivial. Then the Dirichlet domain
$D(p_0)$ is a locally finite fundamental domain for the action of $\Gamma$ on
$U^2$.
\end{propn}

\begin{proof} For each $\gamma \not = 1$ in $P\sltwoz$, let $K_\gamma = U^2
\setminus H_\gamma(p_0)$. Then $K_\gamma$ is closed.  We show that $\mathcal{K}
= \{ K_\gamma : \gamma \not = 1 \mbox{ in } \Gamma\}$ is a locally finite
collection of sets.  Let $z$ be any point of $U^2$, and let $r > 0$ be such
that $B(p_0,r)$ contains $z$.  We will show that $B(p_0,r)$ meets only finitely
many of the sets in $\mathcal{K}$. If $B(p_0,r)$ meets none of the sets in
$\mathcal{K}$ then we are done, so suppose $B(p_0,r)$ intersects the set
$K_\gamma$ in a point $w$.  Then we have \[ d_U(p_0,\gamma p_0)  \leq
d_U(p_0,w) + d_U(w,\gamma p_0) \leq d_U(p_0,w) + d_U(w,p_0)  < 2r. \] Hence,
the ball $B(p_0,2r)$ contains $\gamma p_0$. Now, the closure of this ball is
compact, and $P\sltwoz$ is discrete, so by Lemma~\ref{discretecompact},
$\overline{B(p_0,2r)}$ contains only finitely many points of the orbit $\Gamma
p_0$.  Thus $B(p_0,r) \subseteq \overline{B(p_0,2r)}$ meets only finitely many of the
sets $K_\gamma$.  Therefore $\mathcal{K}$ is a locally finite family of closed
sets in $U^2$.  This means that the union of the sets in $\mathcal{K}$ is
closed.  So $D(p_0)$, which is the complement of this union, is open.

We now construct a fundamental set $F$ and show that $D(p_0) \subseteq F
\subseteq \overline{D(p_0)}$.  From each orbit $\Gamma z$, we may select a
point which is the minimum distance from the point $p_0$ (if there is more than
one closest point, choose any).  Let $F$ be the set of these points. Then $F$
is a fundamental set for $\Gamma$.

Let $z$ be in $D(p_0)$ and $\gamma \not = 1$ in $\Gamma$.  Then, since $\gamma$ is an
isometry, \[ d_U(z,p_0) < d_U(z,\gamma p_0) = d(\gamma^{-1}z,p_0), \] so $z$ is the unique
closest point of the orbit $\Gamma z$ to $p_0$.  Hence $D(p_0) \subseteq F$.

To show $F \subseteq \overline{D(p_0)}$, let $z$ be in $F$.  If $z = p_0$ then
$z \in \overline{D(p_0)}$ immediately, so we assume that $z \not = p_0$.  Let
$\gamma \not = 1$ be in $\Gamma$.  If $d_U(z,p_0) > d_U(z,\gamma p_0)$,
then $d_U(z,p_0) > d(\gamma^{-1}z,p_0)$, which contradicts our construction of
$F$.  Therefore, $d_U(z,p_0) \leq d_U(z,\gamma p_0)$.  Let $[p_0,z]$ be the
geodesic segment in $U^2$ joining the points $p_0$ and $z$, and let $w$ be a
point in the interior of this segment.  Then \[ d_U(z,w) + d_U(w,p_0)=
d_U(z,p_0), \] so  \[ d_U(w,p_0)  = d_U(z,p_0) - d_U(z,w)  \leq d_U(z,\gamma
p_0) - d_U(z,w)  \leq d_U(w,\gamma p_0), \] with equality only if $d_U(z,p_0) =
d_U(z,\gamma p_0)$ and $d_U(z,\gamma p_0) = d_U(z,w) + d_U(w,\gamma p_0)$.
Suppose we have equality.  Then $[z,w]\cup[w,\gamma p_0]$ is the geodesic
segment $[z,\gamma p_0]$.  The segments $[z,p_0]$ and $[z,\gamma p_0]$ have the
same length, and both extend $[z,w]$, so $p_0 = \gamma p_0$.  This contradicts
the stabiliser subgroup $\Gamma_{p_0}$ being trivial.  Thus we have the strict
inequality \[ d_U(w,p_0) < d_U(w,\gamma p_0) \] for all $\gamma \not = 1$ in
$\Gamma$.  Hence, $w \in D(p_0)$.  This implies that the half-open segment
$[p_0,z)$ is contained in $D(p_0)$, and so $z \in \overline{D(p_0)}$.
Therefore,  $F \subseteq \overline{D(p_0)}$.  By Lemma~\ref{fundsets}, $D(p_0)$
is a fundamental region.

We have shown that, for any $w \in D(p_0)$, the geodesic segment $[p_0,w]$ is
in $D(p_0)$.  This means $D(p_0)$ is connected.  Therefore $D(p_0)$ is a
fundamental domain.

Finally we show that $D(p_0)$ is a locally finite fundamental domain.  Write
$D$ for $D(p_0)$.  Let $r > 0$, and suppose the ball $B(p_0,r)$ meets
$\gamma\overline{D}$ for some $\gamma \in \Gamma$.  Then there is a $z$ in $D$
such that $d_U(p_0, \gamma z) < r$.  Moreover, \begin{align*} d_U(p_0,\gamma
p_0) & \leq d_U(p_0, \gamma z) + d(\gamma z, \gamma p_0) \\ & < r + d_U(z,p_0)
\\ & \leq r + d_U(z,\gamma^{-1} p_0) \\ & = r + d_U(\gamma z, p_0) \\& <
2r.\end{align*}  Since $\Gamma$ is discrete, $d_U(p,\gamma p_0) < 2r$ for only
finitely many elements $\gamma$.  Hence, the ball $B(p_0,r)$ meets only
finitely many sets $\gamma \overline{D}$. \end{proof}

Let $T$ be the generalised hyperbolic triangle with vertices at
$\pm \frac{1}{2} + \frac{\sqrt{3}}{2}i$ and $\infty$, and let
$T^\circ$ be the interior of the triangle $T$.  This triangle is pictured
overleaf.

\begin{center}
\begin{pspicture}(0,-0.5)(16,10)
\psline(0,0)(16,0)
\psline(8,0)(8,10)
\psline(6,3.5)(6,10)
\psline(10,3.5)(10,10)
\psline[linestyle=dotted](2,3.5)(2,10)
\psline[linestyle=dotted](14,3.5)(14,10)
\uput[d](0,0){-2}
\uput[d](4,0){-1}
\uput[d](6,0){$-\frac{1}{2}$}
\uput[d](10,0){$\frac{1}{2}$}
\uput[d](8,0){0}
\uput[d](12,0){1}
\uput[d](16,0){2}
\uput[r](8,8){$\displaystyle{ti}$}
\psdot*(8,8)
\uput[r](8,1.95){$\displaystyle{\frac{i}{t}}$}
\psdot*(8,2)
\uput[r](4,8){$\displaystyle{ti-1}$}
\psdot*(4,8)
\uput[r](12,8){$\displaystyle{ti+1}$}
\psdot*(12,8)
\psarc(8,0){4}{60}{120}
\psarc[linestyle=dotted](0,0){4}{0}{90}
\psarc[linestyle=dotted](4,0){4}{0}{180}
\psarc[linestyle=dotted](8,0){4}{0}{180}
\psarc[linestyle=dotted](12,0){4}{0}{180}
\psarc[linestyle=dotted](16,0){4}{90}{180}

\end{pspicture}
\end{center}

\medskip

\begin{propn} Let $p_0 = ti$ for some $t > 1$.  The Dirichlet domain $D(p_0)$
 is equal to $T^\circ$.
\end{propn}

\begin{proof}\cite{bea1:gdg} Let $u = \matrixu$ and $v = \matrixv$.  Then $u$
and $v$ are in $P\sltwoz$, so \[D(p_0) \subseteq \left(H_u(p_0) \cap H_v(p_0)
\cap H_{u^{-1}}(p_0)\right).\]  We have $up_0 = ti+1$, $vp_0 = i/t$ and
$u^{-1}p_0 = ti - 1$.  Thus the geodesic $\re(z) = 1/2$ is the boundary of
$H_u(p_0)$, and $H_u(p_0) = \{ z \in U^2 : \re(z) < 1/2\}$. The boundary of
$H_{u^{-1}}(p_0)$ is the geodesic $\re(z) = -1/2$, and $H_{u^{-1}}(p_0) = \{ z
\in U^2 : \re(z) > -1/2\}$. For $H_v(p_0)$, all points in the intersection of
the circle $|z| = 1$ with $U^2$ are the same hyperbolic distance from $p_0$ as
from $vp_0$.  Thus, this circle, which is a geodesic since it is orthogonal to
the real axis, is the boundary of $H_v(p_0)$, and $H_v(p_0)$ is the set $\{ z
\in U^2 : |z| > 1 \}$.  The three boundary geodesics form the sides of a
generalised hyperbolic triangle with one vertex at infinity.  The finite
vertices are at $\pm \frac{1}{2} + \frac{\sqrt{3}}{2}i$. Hence, $H_u(p_0) \cap
H_v(p_0) \cap H_{u^{-1}}(p_0)$ is the interior of the triangle $T$. Thus
$D(p_0)\subseteq T^\circ$.

Suppose $D(p_0) \subsetneqq T^\circ$, and write $D$ for $D(p_0)$. Then there is a
point $z \in T^\circ$ such that $z \not \in D$. Since $D$ is a fundamental
domain, $U^2 = \{ \gamma\overline{D}: \gamma\in \Gamma \}$, hence there is a
$\gamma \in \Gamma$, $\gamma \not = 1$, such that $z \in \gamma\overline{D}$.
Now, the action of $\Gamma$ on $U^2$ is by homeomorphisms, and $D$ and
$T^\circ$ are both open sets, with $D \subsetneqq T^\circ$.  Thus $T^\circ$
must properly contain the boundary of $D$. So we may, in fact, find some
$\gamma \in \Gamma$ such that  $z \in \gamma D$. Let $w = \gamma^{-1}z \in D$.
Then, since $D \subseteq T^\circ$, $w$ is in $T^\circ$.  We show that this
implies a contradiction.

Let $\delta$ be $\gamma^{-1}$, and suppose that $\delta = \startm a & b \\ c &
d \finishm$.  Then, since $z$ is in $T^\circ$, $|z| > 1$ and $\re(z) < 1/2$, so
\begin{align*} |cz + d|^2 & = c^2|z|^2 + 2cd\re(z) + d^2 \\ & > c^2  - |cd| +
d^2 \\ & = (|c| - |d|)^2 + |cd|. \end{align*} This lower bound is a
non-negative integer.  It is zero if and only if $c = d = 0$; since  $ad - bc =
1$ this is impossible.  Therefore, $|cz + d|^2 > 1$. Thus, \[ \im(\delta z) =
\frac{\im(z)}{|cz +d|^2} < \im(z), \] that is, $\im(w) < \im(z)$. Now replace $z$
by $w$ and $\delta$ by $\gamma$ in the above argument.  We find that
$\im(\gamma w) < \im(w)$, hence $\im(z) < \im(w)$.  This is a contradiction.
Therefore, $D = T^\circ$.\end{proof}

We conclude this discussion of fundamental domains by showing how they are
related to tesselations, and using this relationship to find generators for the
groups $P\sltwoz$ and \sltwoz.  A convex polyhedron in a metric space $X$ is a
non-empty, closed, convex subset of $X$ with finitely many sides and a
non-empty interior. A tesselation of $X$ is a collection $\mathcal{P}$ of
convex polyhedra in $X$ such that \begin{enumerate} \item the interiors of the
polyhedra in $\mathcal{P}$ are mutually disjoint, and \item the union of the
polyhedra in $\mathcal{P}$ is equal to $X$. \end{enumerate} By comparing this
definition with that of a fundamental domain, it can be seen that, since
$T^\circ$ is a locally finite fundamental domain, $\mathcal{T} = \{\gamma T :
\gamma \in PSL(2,\mathbb{Z})\}$ is a locally finite tesselation of $U^2$.
Moreover, it is an exact tesselation, meaning that each side of a triangle
$\gamma T$ is a side of exactly two triangles $\gamma T$ and $\delta T$ in
$\mathcal{T}$.  An exact tesselation $\mathcal{P}$ is said to be connected if
for each $P, Q \in \mathcal{P}$ there is a finite sequence $P_0,P_1
\ldots, P_m$ in $\mathcal{P}$ such that $P=P_0$, $Q = P_m$, and for each $1\leq
i \leq m$, the polyhedra $P_{i-1}$ and $P_i$ have a common side.  The
tesselation $\mathcal{T}$ is pictured in Chapter 1 on page~\pageref{tess}.

\begin{lemma}\label{connected} The tesselation $\mathcal{T}$ of
$U^2$ is connected.
\end{lemma}

\begin{proof} Let $\delta T$ be a triangle in $\mathcal{T}$, and let $V$ be the
union of all the triangles $\gamma T$ where there exists a finite sequence
$\gamma_0T,\gamma_1T, \ldots, \gamma_mT$ such that $\delta T = \gamma_0T$,
$\gamma T = \gamma_mT$, and the triangles $\gamma_{i-1}T$ and $\gamma_iT$ share
a common side for $1 \leq i \leq m$.  Since $V$ contains $\delta T$, $V$ is
non-empty.  The tesselation $\mathcal{T}$ is a locally finite collection of closed
subsets of $U^2$. Hence, $V$ is closed.

We now show that the set $V$ is also open. Let $z$ be in $V$.  If $z$ is in the
interior of some triangle $\gamma T$, then since an interior is open, there is
an $\varepsilon > 0$ such that the ball $B(z,\varepsilon)$ is contained in the
interior, and so this ball is contained in $V$. The other possibility is that
$z$ lies on the side of some triangle in $V$. Choose $\varepsilon$ so that
$\overline{B(z,\varepsilon)}$ meets only the triangles in $\mathcal{T}$ which
have $z$ on one of their sides.  Let $S(z,\varepsilon)$ be the hyperbolic
circle centred at $z$ of radius $\varepsilon$.  Then $\mathcal{T}$ restricts to
an exact tesselation of $S(z,\varepsilon)$.  This restriction is a connected
tesselation, so any triangle in $\mathcal{T}$ which has $z$ on a side is in
$V$.  Thus $V$ contains $B(z,\varepsilon)$. We have shown that $V$ is both open
and closed, and is non-empty. Since $U^2$ is connected, this means $V = U^2$.
\end{proof}

\begin{propn}\label{gensides}
The group $\Gamma = P\sltwoz$ is generated by the set
\[
\Sigma = \{ \gamma \in \Gamma : T \cap \gamma T \mbox{ is a side of } T \}.
\]
\end{propn}

\begin{proof} Let $\gamma$ be any element of $\Gamma$.  Then, since
$\mathcal{T}$ is connected, there is a finite sequence of elements
$\gamma_0,\gamma_1, \ldots, \gamma_m$ of $\Gamma$ such that $T = \gamma_0T$,
$\gamma_mT=\gamma T$, and $\gamma_{i-1}T$ and $\gamma_iT$ share a common side
for $1 \leq i \leq m$.  Note that $\gamma_0 = I$ and $\gamma_m = g$.  Since
$\gamma_{i-1}^{-1}\gamma_{i-1}T = T$, the triangles $T$ and
$\gamma_{i-1}^{-1}\gamma_iT$ share a common side for each $i$.  We may assume
that for each $i$, $\gamma_{i-1} \not = \gamma_i$.  Then
$\gamma_{i-1}^{-1}\gamma_i \in \Sigma$.  The product
\[\gamma_0(\gamma_0^{-1}\gamma_1)(\gamma_1^{-1}\gamma_2)\cdots(\gamma_{m-1}^{-1}\gamma_m)
= \gamma\] is thus a product of the identity $\gamma_0$ and elements of
$\Sigma$, hence $\Sigma$ generates $\Gamma$. \end{proof}

\begin{corollary} The group $\sltwoz$ is generated by the matrices $u =
\matrixu$ and $v = \matrixv$. \end{corollary}


\chapter{Symmetric Spaces and Distance Functions}\label{sym}


\section{Introduction}

The upper half-plane $U^2$ is a metric space with distance function $d_U$.  The
group $P\sltwor$ acts on $U^2$ by isometries, and a tesselation of $U^2$ is
induced by the action of the discrete subgroup $P\sltwoz$. It is natural to
ask if, for $\gamma \in P\sltwoz$ and $z \in U^2$, there is any relationship
between the hyperbolic distance $d_U(z,\gamma z)$, and the number of tiles of
this tesselation which are crossed by the geodesic segment $[z, \gamma z]$.
Now, the tiles in this tesselation of $U^2$ are the images under elements of
$P\sltwoz$ of the triangle $T$ with vertices at $\pm \frac{1}{2} +
\frac{\sqrt{3}}{2}i$ and $\infty$.  Moreover, the group $P\sltwoz$ may be
generated by those elements $u$ and $v$ which map $T$ to a triangle adjacent to
$T$. Thus, it seems there is some relationship between the number of tiles
crossed by the geodesic segment $[z, \gamma z]$, and the expression for
$\gamma$ in terms of $u$ and $v$.

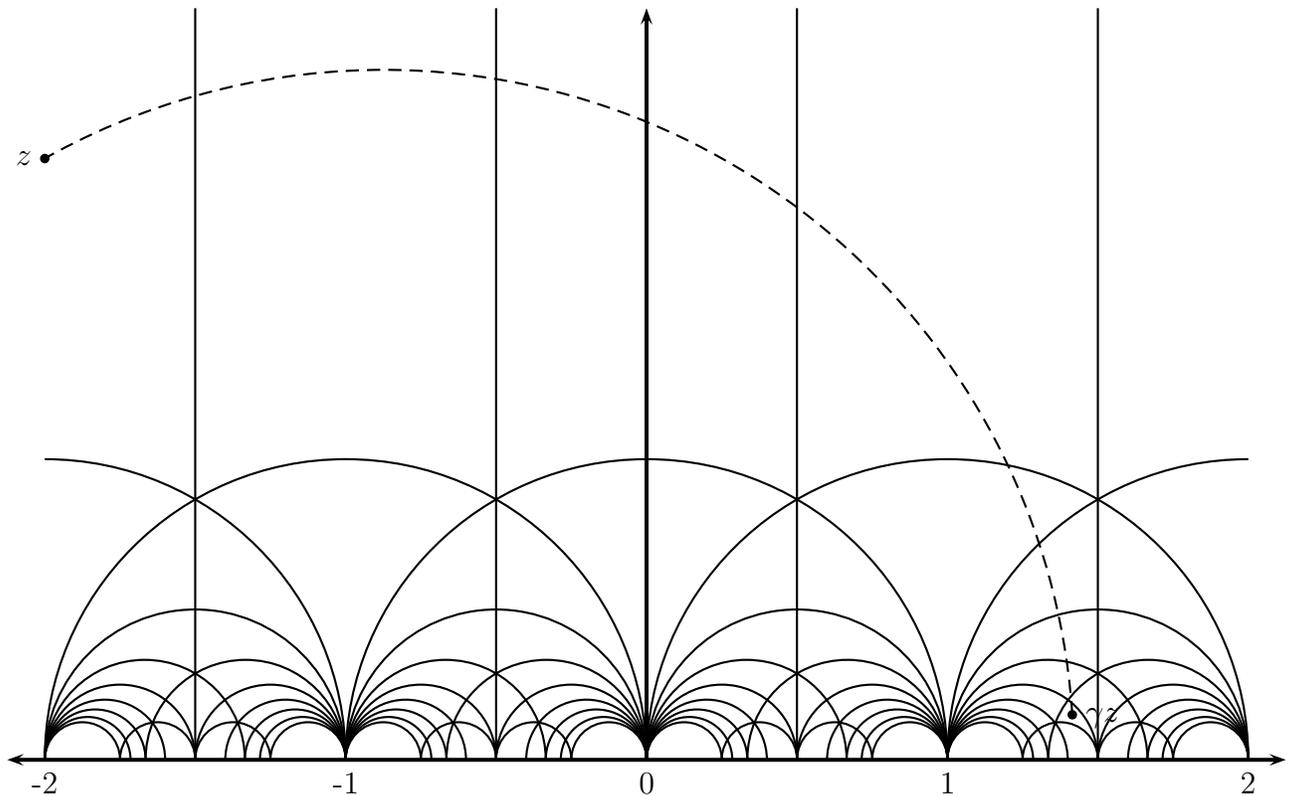
\begin{center}
\begin{pspicture}(0,-0.5)(16,10)
\psline[linewidth=1.5pt]{<->}(-0.5,0)(16.5,0)
\psline[linewidth=1.5pt]{->}(8,0)(8,10)
\psline(2,0)(2,10)
\psline(6,0)(6,10)
\psline(10,0)(10,10)
\psline(14,0)(14,10)
\uput[d](0,0){-2}
\uput[d](4,0){-1}
\uput[d](8,0){0}
\uput[d](12,0){1}
\uput[d](16,0){2}
\psarc(0,0){4}{0}{90}
\psarc(4,0){4}{0}{180}
\psarc(8,0){4}{0}{180}
\psarc(12,0){4}{0}{180}
\psarc(16,0){4}{90}{180}
\psarc(2,0){2}{0}{180}
\psarc(6,0){2}{0}{180}
\psarc(10,0){2}{0}{180}
\psarc(14,0){2}{0}{180}
\psarc(1,0){1}{0}{180}
\psarc(3,0){1}{0}{180}
\psarc(5,0){1}{0}{180}
\psarc(7,0){1}{0}{180}
\psarc(9,0){1}{0}{180}
\psarc(11,0){1}{0}{180}
\psarc(13,0){1}{0}{180}
\psarc(15,0){1}{0}{180}
\psarc(1.33,0){1.33}{0}{180}
\psarc(5.33,0){1.33}{0}{180}
\psarc(9.33,0){1.33}{0}{180}
\psarc(13.33,0){1.33}{0}{180}
\psarc(2.67,0){1.33}{0}{180}
\psarc(6.67,0){1.33}{0}{180}
\psarc(10.67,0){1.33}{0}{180}
\psarc(14.67,0){1.33}{0}{180}
\psarc(0.8,0){0.8}{0}{180}
\psarc(4.8,0){0.8}{0}{180}
\psarc(8.8,0){0.8}{0}{180}
\psarc(12.8,0){0.8}{0}{180}
\psarc(3.2,0){0.8}{0}{180}
\psarc(7.2,0){0.8}{0}{180}
\psarc(11.2,0){0.8}{0}{180}
\psarc(15.2,0){0.8}{0}{180}
\psarc(0.67,0){0.67}{0}{180}
\psarc(4.67,0){0.67}{0}{180}
\psarc(8.67,0){0.67}{0}{180}
\psarc(12.67,0){0.67}{0}{180}
\psarc(3.33,0){0.67}{0}{180}
\psarc(7.33,0){0.67}{0}{180}
\psarc(11.33,0){0.67}{0}{180}
\psarc(15.33,0){0.67}{0}{180}
\psarc(0.57,0){0.57}{0}{180}
\psarc(4.57,0){0.57}{0}{180}
\psarc(8.57,0){0.57}{0}{180}
\psarc(12.57,0){0.57}{0}{180}
\psarc(3.43,0){0.57}{0}{180}
\psarc(7.43,0){0.57}{0}{180}
\psarc(11.43,0){0.57}{0}{180}
\psarc(15.43,0){0.57}{0}{180}
\psarc(0.5,0){0.5}{0}{180}
\psarc(1.5,0){0.5}{0}{180}
\psarc(2.5,0){0.5}{0}{180}
\psarc(3.5,0){0.5}{0}{180}
\psarc(4.5,0){0.5}{0}{180}
\psarc(5.5,0){0.5}{0}{180}
\psarc(6.5,0){0.5}{0}{180}
\psarc(7.5,0){0.5}{0}{180}
\psarc(8.5,0){0.5}{0}{180}
\psarc(9.5,0){0.5}{0}{180}
\psarc(10.5,0){0.5}{0}{180}
\psarc(11.5,0){0.5}{0}{180}
\psarc(12.5,0){0.5}{0}{180}
\psarc(13.5,0){0.5}{0}{180}
\psarc(14.5,0){0.5}{0}{180}
\psarc(15.5,0){0.5}{0}{180}
\psarc[linestyle=dashed](4.5,0){9.18}{3.8}{119.3}
\uput[l](0,8){$z$}
\psdot*(0,8)
\psdot*(13.66,0.6)
\uput[r](13.66,0.56){$\gamma z$}
\end{pspicture}
\end{center}

\medskip

In this chapter, we present the theory needed to more precisely frame questions
about, and analyse, these kinds of relationships between distances,
tesselations and generators.    The upper half-plane  is an example of a very
general concept called a symmetric space. Where $(X,d)$, a symmetric space, is
acted on by a discrete group of isometries $\Gamma$, we define two distance
functions on $\Gamma$, and introduce a notion of equivalence of distance
functions. The geometric distance function on $\Gamma$ is induced by the
distance function $d$ on $X$. The word distance function on $\Gamma$, an
algebraic concept, is induced by the generators of $\Gamma$. The remainder of
this thesis investigates the circumstances in which the geometric and word
distance functions are equivalent.  We prove that, in spaces where the tiles of
the tesselation are compact, these two distance functions are indeed
equivalent. Next, so as to generalise the action of $P\sltwor$ on $U^2$, we
describe a space $\Pn$ on which $P\slnr$, $n \geq 2$, acts by isometries.
Chapter 4 comprises a proof that the geometric and word distances on $P\slnz$
are not equivalent for $n = 2$, but they are equivalent for all $n \geq 3$.

\section{Symmetric Spaces}

Let $X$ be a differentiable manifold which is also a metric space, and let $d$
be a Riemannian distance function on $X$ (see~\cite{hel1:dg} for a discussion
of Riemannian metrics and Riemannian distance functions).  Suppose there exists
a group $G$ which acts transitively and by isometries on $X$. Also suppose
there exists a special point, say 0, in $X$, and an element $k \in G$, such
that $k \circ 0 = 0$ and $Dk(0) = -I$.  Then, $(X,d)$ is said to be a symmetric
space. (The element $k$ is actually a reflection in the point~$0$.)

The upper half-plane is a symmetric space.  Considered as an open subset of
$\mathbb{R}^2$, it is a two-dimensional differentiable manifold, and it is a
metric space with distance function $d_U$. Let $G$ be the group $\sltwor$.  We
know that $G$ acts by isometries on $X$. For transitivity, it suffices to show
that for all $z \in U^2$, there exists a $g \in G$ so that $gi = z$.  Let $z =
x+ iy$ be in $U^2$, and let $g \in G$ be the matrix \[ g= \startm \sqrt{y} &
\frac{x}{\sqrt{y}} \\ 0 & \frac{1}{\sqrt{y}} \finishm. \] Then $gi = z$.  If
$k$ is the element $\matrixv$ in $G$, then $k i = i$ and $Dk(i) = -I$.

Other examples of symmetric spaces are hyperbolic $n$-space $H^n$, Euclidean
$n$-space $E^n$, and spherical $n$-space $S^n$.  There are also symmetric
spaces of matrices, one of which is described in Section~\ref{pnaction} below.
A space which is not a symmetric space is an ellipsoid.  There is no transitive
group of isometries for this space.  For example, reflections which are not in
the axes of the ellipsoid do not preserve the space.

It turns out that symmetric spaces have the features needed to construct
tesselations like that of the upper half-plane, as described in
Subsection~\ref{tesselation}.  More precisely, let $X$ be a symmetric space
which is acted on by a group of isometries $G$. The group $G$ has a normal
subgroup, say $N$, consisting of all elements of $G$ which fix every point of
the space $X$.  (If $X$ is the upper half-plane and $G = \sltwor$, this normal
subgroup is $\{ \pm I \}$.)  Then, the quotient group $G/N$ has a discrete
subgroup $\Gamma$, such that for some point $p_0$ in $X$, the stabiliser
subgroup $\Gamma_{p_0}$ is trivial. The Dirichlet domain $D$ with centre $p_0$
is a locally finite fundamental domain for the action of $\Gamma$ on $X$. The
closure of $D$ is a polyhedron which, acted on by $\Gamma$, yields a
tesselation of $X$.  This tesselation is locally finite, and any compact set in
$X$ meets only  finitely many tiles of the tesselation. Finally, the group
$\Gamma$ may be generated by those finitely many elements of $\Gamma$ which map
$\overline{D}$ to a tile which shares a side with $\overline{D}$. For the
details of the considerable amount of theory needed to establish these results,
see Chapter 4 of~\cite{hel1:dg} and Chapter 6 of~\cite{rat1:fhm}.

\section{Distance functions on discrete groups}

Let $\Gamma$ be a discrete group of isometries of a symmetric space $X$ and let
$d$ and $d'$ be distance functions on  $\Gamma$. Then $d$ and $d'$ are said to
be Lipschitz equivalent if there exist constants $c$, $C > 0$ such that, for
all $\gamma_1$, $\gamma_2 \in \Gamma$, \[ c\,d(\gamma_1,\gamma_2) \leq
d'(\gamma_1,\gamma_2) \leq C d(\gamma_1,\gamma_2). \] Lipschitz equivalence is
an equivalence relation.  We write $d \approx d'$ if distance functions $d$ and
$d'$ are in the same equivalence class.

Lipschitz equivalence allows distance functions on $\Gamma$ to be compared.  We
now define the two distance functions that we wish to compare.

\subsection{The geometric distance function}

Let $p_0$ be a point in $(X,d)$ such that the stabiliser subgroup
$\Gamma_{p_0}$ is trivial.  A geometric distance function $d_R$ on $\Gamma$ is
defined by \[ d_R(\gamma_1,\gamma_2) = d(\gamma_1 p_0, \gamma_2 p_0). \] The
letter $R$ is used since $d$ is a Riemannian distance function.  From the
properties of the distance function $d$ on $X$, it follows immediately that
$d_R$ is non-negative and symmetric, that $d_R$ satisfies the triangle
inequality, and that $\gamma_1= \gamma_2$ implies $d_R(\gamma_1,\gamma_2)=0$.
Now suppose $d_R(\gamma_1,\gamma_2)=0$. Then $d(\gamma_1 p_0,\gamma_2 p_0) =
0$, and so $\gamma_1 p_0 = \gamma_2 p_0$.  Therefore, $p_0 =
\gamma_1^{-1}\gamma_2 p_0$, meaning that $\gamma_1^{-1}\gamma_2$ is in the
stabiliser subgroup of $p_0$.  But this subgroup is trivial, hence $\gamma_1 =
\gamma_2$. Thus $d_R$ really is a distance function.

The point $p_0$ used to define the geometric distance function is not, in
general, the only point of $X$ whose stabiliser subgroup in $\Gamma$ is
trivial.  For instance, in the upper half-plane, all points $ti$, where $t >
1$, have trivial stabiliser subgroups in $P\sltwoz$.  The geometric distance
functions which are then defined using, say, the points $2i$ and $3i$ will not
be the same. For our purposes, this difference does not matter. We are
interested in Lipschitz equivalence classes of distance functions, and the
following result shows that all geometric distance functions on $\Gamma$ acting
on the space $(X,d)$ are Lipschitz equivalent.

\begin{lemma} Let $\Gamma$ be a discrete group of isometries of the symmetric
space $(X,d)$. The geometric distance function on $\Gamma$ depends on the
choice of point in $X$ with trivial stabiliser subgroup, but only up to Lipschitz
equivalence. \end{lemma}

\begin{proof} Let $d_R^p$ and $d_R^q$ be two geometric distance functions
defined with respect to the points $p$ and $q$ in $(X,d)$. Let $\gamma_1$ and
$\gamma_2$ be elements of $\Gamma$. By the triangle inequality, and the fact
that $\Gamma$ acts by isometries, \begin{align*} d_R^p(\gamma_1,\gamma_2) & =
d(\gamma_1 p,\gamma_2 p) \\ & \leq d(\gamma_1 p,\gamma_1 q) + d(\gamma_1
q,\gamma_2 q) + d(\gamma_2 q, \gamma_2 p ) \\ & =  2 d(p,q) +
d_R^q(\gamma_1,\gamma_2). \end{align*} As the group $\Gamma$ is discrete, there
is an $r_q > 0$ such that $d_R^q(\gamma_1 , \gamma_2 ) \geq r_q$ for all distinct
$\gamma_1$ and $\gamma_2$. Choose the constant $C_1$ so that $2d(p,q) \leq (C_1
- 1)r_q $. Then, for all $\gamma_1 \not = \gamma_2$, \begin{align*} 2d(p,q) +
d_R^q(\gamma_1,\gamma_2) & \leq  (C_1 - 1)r_q + d_R^q(\gamma_1,\gamma_2) \\ &
\leq  (C_1 - 1)d_R^q(\gamma_1,\gamma_2) + d_R^q(\gamma_1,\gamma_2) \\ & =  C_1
d_R^q (\gamma_1,\gamma_2). \end{align*} Thus $d_R^p (\gamma_1,\gamma_2) \leq
C_1 d_R^q(\gamma_1,\gamma_2)$ for all $\gamma_1$, $\gamma_2 \in \Gamma$.
Similarly, there exists a constant $C_2 > 0$ such that \[ d_R^q
(\gamma_1,\gamma_2) \leq C_2 d_R^p(\gamma_1,\gamma_2) \] for all $\gamma_1$ and
$\gamma_2$ in $\Gamma$.  Combining these results, \[ \frac{1}{C_1}
d_R^p(\gamma_1, \gamma_2) \leq d_R^q(\gamma_1,\gamma_2) \leq C_2
d_R^p(\gamma_1,\gamma_2), \] which completes the proof. \end{proof}

This lemma shows that changing the point $p_0$ in the definition of a
geometric distance function $d_R$ does not affect Lipschitz equivalence.   Now, the
definition of $d_R$ also depends on the distance function $d$ on $X$.   It is
reasonable to wonder whether the geometric distance  is well-defined, at
least up to Lipschitz equivalence, if the distance function on the space $X$ is
changed.  It turns out that, because $X$ is a symmetric space, any two
Riemannian distance functions $d$ and $d'$ on $X$ are scalar multiples of each
other. Hence, the respective induced geometric distance functions $d_R$ and
$d_R'$ are indeed Lipschitz equivalent.

We will from now on abuse notation and refer to $d_R$ as the geometric distance
function on $\Gamma$.

If $x$ is in $X$ and $\gamma_1$ in $\Gamma$, the distance in $X$ along the
geodesic joining  $x$ and $\gamma_1 x$ is $d(x,\gamma_1 x)$.  Since $\Gamma$
acts transitively on $X$, there is a $\gamma_2 \in \Gamma$ such that $\gamma_2
x = p_0$.  Then, \[d(x, \gamma_1 x) = d(\gamma_2 x, \gamma_2 \gamma_1 x) =
d(p_0, \gamma_2 \gamma_1 \gamma_2^{-1} p_0) =
d_R(1,\gamma_2\gamma_1\gamma_2^{-1}).\] So the geometric distance provides a
way of relating a distance in the space $X$ between points $x$ and $\gamma x$
to a distance in the group $\Gamma$.  In particular, the distance along the
geodesic joining the points $p_0$ and $\gamma p_0$ is Lipschitz equivalent to
 $d_R(1,\gamma )$.

\subsection{The word distance function}

To define the second distance function on $\Gamma$, fix a finite set of
generators for $\Gamma$, denote this set of generators by $\Sigma$, and let
$\Sigma^{-1}$ be the set $\{ \sigma^{-1} : \sigma \in \Sigma\}$. A word
distance function $d_\Sigma$ on $\Gamma$ is induced by $\Sigma$ as follows.
For $\gamma_1$ and $\gamma_2$ in $\Gamma$, $d_\Sigma(\gamma_1,\gamma_2) = n$,
where $n$ is the smallest integer such that $\gamma_1^{-1}\gamma_2$ may be
written as a word of length $n$ in terms of elements of $\Sigma \cup
\Sigma^{-1}$.   Let $l_\Sigma(\gamma)$ denote the length of a shortest word for
$\gamma$ in terms of elements of $\Sigma \cup \Sigma^{-1}$.  Then
$d_\Sigma(\gamma_1, \gamma_2) = l_\Sigma(\gamma_1^{-1}\gamma_2)$.  The word
distance function corresponds to distance in the Cayley graph for $\Gamma$
with respect to~$\Sigma$.

We now show that $d_\Sigma$ does satisfy the definition of a distance function.
Since the identity is the only element of $\Gamma$ which may be written as a
word of length zero, $d_\Sigma(\gamma_1,\gamma_2) = 0$ if and only if $\gamma_1
= \gamma_2$.  All words for elements other than the identity  have positive
length, so $d_\Sigma(\gamma_1,\gamma_2) \geq 0$ for all $\gamma_1, \gamma_2 \in
\Gamma$.  Also, for any $\gamma_1, \gamma_2 \in \Gamma$, by taking the inverse
of a shortest word for $\gamma_1^{-1}\gamma_2$, we see that
$d_\Sigma(\gamma_1,\gamma_2) = d_\Sigma(\gamma_2,\gamma_1)$. Finally, for any
$\gamma_1$, $\gamma_2$ and $\gamma_3$ in $\Gamma$, \begin{align*}
d_\Sigma(\gamma_1,\gamma_3) & =  l_\Sigma(\gamma_1^{-1}\gamma_3) \\ & =
l_\Sigma(\gamma_1^{-1}\gamma_2\gamma_2^{-1}\gamma_3) \\ & \leq
l_\Sigma(\gamma_1^{-1}\gamma_2) + l_\Sigma(\gamma_2^{-1}\gamma_3) \\ & =
d_\Sigma(\gamma_1,\gamma_2) + d_\Sigma(\gamma_2,\gamma_3), \end{align*} proving the
triangle inequality.

\begin{lemma} \label{word_equiv} The word distance function on $\Gamma$
depends on the choice of generators, but only up to Lipschitz
equivalence.
\end{lemma}

\begin{proof} Let $\Sigma = \{ \sigma_1, \sigma_2, \ldots, \sigma_n \}$ and
$\Sigma' = \{ \sigma_1', \sigma_2', \ldots, \sigma_m' \}$ be two finite sets of
generators for $\Gamma$, with $d_\Sigma$ and $d_\Sigma'$ the respective induced
word distance functions.  For $\gamma_1, \gamma_2 \in \Gamma$, we have
$d_\Sigma(\gamma_1,\gamma_2) = d_\Sigma(1,\gamma_1^{-1}\gamma_2)$, and
$d_\Sigma'(\gamma_1,\gamma_2) = d_\Sigma'(1,\gamma_1^{-1}\gamma_2)$. Thus, to
prove that $d_\Sigma \approx d_\Sigma'$, it suffices to prove that for all
$\gamma \in \Gamma$, there exists a constant $C>0$ such that\[
d_\Sigma(1,\gamma) \leq C\,d_\Sigma'(1,\gamma),
\hspace{5mm}\mbox{and}\hspace{5mm} d_\Sigma'(1,\gamma) \leq
C\,d_\Sigma(1,\gamma). \]

As $\Sigma$ and $\Sigma'$ are sets of generators, if $\sigma_i \in
\Sigma$ then both $\sigma_i$ and $\sigma_i^{-1}$ may be expressed
in terms of elements of $\Sigma' \cup (\Sigma')^{-1}$, and similarly
if $\sigma_j' \in \Sigma'$ then $\sigma_j'$ and $(\sigma_j')^{-1}$
may be expressed in terms of elements of $\Sigma\cup \Sigma^{-1}$.

For all $\gamma \in \Gamma$, by taking the inverse of a shortest word for
$\gamma$, we find that $d_\Sigma(1,\gamma) = d_\Sigma(1,\gamma^{-1})$, and
$d'_\Sigma(1,\gamma) = d'_\Sigma(1,\gamma^{-1})$. So let \begin{align*}
c_1&=\max_{1 \leq i \leq n} \{d_\Sigma'(1,\sigma_i) \}, \\ c_2&=\max_{1 \leq
j \leq m} \{d_\Sigma(1,\sigma_j')  \},\\ C&=\max\{c_1,c_2\}. \end{align*}

\noindent Then for any $\gamma \in \Gamma$, \[ d_\Sigma'(1,\gamma) \leq c_1
d_\Sigma(1,\gamma) \leq Cd_\Sigma(1,\gamma),\] and \[ d_\Sigma(1,\gamma) \leq c_2
d_\Sigma'(1,\gamma) \leq C d_\Sigma'(1,\gamma).\]  Therefore
$d_\Sigma \approx d_\Sigma'$. \end{proof}

As with the geometric distance function, this result allows us to abuse
notation.  If we are interested only in Lipschitz equivalence of distance
functions, we will now write $d_W$ for any word distance function, and refer to
$d_W$ as the word distance function on $\Gamma$.  We will continue to denote by
$d_\Sigma$ the word distance function induced by a particular set of generators
$\Sigma$.

The action of $\Gamma$ on $X$ induces a tesselation of $X$.  For $\gamma \in
\Gamma$, the geodesic segment $[p_0, \gamma p_0]$ crosses a sequence of tiles
in this tesselation. Let $T$ be the tile which has the point $p_0$ in its
interior. Then $[p_0,\gamma p_0]$ intersects a finite sequence of tiles $\{
\gamma_i T \}_{i=0}^n$, say, where $\gamma_0 = 1$ and $\gamma_n = \gamma$.
There is a slight complication if the geodesic segment $[p_0 ,\gamma p_0]$
passes through the vertex of some tile in the tesselation. Since the
tesselation is locally finite, if this happens we can always choose a finite
set of successively adjacent tiles meeting at the vertex. Then, we may assume
that for  $1 \leq i \leq n$, the tile $\gamma_{i-1}T$ is adjacent to the tile
$\gamma_i T$. Thus, the tile $T$ is adjacent to the tile
$\gamma_{i-1}^{-1}\gamma_i T$. As in the upper half-plane, the set of elements
of $\Gamma$ which map $T$ to a tile adjacent to $T$ generate $\Gamma$.  This
means that \[ \gamma =
\left(\gamma_0^{-1}\gamma_1\right)\left(\gamma_1^{-1}\gamma_2\right) \cdots
\left(\gamma_{n-1}^{-1}\gamma_n\right) \] is a word for $\gamma$ with respect
to a finite set of generators for $\Gamma$.  Its length is the same as the
number of tiles of the tesselation which are crossed by the geodesic segment
$[p_0,\gamma p_0]$.  There is no shorter word for $\gamma$ with respect to this
set of generators (this is proved for a case when $\Gamma = P\sltwoz$ in
Section~\ref{proofd2}, and the proof given there generalises).  Therefore,
$d_W(1,\gamma)$ has a geometric significance as well as its algebraic meaning:
it is Lipschitz equivalent to the number of tiles in the tesselation which are
crossed by the geodesic segment $[p_0,\gamma p_0]$.

\section{Fundamental domains of compact closure}

A straightforward example of a symmetric space is the Euclidean plane, acted on
by the group of Euclidean isometries. Let $s$ and $t$ be the orthogonal
translations $s(\vecx) = \vecx + (0,1)$ and $t(\vecx)= \vecx + (1,0)$.  Let
$\Gamma$ be the discrete subgroup of isometries generated by $s$ and $t$, and
let $p_0$ be the point $(1/2,1/2)$. Then $p_0$ has trivial stabiliser subgroup in $\Gamma$,
and the Dirichlet domain with centre $p_0$ is the open square
$(0,1)\times(0,1)$.  It seems intuitively reasonable that in the tesselation
induced, which is a grid of unit squares, the Euclidean distance between
$\vecx$ and $\gamma \vecx$ is roughly the same as the number of tiles crossed by
the line segment $[\vecx,\gamma\vecx]$.

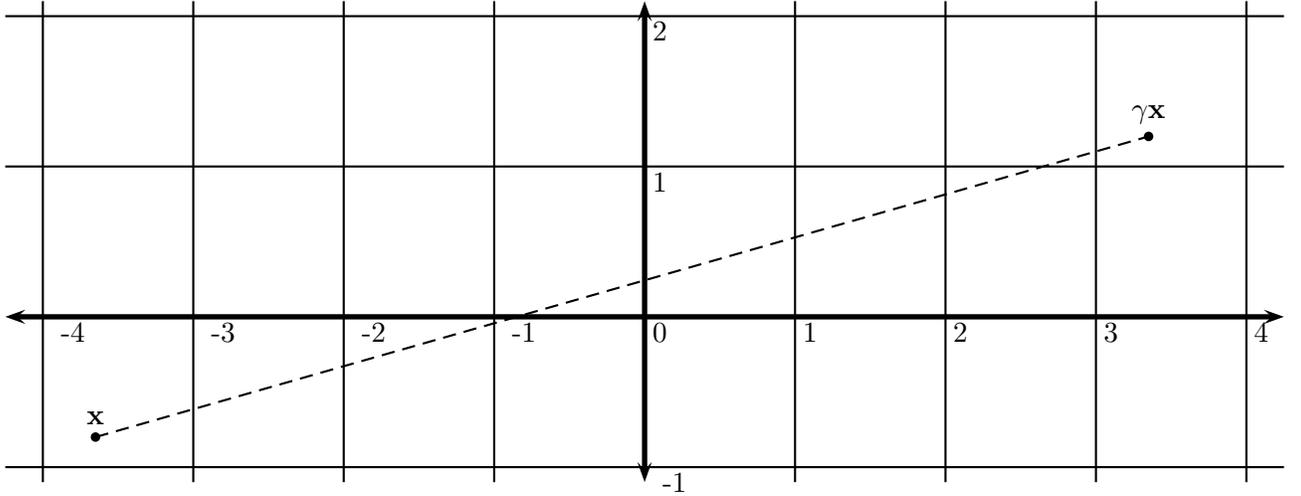
\begin{center}
\begin{pspicture}(-8,-2.2)(8,3.8)
\psline[linewidth=2pt]{<->}(-8.5,0)(8.5,0)
\psline[linewidth=2pt]{<->}(0,-2.2)(0,4.2)
\psline(-8.5,2)(8.5,2)
\psline(-8.5,4)(8.5,4)
\psline(-8.5,-2)(8.5,-2)
\psline(-8,-2.2)(-8,4.2)
\psline(-6,-2.2)(-6,4.2)
\psline(-4,-2.2)(-4,4.2)
\psline(-2,-2.2)(-2,4.2)
\psline(8,-2.2)(8,4.2)
\psline(6,-2.2)(6,4.2)
\psline(4,-2.2)(4,4.2)
\psline(2,-2.2)(2,4.2)
\uput[d](0.2,0.1){\mbox{\small{0}}}
\uput[d](2.2,0.1){\mbox{\small{1}}}
\uput[d](4.2,0.1){\mbox{\small{2}}}
\uput[d](6.2,0.1){\mbox{\small{3}}}
\uput[d](8.2,0.1){\mbox{\small{4}}}
\uput[d](-1.6,0.1){\mbox{\small{-1}}}
\uput[d](-3.6,0.1){\mbox{\small{-2}}}
\uput[d](-5.6,0.1){\mbox{\small{-3}}}
\uput[d](-7.6,0.1){\mbox{\small{-4}}}
\uput[d](0.2,2.1){\mbox{\small{1}}}
\uput[d](0.2,4.1){\mbox{\small{2}}}
\uput[d](0.4,-1.9){\mbox{\small{-1}}}
\psline[linestyle=dashed](-7.3,-1.6)(6.7,2.4)
\psdot*(-7.3,-1.6)
\psdot*(6.7,2.4)
\uput[u](-7.3,-1.6){\mbox{\small{$\vecx$}}}
\uput[u](6.7,2.4){\mbox{\small{$\gamma\vecx$}}}
\end{pspicture}
\end{center}

\medskip

\noindent The following theorem shows that this intuition is correct.  In
spaces, like this one, where the induced tesselation has compact tiles, the
geometric and word distance functions are Lipschitz equivalent.

\begin{theorem}~\cite{lmr2:wrm} \label{compactfd} Let $D$ be a fundamental
domain for the action of $\Gamma$ on $(X,d)$.  Suppose $D$ has compact
closure.  Then the geometric distance function $d_R$ on $\Gamma$ is Lipschitz
equivalent to the word distance function $d_W$ on $\Gamma$. \end{theorem}

\begin{proof} Let $\Sigma$ be a finite set of generators for $\Gamma$ such that
$\Sigma = \Sigma^{-1}$, and let $d_\Sigma$ be the word distance function
induced by the set $\Sigma$. Let $\gamma$ be in $\Gamma$.  Choose $C_0$ to be
the constant \[C_0 = \max\{ d_R(1,\sigma):\sigma \in \Sigma \}.\] We prove
that  $d_R(1,\gamma) \leq C_0 d_\Sigma(1,\gamma)$ by induction on $m =
d_\Sigma(1,\gamma)$. For the inductive step, we may write $\gamma =
\sigma_1\sigma_2\cdots\sigma_m$. Then, \begin{align*} d_R(1,\gamma) & =
d(p_0,\sigma_1\sigma_2\cdots\sigma_m p_0) \\ & =
d(\sigma_1^{-1}p_0,\sigma_2\cdots\sigma_m p_0) \\ &\leq d(\sigma_1^{-1}p_0,p_0)
+ d(p_0,\sigma_2\cdots\sigma_m p_0) \\ & = d_R(\sigma_1^{-1},1) +
d_R(1,\sigma_2\cdots\sigma_m ) \\ & \leq C_0d_\Sigma(\sigma_1^{-1},1) +
C_0d_\Sigma(1,\sigma_2\cdots\sigma_m) \\ & \leq  C_0 + C_0(m-1) \\ & =  C_0 m
\\ & =  C_0 d_\Sigma(1,\gamma). \end{align*}
Thus there exists a constant $C_1 > 0$ such that, for all $\gamma \in \Gamma$,
$d_R(1,\gamma) \leq C_1 d_W(1,\gamma)$.

To show that $d_W(1,\gamma) \leq C_2 d_R(1,\gamma)$, consider the geodesic
segment $[p_0, \gamma p_0]$. Since $\Gamma$ is discrete, the constant $K_0$
defined by \[ K_0 = \inf\{ d(p_0,\theta p_0) : \theta \in \Gamma \} \] is
strictly positive.  By the division algorithm, there exists a positive  integer
$q$ and real number $r$, where $0 \leq r < K_0 / 2$, such that \[ d(p_0,\gamma
p_0) = \frac{K_0}{2}\,q + r. \] Thus, there exists a finite sequence of points
$p_0,p_1, p_2, \ldots, p_n = \gamma p_0$ lying on the geodesic segment $[p_0,
\gamma p_0]$ such that, for $1 \leq i \leq n$, \[ \frac{K_0}{2} \leq
d(p_{i-1},p_i) < K_0. \] Note that  \[ d_R(1,\gamma) = d(p_0,\gamma p_0) =
\sum_{i=1}^n d(p_{i-1},p_i) \geq n \frac{K_0}{2}. \]

Now, the tiles of the tesselation induced by $\Gamma$ are the images of
$\overline{D}$ under elements of $\Gamma$.  So each point $p_i$ belongs to (at
least one) image of $\overline{D}$. Thus, for $0 \leq i \leq n$, we have $p_i
\in \gamma_i \overline{D}$, for some $\gamma_i \in \Gamma$.  In particular,
$\gamma_0 = 1$ and $\gamma_n = \gamma$. Then, for $1 \leq i \leq n$,  define
$r_i$ to be the element $r_i = \gamma_{i-1}^{-1}\gamma_i \in \Gamma$,  so
that \[ r_1 r_2 \cdots r_n = \gamma_0 \gamma_1^{-1}\gamma_1 \gamma_2 \cdots
\gamma_{n-1}^{-1}\gamma_n= \gamma_0^{-1} \gamma_n = 1 \gamma = \gamma. \] Let $x_i$ be the point $\gamma_i^{-1}p_i \in \overline{D}$. Since
$d(p_{i-1},p_i) < K_0$,\[ d(\gamma_{i-1}^{-1}p_{i-1},\gamma_{i-1}^{-1}p_i) =
d(x_{i-1}, \gamma_{i-1}^{-1}\gamma_i x_i) = d(x_{i-1},r_i x_i ) < K_0, \] so
the point $r_i x_i$ is in the ball $B(x_{i-1},K_0)$.

Consider the set \[ B = \bigcup_{x \in \overline{D}} B(x,K_0). \] Since
$\overline{D}$ is compact, there is an $R_0 > 0$ such that \begin{equation}
\label{fdbound}\sup \{ d(x,y): x,y \in \overline{D}\} \leq R_0. \end{equation}
That is, the distance between points in $\overline{D}$ is bounded above by the
constant $R_0$.  Let $\Theta$ be the following set of elements in $\Gamma$: \[
\Theta = \{ \theta \in \Gamma : \theta\overline{D} \cap B \not = \emptyset \}.
\]  Then, using~(\ref{fdbound}), and the fact that $p_0$ is in $\overline{D}$,
\[B \subseteq \bigcup_{\theta \in \Theta} \theta \overline{D} \subseteq
\overline{B(p_0,K_0 + R_0)}.\] Since the closed ball $\overline{B(p_0,K_0 +
R_0)}$ is compact, it meets only finitely many images of $\overline{D}$.
Thus $\Theta$ is a finite set.  Therefore, if $x$ and $y$ are points
in $\overline{D}$, and the point $\theta y$ is in the ball $B(x,K_0)$, the
element $\theta$ belongs to the finite set $\Theta$.

Now, for $1 \leq i \leq n$, the points $x_{i-1}$ and $x_i$ are in
$\overline{D}$, and the point $r_i x_i$ is in the ball $B(x_{i-1},K_0)$. This
means the element $r_i$ belongs to the set $\Theta$. Moreover, since $r_1 r_2
\cdots r_n$ is a word for $\gamma$, and $\Theta$ is finite and independent of
the element $\gamma$, the set $\Theta$ is a finite set of generators for the
group $\Gamma$. Let $d_\Theta$ be the word distance function induced by the set
$\Theta$.  Then, $d_\Theta(1,\gamma) \leq n$. Thus \[ d_R(1,\gamma) \geq  \frac{K_0}{2} n
\geq \frac{K_0}{2} d_\Theta(1,\gamma). \] Therefore, there is a constant $C_2 > 0$
such that $d_W(1,\gamma) \leq C_2 d_R(1,\gamma)$. \end{proof}

Note that we did not need the fact that $\overline{D}$ is compact to
prove the inequality $d_R(1,\gamma) \leq C_1 d_W(1,\gamma)$.

\section{A group action of $P\slnr$}\label{pnaction}

As described in Subsection~\ref{linfrac}, the group $P\sltwor$ acts on the
upper half-plane by isometries. We now describe how the group $P\slnr$, where
$n \geq 2$, acts by isometries on a symmetric space $\Pn$ of $n \times n$
matrices.  First, we define $\Pn$, and show that it is a metric space.  Then,
we define an action of the group $\slnr$ on $\Pn$, and show that this action is
by isometries and is transitive.  Following this, we conclude that $\Pn$ is a
symmetric space.  Next, we show that the upper half-plane is isometric to
$P^{(2)}$, once the distance function on $U^2$ is multiplied by a constant.
This shows that  the action of $\slnr$ on $\Pn$ is a logical generalisation of
the action of $\sltwor$ on $U^2$. Finally, we define the action of the quotient
group $P\slnr$ on $\Pn$.  The proofs of many statements in this section depend
on elementary results of matrix theory.  See, for example,~\cite{lar1:pla}.

A linear map $T:\R^n \rightarrow \R^n$ is said to be positive definite if
$T\vecx \cdot \vecx \geq 0$ for all $\vecx$ in $\R^n$, and $T\vecx \cdot
\vecx = 0$ if and only if $\vecx = \veczero$.  If $T$ is positive definite and
$T\vecx = \veczero$, then $T\vecx \cdot \vecx = 0$, so $\vecx = \veczero$.
Thus positive definite matrices are invertible.

Let $\Pn$ be the set of symmetric, positive definite $n \times n$ real matrices
of determinant 1.  In other words, $\Pn$ is the subset of symmetric positive
definite matrices in \slnr. Note that the identity matrix $I$ is in $\Pn$.  The
space $\Pn$ is a differentiable manifold of dimension $\frac{n}{2}(n + 1)$.

We now show that $\Pn$ is a metric space. From now on, we denote the Euclidean
norm of a vector $\vecx \in \R^n$ by $\| \vecx \|$.  The operator norm of a
matrix $A$ is then, by definition, \[ \| A \| = \sup_{\| \vecx \| = 1} \| A\vecx \|. \] Note that
this matrix norm is different from the norm used to equip the group $\slnr$
with a topology.  We define the distance function  $d_P$ on $\Pn$ by  \[ d_P(S_1,S_2) =
\log\| S_1^{-1}S_2 \| + \log\|S_2^{-1}S_1\|. \] Clearly, $d_P(S_1,S_2) =
d_P(S_2,S_1)$.  For any two $n \times n$ matrices $A$ and $B$, we have $\| AB
\| \leq \| A \| \| B \|$ (this is a standard result for bounded operators on a
Hilbert space; see~\cite{con1:cfa}). Since $I = S_1^{-1}S_2S_2^{-1}S_1$, it
follows that \[ 1 = \| I \| \leq \| S_1^{-1}S_2\|\| S_2^{-1}S_1 \|. \] Thus \[
d_P(S_1,S_2) = \log \| S_1^{-1}S_2\|\| S_2^{-1}S_1 \| \geq 0. \] If $S_1=S_2$
then $d_P(S_1,S_2) = \log\|I\|^2 = \log 1 = 0$.  If $d_P(S_1,S_2) = 0$, then
\[\|S_1^{-1}S_2\| \| S_2^{-1}S_1 \| = 1.\] Write $M = S_1^{-1}S_2$, so that $\|
M \| \| M^{-1} \| = 1$. Since $M$ is in \slnr, there are matrices $K_1$ and
$K_2$ in $\sonr$, and a diagonal matrix $A \in \slnr$ with positive diagonal
entries, such that $M = K_1 A K_2$.  (This decomposition is established for the
case $n = 2$ in Lemma~\ref{sl2decomp}.  The general case is proved below in
Lemma~\ref{slndecomp}, so as not to  interrupt the proof that $\Pn$ is a metric
space.)  Then, since $\| K \vecx \| = \| \vecx \|$ for all $K \in \sonr$ and
all $\vecx \in \R^n$, \[ \| M \| = \| K_1 A K_2 \| = \|A\|,
\hspace{5mm}\mbox{and}\hspace{5mm}  \| M^{-1}\| = \| K_2^{-1} A^{-1} K_1^{-1}
\| = \| A^{-1} \|. \] Also, $\det(A) = 1$.  So we have a diagonal matrix $A$
with determinant 1 and $\| A \| \| A^{-1} \| = 1$.  We now show that $A$ must
be the identity.

\begin{lemma}\label{diagnorm} Let $D$ be a diagonal matrix.  Suppose $D$ has
positive diagonal entries $\{
\lambda_1,\lambda_2,\ldots,\lambda_n\}$.  Then the operator norm of $D$ is
\[\| D \| = \max_{1\leq i\leq n}\{ \lambda_i \}.
\]\end{lemma}

\begin{proof} We have \[ \| D \|^2  =  \sup_{\|\vecx\| = 1}\| D \vecx \|^2  =
\sup\{ \lambda_1^2 x_1^2 + \lambda_2^2 x_2^2 +\cdots + \lambda_n^2 x_n^2 :
x_1^2 +  x_2^2 + \cdots + x_n^2=1 \}.\] If $\lambda_I$ is the maximum of the
set $\{ \lambda_1,\lambda_2,\ldots,\lambda_n\}$, then for $\vecx$ such that
$\|\vecx\| = 1$,\[ \lambda_1^2 x_1^2 + \lambda_2^2x_2^2 + \cdots + \lambda_n^2
x_n^2  \leq  \lambda_I^2 (x_1^2 +x_2^2+ \cdots + x_n^2 )  =  \lambda_I^2. \] So
$\| D \|^2 \leq \lambda_I^2$.  If $\vecx = \vece_I$, then \[ \lambda_1^2 x_1^2
+ \lambda_2^2x_2^2 + \cdots + \lambda_n^2 x_n^2 = \lambda_I^2, \] so equality is
attained.  \end{proof}

Let the diagonal entries of $A$ be $\{
\lambda_1,\lambda_2,\ldots,\lambda_n\}$.  We have $\det(A) = 1$, so the product
$\prod_i \lambda_i = 1$.  By Lemma~\ref{diagnorm}, the equation $\| A \| \|
A^{-1} \| = 1$ means that $(\max \lambda_i)(\max \lambda_i^{-1}) = 1$. Hence,
\begin{equation}\label{diag} (\max_i \lambda_i)(\min_i \lambda_i)^{-1} = 1.
\end{equation} Suppose first that $\max \lambda_i > 1$.  Then $\min \lambda_i < 1$, as
otherwise $\det(A) = \prod_i \lambda_i > 1$.  It follows that \[ (\max_i
\lambda_i)(\min_i \lambda_i)^{-1} > 1, \] which contradicts~(\ref{diag}).  Therefore $\max
\lambda_i \leq 1$.  But if $\max \lambda_i < 1$, $\det(A) < 1$.  Hence, $\max
\lambda_i = 1$.  For the minimum diagonal entry, we now have $\min \lambda_i
\leq 1$. If $\min \lambda_i < 1$ then $\det(A) < 1$, so $\min \lambda_i = 1$.
Therefore, $\lambda_i = 1$ for all $i$, and $A$ is the identity matrix.

Since $A$ is the identity, $M = K_1 K_2 \in \sonr$. So $S_1^{-1}S_2$ is in
$\sonr$. Hence, $S_2 = S_1 U$ for some $U$ in \sonr.  As the matrices $S_1$ and
$S_2$ are symmetric, we have $S_i = (S_i S_i^T)^{1/2}$ for $i = 1,2$.  Then, \[
S_2  =  \left(S_2 S_2^T\right)^{1/2}  =  \left(S_1 U U^T S_1^T\right)^{1/2} =
\left(S_1 S_1^T\right)^{1/2}  =  S_1. \] This proves that $d_P(S_1,S_2) = 0$ if and only
if $S_1 = S_2$.

For the triangle inequality, \begin{align*} d_P(S_1,S_3) & =
\log\|S_1^{-1}S_3 \| \| S_3^{-1}S_1\| \\ & =  \log\|S_1^{-1}S_2S_2^{-1}S_3 \|
\| S_3^{-1}S_2S_2^{-1}S_1\| \\ & \leq  \log \|S_1^{-1}S_2\|\|S_2^{-1}S_3 \| \|
S_3^{-1}S_2\|\|S_2^{-1}S_1\| \\ & =  \log \|S_1^{-1}S_2\|\|S_2^{-1}S_1 \| +
\log \| S_2^{-1}S_3\|\|S_3^{-1}S_2\| \\& =  d_P(S_1,S_2) + d_P(S_2,S_3).
\end{align*} We conclude that $\Pn$ with the distance function $d_P$ is a
metric space.

We now define an action of the group $\slnr$ on \Pn. For $M \in \slnr$ and $S \in \Pn$, the
action of $M$ on $S$ is \[ M \circ S = MSM^T. \] Since $S$ is symmetric, $MSM^T
= (MSM^T)^T$. Since $S$ is positive definite,  \[MSM^T\vecx\cdot\vecx =
S(M^T\vecx)\cdot(M^T\vecx) \geq 0,\] with equality if and only if $M^T\vecx =
\veczero$. But $M^T \in \slnr$, so $M^T\vecx = \veczero$ if and only if $\vecx
= \veczero$.  Hence $M \circ S \in \Pn$. Since $ISI^T = S$, and
\[(M_1M_2)S(M_1M_2)^T = M_1(M_2 S M_2^T)M_1^T,\] this is indeed a group
action.

The stabiliser subgroup of $I$ in $\slnr$ is \[  \{ M \in \slnr : M\circ I = I\}
= \{ M \in \slnr : MM^T = I \}, \] which is the special orthogonal group
$SO(n,\R)$.  This fact is used to prove the following decomposition of \slnr.

\begin{lemma} \label{slndecomp}  Let $M \in \slnr$.  Then there exist matrices
$K_1$ and $K_2$ in $\sonr$, and a diagonal matrix $A$ in $\slnr$ with positive
diagonal entries, such that $M = K_1 A K_2$.
\end{lemma}

\begin{proof} First, $M \circ I = MIM^T$ is a symmetric positive definite
matrix.  Since  it is symmetric, there exists an orthogonal basis of
eigenvectors of $MIM^T$.   If $K$ is the orthogonal matrix which has these
eigenvectors as columns, then the matrix $KMIM^TK^{-1} = KMIM^TK^T$ is diagonal, and the
diagonal entries of $KMIM^TK^T$ are the eigenvalues of $MIM^T$.  By multiplying
the columns of the matrix $K$ by a suitable constant, we may assume that $\det(K) = 1$, so
that $K \in \sonr$.  As the matrix $MIM^T$ is positive definite, its eigenvalues are
real and positive. So there exists a diagonal matrix $D$ such that
$KMIM^TK^T = D^2$.  Note that $\det(D) = \sqrt{\det(D^2)}
= 1$, so $D \in \slnr$.
Then \[   D^{-1} K MIM^T K^T D^{-1} =
(D^{-1}KM) I (D^{-1}KM)^T = (D^{-1}KM) \circ I = I. \] As the group $\sonr$
is the stabiliser subgroup of $I$ in $\Pn$, we have $D^{-1}KM = K'$ for some
$K' \in \sonr$.  Making $M$ the subject and then relabelling, we obtain $M =
K^{-1}DK' = K_1 A K_2$. \end{proof}

We now further investigate the action of $\slnr$ on $\Pn$. For $M \in \slnr$
and $S_1,S_2 \in \Pn$, \begin{align*} d_P(M \circ S_1,M \circ S_2) & =  \log\|
S_1^{-1}M^{-1}MS_2 \| \| S_2^{-1}M^{-1}MS_1 \| \\ & =  \log \| S_1^{-1}S_2 \|
\| S_2^{-1}S_1 \| \\ & =  d_P(S_1,S_2). \end{align*} Thus  $\slnr$ acts by
isometries on $\Pn$.

The action of $\slnr$ on $\Pn$ is also transitive.   Let $S$ be in $\Pn$. As in
the proof of Lemma~\ref{slndecomp}, since $S$ is symmetric and positive
definite, there exists $K \in \sonr$, and a diagonal matrix $D \in \slnr$ with
positive diagonal entries, such that $KSK^T = D^2$. Then \[ S  =  K^TDDK  =
(K^TD)I(K^TD)^T. \] Setting $M = K^TD$, we have $M \in \slnr$ and $M\circ I =
S$.

The only thing now required to conclude that $\Pn$ is a symmetric space is the
point $0$ in $\Pn$ and element $k \in \slnr$ such that $k$ fixes $0$ and $Dk(0)
= -I$.  For any $n$, we take the point $0$ to be the identity matrix.  Then,
for $n = 2$, take $k = \matrixv$.  For $n \geq 3$, we may take as $k$ a matrix which
has all zeros except for the diagonal running from bottom left to top right.
On this diagonal, each entry is either $1$ or $-1$, arranged so that $\det(k) =
1$.

The following lemma motivates our investigation of the space $\Pn$, as it shows
that the upper half-plane and the space $\Ptwo$ are isometric after adjusting
the distance function on $U^2$ (or $\Ptwo$) by a constant.  Lipschitz
equivalence is not affected by multiplying a distance function by a constant.

\begin{lemma} \label{U2isomP2} There is a bijection $\varphi: U^2 \rightarrow
\Ptwo$ so that for some constant $C> 0$, \[ d_U(z,w) = C
d_P(\varphi(z),\varphi(w)) \] for all $z, w \in U^2$. That is, $\varphi$ is a
similarity.   \end{lemma}

\begin{proof} Let $G = \sltwor$ and $K = SO(2,\R)$.  Recall that $K$ is the
stabiliser subgroup of $i$ when $G$ is acting on $U^2$ (see the proof of
Lemma~\ref{sl2decomp}), and of $I$ when $G$ is acting on $\Ptwo$.

Define a map $\varphi: U^2 \rightarrow \Ptwo$ by $\varphi(gi) = g \circ I$.
The domain of $\varphi$ is $U^2$ because $U^2 = Gi$, and $\varphi$ is onto $\Ptwo$
since  $\Ptwo = G \circ I$.  To show $\varphi$ is injective, suppose $\varphi(g_1
i) = \varphi(g_2 i)$. Then $g_1 \circ I = g_2 \circ I$, and so $I = g_2^{-1}g_1
\circ I$, hence $g_2^{-1}g_1 \in K$.  Thus $g_2^{-1}g_1 i = i$, and so $g_1 i =
g_2 i$.  Therefore $\varphi$ is a bijection.

Now, for all $g_1, g_2 \in G$, \[ d_U(g_1 i , g_2 i)  =  d_U( i , g_1^{-1}g_2
i), \hspace{5mm}\mbox{and}\hspace{5mm} d_P(g_1 \circ I, g_2 \circ I)  =
d_P(I ,g_1^{-1}g_2 \circ I). \] So, to show that $d_U(g_1i,g_2i) = C
d_P(\varphi(g_1i),\varphi(g_2i))$, it suffices to prove that for all $g \in G$,
\[ d_U(i,gi) = C d_P(I,g \circ I). \] By Lemma~\ref{sl2decomp}, we may write
$g$ as $k_1 a k_2$, where $k_1, k_2 \in K$ and $a = \startm s & 0 \\ 0 & s^{-1}
\finishm$ for some $s \geq 1$.  Since $K$ stabilises $i$ in $U^2$ and $I$ in
$\Ptwo$, $d_U(i,gi) = d_U(i,ai)$ and $d_P(I,g\circ I) = d_P(I,a\circ I)$. In
$U^2$, $ai = s^2i$, and $d_U(s^2i,i) = 2\log s$.  In $\Ptwo$, \[ a \circ I =
\startm s & 0 \\ 0 & s^{-1} \finishm \circ I = \startm s & 0 \\ 0 & s^{-1}
\finishm^2 = \startm s^2 & 0 \\ 0 & s^{-2} \finishm. \] Then, by
Lemma~\ref{diagnorm}, $\| a \circ I \| = \| (a \circ I)^{-1} \| = s^2$, so
$d_P(a\circ I , I) = 2 \log s^2 = 4 \log s$. Thus, with $C = 2$, $d_U(g_1 i,
g_2 i) = C d_P(g_1 \circ I, g_2 \circ I)$. \end{proof}

We now, finally, define the action of the group $P\slnr$ on $\Pn$. Let $M$ be
in $\slnr$ and $S$ in \Pn.   Then $MSM^T = (-M)S(-M)^T$, so $M \circ S = (-M)
\circ S$. Write $[M]$ for the coset of $M$ in the group $P\slnr$. Then $[M]
\circ S = MSM^T$ is a well-defined group action of $P\slnr$ on the space
$\Pn$.  So, the group $\pslnr$ acts transitively and by isometries on the
metric space $\Pn$.


\chapter{Equivalence of Distance Functions on $PSL(n,\mathbb{Z})$}


\section{Introduction}

We have, in Chapter 3, described the action of the group $P\slnr$ on the
symmetric space $\Pn$, and shown that this action is a generalisation of the
action of the group $P\sltwor$ on the upper half-plane.  In Chapter 4, we
consider the Lipschitz equivalence of distance functions on the discrete
subgroup $P\slnz$ of  $P\slnr$.  The content of this chapter is an elaboration of the
paper~\cite{lmr1:cseg} by Lubotzky, Mozes and Raghunathan.

A fundamental domain for the action of $P\slnz$ on $\Pn$ does not have compact
closure~\cite{hum1:ag}.   So Theorem~\ref{compactfd} does not apply to
$P\slnz$.  Our aim in this chapter is to prove the following theorem.

\begin{theorem}\label{bigone} Let $d$ be an integer greater than or equal to
$2$.  Let $\Gamma$ be the group $P\sldz$ acting on the space $\Pd$.
When $d = 2$, the geometric distance function $d_R$ on $\Gamma$ is not
Lipschitz equivalent to the word distance function $d_W$ on $\Gamma$.  For all
$d \geq 3$, the distance function $d_R$ on $\Gamma$ is Lipschitz equivalent to the distance
function $d_W$ on $\Gamma$. \end{theorem}

\noindent (The switch to the letter $d$ is so that the letter $n$ is available
to be used in proofs.)

Before proving Theorem~\ref{bigone}, we will discuss briefly
how the distance functions  $d_R$ and $d_W$ are actually defined for the group
$P\sldz$.  Then, we prove Theorem~\ref{bigone} for the case where $\Gamma =
P\sltwoz$.  To prove Theorem~\ref{bigone} for the case $\Gamma = P\sldz$, where
$d \geq 3$, it suffices to prove that there exist constants $C_1, C_2 > 0$ such
that, for all $\gamma \in \Gamma$,  \[ d_R(1,\gamma) \leq C_1 d_W(1,\gamma),
\mand d_W(1,\gamma) \leq C_2 d_R(1,\gamma). \] The proof is then broken into
three main steps. \begin{enumerate} \item The inequality $d_R(1,\gamma) \leq C_1
d_W(1,\gamma)$ can be established in exactly the same way as in
Theorem~\ref{compactfd}. \item The element $\gamma \in P\sldz$ is the coset of
some matrix $g \in \sldz$. There exists a constant
$K > 0$ such that $d_W(1,\gamma) \leq K \log \| g\|$, where $\| g\|$ is the
operator norm of the matrix $g$. This step is Theorem~\ref{normthm}. \item We conclude, using the definition of the
distance function $d_P$ on $\Pn$, that $d_W(1,\gamma) \leq
C_2d_R(1,\gamma)$. \end{enumerate}  The proof of Theorem~\ref{normthm} forms the bulk
of this chapter.

We will often need to compare the growth of functions, and so state here a
useful definition.  Let $f$ and $g$ be functions on a discrete set $X$. We say
that the function $f$ is $O(g)$, or $f= O(g)$, if there exists a  constant $C >
0$, and a finite subset $X_0 \subseteq X$,  such that $f(x) \leq C g(x)$ for
all $x \in X \setminus X_0$.  The following lemma is also useful.

\begin{lemma} \label{basicineq} Let $f$ and $g$ be functions $\R \rightarrow
\R$.  Then $f(x) \leq g(x)$ for all $x \geq x_0$ if $f(x_0) \leq
g(x_0)$, and  $f'(x) \leq g'(x)$ for all $x \geq x_0$. \end{lemma}

\section{Distance functions on $P\sldz$}

Here we discuss briefly how the geometric and word distance functions on the
discrete group $P\sldz$ are defined.

The definition of the geometric distance function is with respect to some
matrix $S_0 \in \Pd$ which has a trivial stabiliser subgroup in $P\sldz$.
Setting $S_0$ to be the identity $I$ does not suffice, because if $K$ is any
orthogonal matrix then $K \circ I = I$. However, there
does exist a suitable $S_0$ arbitrarily close to $I$ in $\Pd$.  For example,
when $d = 2$, for small $\varepsilon$ the matrix \[ \startm  1+ \varepsilon^2 &
\varepsilon \\ \varepsilon & 1 \finishm \in \Ptwo \] has a trivial stabiliser
subgroup in $P\sltwoz$.

For the word distance function, it will often be convenient to consider words
in the group \sldz\ rather than in $P\sldz$. For $g \in \sldz$, write $[g]$ for
the coset of $g$ in $P\sldz$.  Then if $g_1 g_2 \cdots g_n$ is a word for $g$
in terms of elements of $\sldz$, a word for the coset $[g]$ is $
[g_1][g_2]\cdots[g_n]$.  Going in the other direction, if $[g] =
[g_1][g_2]\cdots[g_n]$ is a word in $P\sldz$ then, in \sldz, either $g =
g_1g_2\cdots g_n$, or $g = (-I)g_1g_2\cdots g_n$.  Multiplying by the matrix
$-I$ adds 1 to the word length, which will not affect Lipschitz equivalence.
Thus, it is reasonable to move back and forth between words in $\sldz$ and
words in $P\sldz$ when constructing words and proving results about the word
distance function $d_W$.  We will usually write $\gamma$ both for the coset
$\gamma \in P\sldz$, and for an element $g \in \sldz$ such that $[g] = \gamma$;
we may describe $\gamma$ as a matrix rather than  a coset.

We now describe a particular finite set $\Sigma_d$ of generators of $P\sldz$.
This set contains the (cosets of the) following matrices.  First, it contains
the upper-triangular matrices which have all diagonal entries equal to 1, and
one 1 above the diagonal.  The set $\Sigma_d$ also contains all permutation
matrices which have determinant 1. Finally, $\Sigma_d$ contains all other
permutation matrices, but with $1$ changed to $-1$ wherever this will ensure a
determinant of $+1$. So, generators for $P\sltwoz$ are \[ \Sigma_2 =
\left\{\startm 1 & 1 \\ 0 & 1 \finishm, \matrixv \right\}, \] and generators
for $PSL(3,\mathbb{Z})$ are \begin{align*} \Sigma_3 &= \left\{ \startm 1 & 1 &
0 \\ 0 & 1 & 0 \\ 0 & 0 & 1 \finishm, \startm 1 & 0 & 1 \\ 0 & 1 & 0 \\ 0 & 0 &
1 \finishm, \startm 1 & 0 & 0 \\ 0 & 1 & 1 \\ 0 & 0 & 1 \finishm \right\} \\ &
\mbox{} \bigcup \left\{ \startm 0 & 1 & 0 \\ 0 & 0 & 1 \\ 1 & 0 & 0 \finishm,
\startm 0 & 0 & 1 \\ 1 & 0 & 0 \\ 0 & 1 & 0 \finishm \right\}  \\ &  \mbox{}
\bigcup \left\{ \startm \pm 1 & 0 & 0 \\ 0 & 0 & \pm 1 \\ 0 & \pm 1 & 0
\finishm, \startm 0 & \pm 1 & 0 \\ \pm 1 & 0 & 0 \\ 0 & 0 & \pm 1 \finishm,
\startm 0 & 0 &\pm 1  \\ 0 & \pm 1 & 0 \\ \pm 1 & 0 & 0 \finishm \right\},
\end{align*} where the signs in the last set of matrices are arranged so that
the determinant in each case is 1.  The proof that the sets $\Sigma_d$ do
generate the groups $P\sldz$ is by induction on $d$. For the inductive step, it
is possible by multiplying by elements of $\Sigma_d \cup \Sigma_d^{-1}$ to
transform a matrix $\gamma$ in $P\sldz$ into a matrix $\theta$ which has first
column $(1,0,\ldots,0)^T$ and first row $(1,0,\ldots,0)$. Then, by the
inductive assumption, the $(d-1)\times(d-1)$ matrix in the lower right-hand
corner of $\theta$ may be written in terms of elements of $\Sigma_{d-1} \cup
\Sigma_{d-1}^{-1}$.  The set of generators $\Sigma_{d-1}$ may then be
considered a subset of $\Sigma_d$, by embedding $PSL(d-1,\mathbb{Z})$ in the
lower right-hand corner of $P\sldz$.

\section{Proof of Theorem~\ref{bigone} where $\Gamma = P\sltwoz$}\label{proofd2}

The group $P\sltwoz$ may be generated by  $u =\matrixu$ and $v~=~\matrixv$.
Let $\Sigma = \{ u , v \}$ and let $d_\Sigma$ be the word distance function
induced by the set of generators $\Sigma$.   We show that there is no $C_1 > 0$
such that $d_W(1,u^n) \leq C_1 d_R(1,u^n)$ for all $n \geq 2$.

Let $n \geq 2$.  Then $d_\Sigma(1,u^n) \leq n$.  To show that
$d_\Sigma(1,u^n) \geq n$, we consider the tesselation of the upper half-plane
induced by the action of  $\Gamma = P\sltwoz$.  Recall that the point
$p_0 = 2i$ lies in the tile $T$, where $T$ is the hyperbolic triangle with
vertices at $\pm \frac{1}{2} + \frac{\sqrt{3}}{2}i$ and $\infty$.  Let $w_1w_2\cdots w_l$ be any
word for $u^n$ in terms of the elements of $\Sigma \cup \Sigma^{-1}$.  Let
$W_0$ be the identity in $\Gamma$.  For $1 \leq j \leq l$, let $W_j$ be the
partial word $w_1w_2\cdots w_j$.  Then the sequence of partial words, $\{ W_j
\}_{j=0}^l$, corresponds to a sequence of tiles, $\{W_jT\}_{j=0}^l$, such that
$W_0T = T$ and $W_lT = u^nT$.  In this sequence of tiles, for each $j$ such
that  $1 \leq j \leq l$, the tiles $W_{j-1}T$ and $W_jT$ share a boundary.  For
$k \in \Z$, let $L_k$ be the geodesic of $U^2$ which is given by $\re(z)=k+1/2$.  Then
$L_k$ is a boundary between pairs of tiles.  See the picture of the tesselation
on page~\pageref{tess}.  Since $u^np_0 = n + 2i$, the $n$
boundaries $L_0,L_1,\ldots,L_{n-1}$ all lie between the points $p_0$ and
$u^np_0$.  Thus, the sequence $\{W_jT\}_{j=0}^l$ must contain at
least $n$ successive pairs of tiles which have a geodesic $L_k$ as their common
boundary.  Therefore, $l \geq n$.  So $d_\Sigma(1,u^n) \geq n$;
hence $d_\Sigma(1,u^n) = n$.  This implies that the word distance $d_W(1,u^n)$
grows linearly in $n$.

For the geometric distance function, we use Lemma~\ref{U2isomP2}, which states
that the upper half-plane is isometric to $\Ptwo$ once the distance function on
one of these spaces is multiplied by a scalar.  Multiplying a distance function
by a scalar does not affect Lipschitz equivalence, so there is some constant $C
> 0$ such that  \[ d_R(1,u^n) = d_P(S_0, u^n \circ S_0) =  C d_U (p_0, u^n p_0) = C
\cosh^{-1}\left( 1 + \frac{n^2}{8}\right). \] The function $\cosh^{-1}( 1 +
n^2/8)$ is $O(\log n)$, so $d_R(1,u^n)$ is $O(\log n)$.  But there is no
constant $C_1>0$ such that $ n \leq C_1 \log n $ for all $n \geq 2$.  Hence,
there is no  $C_1 > 0$ such that $d_W(1,u^n) \leq C_1 d_R(1,u^n)$ for all $n
\geq 2$.  So, in this case, the geometric distance function $d_R$ is not
Lipschitz equivalent to the word distance function $d_W$.

\section{Proof of Theorem~\ref{bigone} where $\Gamma = P\sldz$ and $d \geq 3$}

The first step in this proof is to show that there exists a constant $C_1 > 0$
such that, for all $\gamma \in \Gamma$,  \begin{equation}\label{stepone}
d_R(1,\gamma) \leq C_1 d_W(1,\gamma). \end{equation} As claimed in the
introduction to this chapter, the proof  in Theorem~\ref{compactfd} suffices.
This is because the proof of inequality~(\ref{stepone}) in
Theorem~\ref{compactfd} did not depend on the compactness of tiles in the
tesselation induced by $\Gamma$.  Indeed,
inequality~(\ref{stepone}) holds in the case $\Gamma = P\sltwoz$.

The second step in the proof of Theorem~\ref{bigone} for the case $d \geq 3$ is
the following theorem.  We denote by $\| g \|$ the operator norm of a matrix $g
\in \sldz$, that is, $\| g \| = \sup_{\| \vecx \| = 1}\| g\vecx \|$.  Now, if
$g \in \sldz$, then $\| g \| = \| {-g} \|$.  So, if
$\gamma \in P\sldz$ is the coset $\gamma = [g]$, then it is well-defined to
declare the norm of $\gamma$ to be $\| \gamma \| = \| g \|$.

\begin{theorem} \label{normthm} Let $d$ be an integer greater than or equal to
3.  Then there exists a constant $K > 0$ such that, for all $\gamma  \in P\sldz$,
$d_W(1,\gamma) \leq K \log \| \gamma \|$.  \end{theorem}

Before we prove Theorem~\ref{normthm}, we show how it implies that there exists
a constant $C_2 > 0$ such that, for all $\gamma \in \Gamma$,
\begin{equation}\label{stepthree} d_W(1,\gamma) \leq C_2 d_R(1,\gamma).
\end{equation} Inequality~(\ref{stepthree}) is the final step in the proof of
Theorem~\ref{bigone}. Applying Theorem~\ref{normthm} to first $\gamma$ then
$\gamma^{-1}$, we obtain \[ d_W(1,\gamma) \leq K \log \| \gamma \|, \mand
d_W(1,\gamma^{-1}) \leq K \log \| \gamma^{-1} \|. \] Now, $d_W(1,\gamma) =
d_W(1,\gamma^{-1})$, which gives us \begin{equation}\label{wordlesslog}
d_W(1,\gamma) \leq \frac{K}{2} \log \| \gamma \| \|\gamma^{-1} \|.
\end{equation} Next, we consider distances in the space $\Pd$. The distance
between the identity $I$ and the matrix $\gamma \circ I$ is \[ d_P(I ,\gamma
\circ I) =
d_P(I, \gamma \gamma^T) = \log \| \gamma \gamma^T \| \| (\gamma \gamma^T )^{-1}
\| = \log \| \gamma \gamma^T \| \| (\gamma^{-1})^T \gamma^{-1} \|. \] It is a
standard result that if $A$ is a bounded operator on a Hilbert space then
$\| A \|^2 = \| A^* \|^2 = \| A A^* \|$, where $A^*$ is the adjoint of $A$;
see~\cite{con1:cfa}. Applying this to the matrix $\gamma$, which has adjoint
$\gamma^T$, we obtain \[ d_P(I, \gamma \circ I ) = \log \| \gamma \|^2 \|
\gamma^{-1} \|^2 = 2 \log \| \gamma \| \| \gamma^{-1} \|. \] Combining this
with~(\ref{wordlesslog}) yields \begin{equation} \label{wordlessdP}
d_W(1,\gamma) \leq \frac{K}{4} d_P(I,\gamma \circ I). \end{equation} The geometric
distance function on $\Gamma$ is $d_R(1,\gamma) = d_P(S_0,\gamma \circ S_0)$.  Since
$\Gamma$ is discrete, there is a constant $r_0 > 0$ such that for all $\gamma
\in \Gamma$, with $\gamma \not = 1$,  we have $d_R(1,\gamma) \geq r_0$.  Choose
the constant $C$ so that $2 d_P(I, S_0) \leq (C - 1)r_0 $.  Then, by the
triangle inequality, for $\gamma \not = 1$, \begin{align*} d_P(I, \gamma \circ I) &
\leq  d_P(I, S_0) + d_P(S_0,\gamma \circ S_0) + d_P(\gamma \circ S_0, \gamma
\circ I) \\ & =
2d_P(I, S_0) + d_R(1,\gamma) \\ & \leq  (C - 1)r_0 + d_R(1,\gamma) \\ & \leq
(C -1)d_R(1,\gamma) + d_R(1,\gamma) \\ & =  C d_R(1,\gamma). \end{align*}
Hence, using~(\ref{wordlessdP}), we conclude that there exists a constant $C_2
> 0$ such that, for all $\gamma \in \Gamma$, \[ d_W(1,\gamma) \leq C_2
d_R(1,\gamma). \]

\subsection{Outline of the proof of Theorem~\ref{normthm}}

Since Theorem~\ref{normthm} seeks to bound the word length of $\gamma \in
\Gamma$, we begin by, in Subsection~\ref{factorsldz}, constructing a special
word  \[ \gamma = \delta_1 \delta_2 \cdots \delta_{d^2}. \] Each of the terms
$\delta_i$ in this factorisation belongs to some copy of $P\sltwoz$ in
$P\sldz$.  We want to relate the word for $\gamma$ to the norm of $\gamma$.
Since the norm of $\gamma$ depends on the entries of $\gamma$, we show that, in
fact, the elements $\delta_i$ may be chosen so that the entries of each
$\delta_i$ are bounded by a fixed polynomial in the entries of $\gamma$.

Next, in Subsection~\ref{factorsltwoz}, we find a bound on the word length of
any $\delta$ belonging to some copy of $P\sltwoz$ in $P\sldz$.  Let $d_W$ be
the word distance function on $P\sldz$.  We show that there exists a constant
$K_1 > 0$ such that, for all such $\delta$,  \begin{equation}\label{step22ineq}
d_W(1,\delta) \leq K_1 \log \| \delta \|. \end{equation} To establish this
inequality, we first consider the elementary matrices which have one
above-diagonal entry $1$, all diagonal entries $1$, and all other entries $0$.
We show that if $\delta$ is such a matrix, then, roughly, $d_W(1,\delta^n) =
O(\log |n|)$.  It is here that $d \geq 3$ is necessary, for if $d = 2$ and
$\delta = \matrixu$, then $d_W(1,\delta^n) \not = O(\log |n|)$.  Next, we
restrict our attention to $P\sltwoz$. We construct a special word for any
$\delta \in P\sltwoz$ in terms of only the matrices $\matrixv$ and $\startm 1 &
n \\ 0 & 1 \finishm$, where $n \in \Z$.  Using the tesselation of the upper
half-plane, we show that a particular function defined on the terms in this
special word is bounded by the hyperbolic distance from $p_0$ to $\delta p_0$.
Then, we show that this hyperbolic distance is bounded by a logarithm of $\|
\delta \|$.  Combining these results allows us to prove
inequality~(\ref{step22ineq}).

In Subsection~\ref{opnorm}, we prove an inequality relating the operator norm
of a matrix $\gamma$ to the entry in $\gamma$ which has maximum absolute
value.  Subsection~\ref{finishnormproof} combines the above results to complete
the proof of Theorem~\ref{normthm}.

\subsection{Factorisation of $\gamma \in P\sldz$}\label{factorsldz}

Since a word in $\sldz$ yields a word in $P\sldz$, the statements of
Lemma~\ref{e1_transpose} and Corollary~\ref{product1} here are in terms of the
group $\sldz$.

First, some definitions.   A vector $\mathbf{x} = (x_1,x_2,\ldots,x_d)^T$ in
$\Z^d$ is said to be unimodular if the greatest common divisor of its
components, $\gcd(x_1,x_2, \ldots, x_d)$, is 1. Some elements of the set of
integers $\{x_1,x_2,\ldots,x_d\}$ may be zero.  If, for some $1 \leq j \leq d$,
we have $x_j = 0$,  then define \[\gcd(x_1,x_2, \ldots, x_d) = \gcd(x_1,x_2,
\ldots, x_{j-1},x_{j+1}, \ldots,x_d).\]  If $x_1 = x_2 = \cdots = x_d = 0$,
define $\gcd(x_1,x_2,\ldots,x_d) = 0$.  Next, for $1 \leq i \not = j \leq d$,
we will denote by $E_{ij}(t)$ the elementary matrix having $t$ at the $(i,j)$
entry, all diagonal entries $1$, and all other entries $0$. Another definition
needed is that of the group $SL^{s,t}(2,\mathbb{Z})$, where $1 \leq s \not = t
\leq d$.  This is the copy of \sltwoz\ in \sldz\ which fixes the basis vectors
$\mathbf{e}_j$ for $j$ not equal to $s$ or $t$.  For example, if $d = 3$ and
$\startm a_{11} & a_{12} \\ a_{21} & a_{22} \finishm \in \sltwoz$, then \[
SL^{1,2}(2,\Z) = \left\{ \startm a_{11} & a_{12} &  0 \\ a_{21} & a_{22} & 0 \\
0 & 0 & 1 \finishm \right\}, \mand SL^{1,3}(2,\Z) = \left\{ \startm a_{11} & 0
& a_{12}  \\ 0  & 1  & 0 \\ a_{21} & 0 & a_{22} \finishm \right\}. \]

\begin{lemma} \label{e1_transpose} Let $\mathbf{x} = (x_1,x_2,\ldots,x_d)^T \in
\mathbb{Z}^d$ be a unimodular vector.  Then there is an element $\gamma_1 \in
\{ E_{1i}(1)\ |\ i \not = 1 \} \cup \{I\}$, and elements $\gamma_i \in
SL^{1,i}(2,\mathbb{Z})$ for $2 \leq i \leq d$, such that \[\gamma_d
 \cdots \gamma_2 \gamma_1 \mathbf{x} = (1,0,\ldots,0)^T.\]
Moreover, there is a fixed polynomial $p$ in the components of $\vecx$, such
that, for $1 \leq i \leq d$, and $1 \leq k,l \leq d$, the
$(k,l)$ entry of $\gamma_i$ satisfies\[ |(\gamma_i)_{kl}| \leq
|p(x_1,x_2,\ldots,x_d)|.\] \end{lemma}

\begin{proof} Since $\mathbf{x}$ is unimodular, $\mathbf{x} \not = \mathbf{0}$.
If the first component $x_1$ is non-zero then just put $\gamma_1 = I$. If $x_1
= 0$, then $x_j \not = 0$ for some $2 \leq j \leq d$; put $\gamma_1 =
E_{1j}(1)$. Then the first component of $\gamma_1 \mathbf{x}$ is non-zero.

For $i = 2,3,\ldots, d$ in turn, we construct $\gamma_i$.  Let $\mathbf{y}^i$
be the vector \[\mathbf{y}^i = (y_1^i,y_2^i,\ldots,y_d^i)^T =
\gamma_{i-1}\cdots\gamma_2\gamma_1 \mathbf{x}.\] By the construction which follows, we
have $y_1^i \not = 0$, and $y_j^i = 0$ for $1 < j < i$. Let $a_i$ be the
greatest common divisor of $y_1^i$ and $y_i^i$.  Then, by the Euclidean
algorithm, there exist integers $u$ and $v$ such that $uy_1^i + vy_i^i = a_i$.

We claim that $u$ and $v$ may be chosen so that $|u|$ and $|v|$ are both less
than or equal to the polynomial $(y_1^i)^2 + (y_i^i)^2$. If $y_i^i = 0$, then
with $u = 1$ and $v = 0$ the claim follows immediately.  If $|y_1^i| = |y_i^i|
\not = 0$ then, again, put $u = 1$ and $v = 0$. Otherwise, we may suppose
without loss of generality that $|y_1^i| > |y_i^i|$.  Now, since $uy_1^i +
vy_i^i = a_i$, it follows that \[(u + my_i^i)y_1^i + (v - my_1^i)y_i^i = a_i\]
for all $m \in \mathbb{Z}$ as well.  We may now choose $m$ so that $|u +
my_i^i| \leq |y_i^i|/2$.  Then, relabel $u + my_i^i$ as $u$ and $v - my_1^i$ as
$v$, so that $|u| \leq |y_i^i|/2 \leq |y_i^i|$.   Because $|y_1^i| \not =
|y_i^i|$, the greatest common divisor $a_i$ is not equal to $y_1^i$ or to
$-y_1^i$, thus $|a_i| \leq |y_1^i|/2$.  So, for $|v|$, we have \[  |v|  =
\left| \frac{a_i}{y_i^i} - \frac{u}{y_i^i}y_1^i \right| \leq
\frac{|a_i|}{|y_i^i|} + \frac{|u|}{|y_i^i|}|y_1^i|  \leq |a_i| +
\frac{1}{2}|y_1^i| \leq \frac{1}{2}|y_1^i| + \frac{1}{2}|y_1^i| = |y_1^i|. \]
And thus, since $|u| \leq |y_i^i|$ and $|v| \leq |y_1^i|$, it follows that
$|u|$, $|v| \leq (y_1^i)^2 + (y_i^i)^2$.

Since $a_i$ is a factor of $y_1^i$ and $y_i^i$, there exist integers $b_1$ and
$b_i$ such that $y_1^i = a_ib_1$ and $y_i^i = a_ib_i$.  Put $z=b_1$ and
$w=-b_i$, then  \[ a_i = uy_1^i + vy_i^i = a_i(ub_1 + vb_i) = a_i(uz - vw). \]
Therefore the matrix $\startm u & v \\ w & z
\finishm$ is in $SL(2,\mathbb{Z})$.  Then, \[ \startm u & v \\ w & z \finishm\startm y_1^i \\ y_i^i
\finishm = \startm uy_1^i + vy_i^i \\ -b_iy_1^i + b_1y_i^i
\finishm. \] Note that $-b_iy_1^i + b_1y_i^i = -b_ia_ib_1 + b_1a_ib_i =
0$.  It follows that \[ \startm u & v \\ w & z
\finishm \startm y_1^i \\ y_i^i \finishm = \startm a_i
\\ 0 \finishm. \] The integers $|z|$ and $|w|$ are bounded by the same
polynomial as $|u|$ and $|v|$.  We have \[ |z| = \frac{|y_1^i|}{|a_i|} \leq
|y_1^i| \leq (y_1^i)^2 + (y_i^i)^2, \] and, similarly, $|w| \leq (y_1^i)^2 +
(y_i^i)^2$.

Let $\gamma_i$ be the matrix $\startm u & v \\ w & z \finishm \in
SL^{1,i}(d,\mathbb{Z})$.   Then \[\gamma_i \vecy^i =
\gamma_i\gamma_{i-1}\cdots\gamma_2\gamma_1\mathbf{x} = (a_i,
0,\ldots,0,y_{i+1}^i,\ldots,y_d^i)^T.\] By our construction, $y_i^i = x_i$, so
\begin{align*} a_i & = \gcd(y_1^i,y_i^i) \\ & = \gcd
(\gcd(y_1^{i-1},y_{i-1}^{i-1}),x_i) \\ & = \gcd(y_1^{i-1},x_{i-1},x_i) \\ & =
\cdots \\ & = \gcd(x_1',x_2,\ldots,x_i), \end{align*} where $x_1'$ is the first
component of $\gamma_1\mathbf{x}$. In particular, \[ a_d  =
\gcd(x_1',x_2,\ldots,x_d)  = \gcd(x_1,x_2,\ldots,x_d)  =  1.\]  Thus we obtain
$\gamma_1,\gamma_2,\ldots,\gamma_d$ such that $\gamma_d  \cdots
\gamma_2 \gamma_1 \mathbf{x} = (1,0,\ldots,0)^T$.

We have shown that the entries of $\gamma_i$, for $2 \leq i \leq d$, are
bounded by the polynomial $(y_1^i)^2 + (y_i^i)^2 = (\gcd(y_1^{i-1},
y_{i-1}^{i-1}))^2 + x_i^2$.  Now,  \begin{align*}
(\gcd(y_1^{i-1},y_{i-1}^{i-1}))^2 + x_i^2 & = (\gcd(x_1',x_2,\ldots,x_{i-1}))^2
+ x_i^2 \\ & \leq ((x_1')^2 + x_2^2 + \cdots + x_{i-1}^2 )^2 + x_i^2 \\ & \leq
(2(x_1^2 + x_2^2 + \cdots + x_d^2))^2 + x_i^2 \\ & \leq 5(x_1^2 + x_2^2 +
\cdots + x_d^2)^2. \end{align*} Let $p$ be the polynomial \[
p(x_1,x_2,\ldots,x_d) = 5(x_1^2 + x_2^2 + \cdots + x_d^2)^2. \] Then $p$ is a
fixed polynomial in the components of $\mathbf{x}$, and the entries of each
$\gamma_i$, for $2 \leq i \leq d$, are bounded by $p$. The matrix $\gamma_1$ is
either the identity, or $E_{1j}(1)$ for some $j$, so in either case the entries
of $\gamma_1$ are bounded by $1$. The entries of $\gamma_1$ are thus bounded by
the polynomial $p$ as well. \end{proof}

\begin{corollary} \label{product1} Let $\gamma \in \sldz$.  Then $\gamma$ may
be written as a product \[\gamma = \delta_1 \delta_2 \cdots \delta_{d^2},\]
where each $\delta_i$ belongs to some $SL^{s,t}(2, \mathbb{Z}) \subseteq
\sldz$.  Moreover, there is a fixed polynomial $P$ in the entries of $\gamma$,
such that, for $1 \leq i \leq d^2$, and for $1 \leq k,l\leq d$, the $(k,l)$
entry of $\delta_i$ satisfies \[ |(\delta_i)_{kl}| \leq
|P(\gamma_{11},\gamma_{12},\ldots,\gamma_{dd})|.\]
\end{corollary}

\begin{proof} We prove this corollary by induction on $d$.  The idea is to
multiply $\gamma$ by appropriate matrices on the left to obtain a matrix
$\gamma'$ with first column $(1,0,\ldots,0)^T$.  Then, we multiply $\gamma'$ by
appropriate matrices on the right to obtain a matrix $\gamma''$ with first
column $(1,0,\ldots,0)^T$ and first row $(1,0,\ldots,0)$.  The inductive
assumption is then applied to $\gamma''$.

If $d=2$ then we may let $\delta_1 = \gamma$, and for $2 \leq i \leq d^2$, let
$\delta_i$ be the identity matrix.   For $d \geq 3$, let $\mathbf{x}  =
(x_1,x_2,\ldots,x_d)^T\in \mathbb{Z}^d$ be the first column of the matrix
$\gamma$.  Since $\mathbf{x}$ is a column of a matrix in $SL(d,\mathbb{Z})$,
$\mathbf{x}$ is unimodular.  This is because the determinant of $\gamma$ may be
calculated by expanding along its first column, so we may obtain the equation \[
m_1x_1 + m_2x_2 + \cdots + m_dx_d = 1, \] where the $m_i$, for $1 \leq i \leq d$,
are integers.  Thus, by Lemma~\ref{e1_transpose}, there exist matrices
$\gamma_1,\gamma_2,\ldots,\gamma_d$ belonging to various copies of
$SL(2,\mathbb{Z})$ in $SL(d,\mathbb{Z})$ such that
\[\gamma_d\cdots\gamma_2\gamma_1\mathbf{x} = (1,0,\ldots,0)^T.\] Hence,
\[\gamma' = \gamma_d\cdots\gamma_2\gamma_1\gamma\] is an element of
$SL(d,\mathbb{Z})$ which has as its first column $(1,0,\ldots,0)^T$.  Also by
Lemma~\ref{e1_transpose}, each of the matrices $\gamma_i$ has entries bounded
by the fixed polynomial $p$ in the components $x_j$, where $1 \leq j \leq d$.  Since
$x_j = \gamma_{j1}$, which is an entry in the first column of the matrix
$\gamma$, we may now consider $p$ to be a fixed polynomial in the entries of
$\gamma$.  Therefore, each matrix $\gamma_i$ has entries bounded by a fixed
polynomial $p$ in the entries of $\gamma$.

Let the first row of $\gamma'$ be $(1,y_2,y_3,\ldots,y_d)$.  Multiply $\gamma'$
on the right by $\gamma'_2\gamma'_3\cdots\gamma'_d$, where for $2 \leq j \leq
d$, the matrix $\gamma'_j$ is $\startm 1 & -y_j \\ 0 & 1 \finishm$ in
$SL^{1,j}(2,\mathbb{Z})$.  This results in a matrix \[\gamma''=
\gamma_d\cdots\gamma_2\gamma_1\gamma\gamma'_2\gamma'_3\cdots\gamma'_d\] with
first column $(1,0,\ldots,0)^T$ and first row $(1,0,\ldots,0)$.  We may thus
consider $\gamma''$ to belong to $SL(d-1,\mathbb{Z})$ embedded in the lower
right-hand corner of \sldz.

The entries in the matrices of the form $\gamma'_j$, for $2 \leq j \leq d$, are
bounded by the polynomial $y_j^2 + 1$.  Now, if $\gamma_1$ is the identity, and
$\gamma_j = \startm u & v \\ w & z \finishm \in SL^{1,j}(d,\mathbb{Z})$, then
$y_j = u\gamma_{1j} + v\gamma_{jj}$. If $\gamma_1 = E_{1k}(1)$ for some $k \not
= 1$, and $\gamma_j = \startm u & v \\ w & z \finishm \in
SL^{1,j}(d,\mathbb{Z})$, then $y_j = u(\gamma_{1j} + \gamma_{kj}) +
v\gamma_{jj}$.  In either case, the entries in the matrices of the form
$\gamma'_j$, for $2 \leq j \leq d$, are bounded by \begin{align*} y_j^2 + 1 &
\leq (|u|(|\gamma_{1j}| + |\gamma_{kj}|) + |v||\gamma_{jj}|)^2 + 1 \\ & \leq
(p(\gamma) (\gamma_{1j}^2 + \gamma_{kj}^2 + \gamma_{jj}^2))^2 + 1 \\ & \leq
(p(\gamma)q(\gamma))^2 + 1, \end{align*} where $p(\gamma)$ is the fixed polynomial
$p$ in the entries of $\gamma$, and $q(\gamma)$ is the polynomial $q$ in the
entries of $\gamma$ given by \[ q(\gamma_{11},\gamma_{12},\ldots,\gamma_{dd}) =
2(\gamma_{12}^2  + \gamma_{13}^2 + \cdots + \gamma_{1d}^2 + \gamma_{22}^2 +
\gamma_{33}^2 + \cdots + \gamma_{dd}^2). \] Thus the entries of the matrices of
the form $\gamma'_j$, for $2 \leq j \leq d$, are bounded by a fixed polynomial
$(p(\gamma)q(\gamma))^2 + 1$ in the entries of $\gamma$.

Applying the inductive hypothesis to the matrix $\gamma''$, we may write
$\gamma''$ as a product $\gamma'' = \delta_1\delta_2\cdots\delta_{(d-1)^2}$,
where for $1 \leq k \leq (d-1)^2$, each $\delta_k$ belongs to some
$SL^{s,t}(2,\mathbb{Z})$, and each $\delta_k$ has entries bounded by a fixed
polynomial in the entries of the $(d-1)\times(d-1)$ submatrix in the lower
right-hand corner of $\gamma''$.  So we may consider each $\delta_k$ to have
entries bounded by a fixed polynomial in the entries of $\gamma''$.  Now,
$\gamma''$ is the product of the matrix $\gamma$, and matrices of the form
$\gamma_i$ and $\gamma_j'$.  So, each entry of $\gamma''$ is some polynomial in
the entries of $\gamma$, the $\gamma_i$ and the $\gamma_j'$. But the matrices
$\gamma$, $\gamma_i$ and $\gamma_j'$ all have entries bounded by fixed
polynomials in the entries of $\gamma$.  This means the entries of $\gamma''$
are bounded by a (higher degree) fixed polynomial in the entries of $\gamma$.
So we can find a fixed polynomial $P$ in the entries of $\gamma$ which is
sufficiently large to bound the entries of the $\gamma_i$, the $\gamma_j'$ and
the $\delta_k$.

Rearranging to make $\gamma$ the subject, we now have \[\gamma =
\gamma^{-1}_1\gamma^{-1}_2\cdots\gamma^{-1}_d
\delta_1\delta_2\cdots\delta_{(d-1)^2}
(\gamma'_d)^{-1}\cdots(\gamma'_3)^{-1}(\gamma'_2)^{-1}.\] The total number of
matrices on the right-hand side of this expression is $d^2$.  The inverses
$\gamma_i^{-1}$ and $(\gamma_j')^{-1}$ have entries bounded by the same fixed
polynomial as for $\gamma_i$ and $\gamma_j'$.  This is because the inverse of
$\gamma_i = \startm u & v \\ w & z \finishm$ is $\startm z & -v \\ -w & u
\finishm$, and the inverse of $\gamma'_j = \startm 1 & -y_i \\ 0 & 1 \finishm$
is $\startm 1 & y_i \\ 0 & 1 \finishm$.  In both cases, the bounds on the
entries are unchanged. We conclude that $\gamma$ may be expressed as a product
of the required form. \end{proof}

\subsection{Factorisation of $\delta \in P\sldz$}\label{factorsltwoz}

As in Subsection~\ref{factorsldz}, the results established here are in terms of
the group $\sldz$, rather than $P\sldz$.  Propositions~\ref{Eij_is_U1}
and~\ref{product2}, and Lemma~\ref{dU_lemma}, are all used to prove
Corollary~\ref{word_log}. Corollary~\ref{word_log} says that there exists a
constant $K_6 > 0$ such that, for all elements $\delta \in
SL^{s,t}(2,\Z)\subseteq\sldz$, we have $d_W(1,\delta) \leq K_6 \log \|
\delta\|$.  The proofs of Propositions~\ref{Eij_is_U1} and~\ref{product2} are
long, each involving several lemmas.

We say that $\gamma \in \sldz$ is a U1-element of $\sldz$, or just $\gamma$ is
U1, if there exists a constant $C_\gamma > 0$ such that, for all
integers $n$, \[ d_W(1,\gamma^n) \leq C_\gamma\log (|n| + 1). \]

\begin{propn} \label{Eij_is_U1} The matrix $\gamma = E_{ij}(1)$ (where $1 \leq
i \not = j \leq d$) is  U1-element of $SL(d,\mathbb{Z})$ for $d \geq 3$.
\end{propn}


\begin{proof} First, we show why it suffices to prove this result for just the
case where $\gamma$ is $E_{13}(1) \in SL(3,\mathbb{Z})$.  Then, we identify a
subgroup $V$ of $\slthreer$ with the vector space $\R^2$, and identify a
discrete subgroup of $V$ with the lattice $\Z^2$.  The matrix $E_{13}(1)$ is in
this discrete subgroup.  We construct a set $S \subseteq \Z^2$ such that every
point of $\Z^2$ is at a bounded distance from some point of $S$.  For a vector
$\vecv \in S$, we find a bound on the norm of $\vecv$, and a bound on the  word
length of the element of $\slthreer$ which corresponds to $\vecv$.   This
allows us to bound the word length of those elements of $\slthreer$ which are
identified with the lattice $\Z^2$.  To make the proof of this proposition more
digestible, most of these steps are presented as lemmas.

To show why it suffices to prove this proposition for just the case where
$\gamma$ is $E_{13}(1) \in SL(3,\mathbb{Z})$, let $\gamma$ and $\delta$ be
elements of $SL(d, \mathbb{Z})$, and suppose that $\gamma$ is U1.   Then, there
exist constants $C_\gamma$ and $C_\delta$ such that, for all non-zero integers
$n$, \begin{align*} d_W(1,(\delta\gamma\delta^{-1})^n) & =
d_W(1,\delta\gamma^n\delta^{-1}) \\ & \leq  d_W(1,\gamma^n) + 2d_W(1,\delta)
\\  & \leq  C_\gamma \log (|n| + 1) + 2d_W(1, \delta) \\ & \leq  2(C_\gamma +
2d_W(1,\delta))\log (|n| + 1) \\ & \leq  C_\delta \log (|n| + 1) . \end{align*}
If $n = 0$ then both sides are equal to $0$. This shows that if $\gamma$ is
U1,  so are all conjugates of $\gamma$.

We now show that all elements $E_{ij}(1)$, where $1 \leq i \not = j \leq d$, are
conjugate to $E_{13}(1)$ in $SL(d, \mathbb{Z})$.  We proceed by induction on
$d$.  If $d = 3$, then each matrix $E_{ij}(1)$, for $1 \leq i \not = j \leq 3$,
may be written as $\delta E_{13}(1) \delta^{-1}$ for some $\delta \in SL(3,
\mathbb{Z})$:

$ \begin{array}{ll} \mbox{for } E_{12}(1), \, \delta = \startm 1 & 0 & 0 \\ 0 &
0 & 1 \\ 0 & -1 & 0 \finishm; & \mbox{for } E_{21}(1),\, \delta = \startm  0 &
0 & 1 \\ 1 & 0 & 0 \\ 0 & 1 & 0 \finishm; \\ \rule{0cm}{1cm} \mbox{for }
E_{23}(1), \,  \delta = \startm 0 & -1 & 0 \\ 1 & 0 & 0 \\ 0 & 0 & 1 \finishm ;
& \mbox{for }  E_{31}(1),\, \delta = \startm 0 & 0 & 1 \\ 0 & -1 & 0 \\ 1 & 0 &
0 \finishm; \\ \rule{0cm}{1cm} \mbox{for }  E_{32}(1),\, \delta = \startm 0 & 1
& 0 \\ 0 & 0 & 1 \\ 1 & 0 & 0 \finishm. & \end{array} $ \vspace{3mm}

\noindent The inductive assumption is that for some $d \geq 3$, all of the elements
$E_{ij}(1)$, where $1 \leq i \not = j \leq d$, are conjugate to $E_{13}(1)$ in
$SL(d,\mathbb{Z})$.  Consider $E_{ij}(1) \in SL(d + 1,\mathbb{Z})$.  If $i$ and
$j$ are both less than or equal to $d$, then $E_{ij}(1)$ may be regarded as an
element of \sldz\ embedded in the upper left-hand corner of
$SL(d+1,\mathbb{Z})$, so by assumption $E_{ij}(1)$ is conjugate to
$E_{13}(1)$.  Otherwise, one of $i$ or $j$ must be equal to $d + 1$; we show
that $E_{ij}(1)$ is conjugate to some matrix $E_{st}(1)$ with $1 \leq s \not =
t \leq d$, hence is conjugate to $E_{13}(1)$. For $i=d+1$ and
$1 \leq j \leq d-1$, and for $j=d+1$ and $1 \leq i \leq d-1$, the matrix by which
we conjugate is \[ \startm I_{d-1} & 0 & 0 \\ 0 & 0 & 1 \\ 0 & 1 & 0 \finishm,
\] where $I_{d-1}$ is the $(d-1)\times(d-1)$ identity matrix.  For
$(i,j)=(d+1,d)$ and $(i,j)=(d,d+1)$, we conjugate by \[ \startm I_{d-2} & 0 & 0
& 0 \\ 0 & 0 & 0 & 1 \\ 0 & 0 & 1 & 0 \\ 0 & 1 & 0 & 0 \finishm. \]

It remains to show that if the matrix $E_{13}(1)$ is U1 in $SL(3,\mathbb{Z})$,
then it is U1 in \sldz. For this, let $\Sigma_3$ be a finite set of generators
for $SL(3,\mathbb{Z})$. By embedding $SL(3, \mathbb{Z})$ in the top left-hand
corner of \sldz, we may extend $\Sigma_3$ to $\Sigma$, a finite set of
generators for \sldz.  Let $\gamma$ be a U1 element of $SL(3,\mathbb{Z})$,
embedded in \sldz. The sets of generators $\Sigma_3$ and $\Sigma$ induce word
distance functions which we will denote by $d_\Sigma^3$ and $d_\Sigma$,
respectively. Then, there exists a constant $C_\gamma > 0$ such that, for all
integers $n$, \[ d_\Sigma(1,\gamma^n)  \leq  d_\Sigma^3(1,\gamma^n)
\leq  C_\gamma \log (|n|+1). \] Therefore $d_W(1,\gamma^n) \leq C_\gamma\log
(|n|+1)$, and $\gamma$ is U1 in $SL(d,\mathbb{Z})$.


We have now shown that it suffices to prove that the matrix $E_{13}(1)$ is U1
in $SL(3,\mathbb{Z})$. In the next part of the proof, we consider a special
element and some subgroups of $\slthreez$. Let \[ A = \startm 2 & 1 & 0 \\ 1 &
1 & 0 \\ 0 & 0 & 1 \finishm, \hspace{8mm}  V = \left\{ v(x,y) = \startm 1 & 0 &
x \\ 0 & 1 & y \\ 0 & 0 & 1 \finishm :  x, y \in \R \right\}, \] \[ L = V \cap
SL(3,\mathbb{Z}) = \left\{  \startm 1 & 0 & k \\ 0 & 1 & l \\ 0 & 0 & 1
\finishm : k, l \in \Z \right\}. \] Multiplication in $V$ is given by
$v(x,y)v(x',y') = v(x+x',y+y')$, so $V$ is an Abelian group.  The group $V$ is
isomorphic to the vector space $\mathbb{R}^2$, where the latter is considered
as an Abelian group under addition of vectors.  We may, then, identify
$L\subseteq V$ with the integer lattice $\mathbb{Z}^2 \subseteq \mathbb{R}^2$.
The action of $A$ on $V$ by conjugation takes $v(x,y)$ to $v(2x + y, x+y)$,
which corresponds under this isomorphism to the vector in $\mathbb{R}^2$ given
by \[ \startm 2x + y \\ x + y\finishm  = \startm 2 & 1 \\ 1 & 1 \finishm\startm
x \\ y \finishm. \] Hence, the action of $A$ on $V$ by conjugation corresponds
to the linear action of $A' = \startm 2 & 1 \\ 1 & 1 \finishm$ on
$\mathbb{R}^2$.  Moreover, this action preserves $L$, and so preserves the
integer lattice $\mathbb{Z}^2$.

We will now consider $V$ as the vector space $\mathbb{R}^2$ and write the
multiplication in it as addition.  We also identify $L$ with $\Z^2$. For
$\mathbf{v} \in V$ we will write $A\mathbf{v}$ meaning the conjugation of
$\mathbf{v}$ by $A$, that is, multiplication by the matrix $A'$, which we will
write as $A$. The eigenvectors of the $2 \times 2$ matrix $A$ are then $\lambda
= (3 + \sqrt{5})/2$ and $\lambda^{-1} = (3 - \sqrt{5})/2$.  Note that $2 <
\lambda < 3$ and $\lambda^{-1} < 1$. We will denote by $\mathbf{v}_1$ and
$\mathbf{v}_2$ fixed eigenvectors corresponding to the eigenvalues $\lambda$
and $\lambda^{-1}$ respectively.  For convenience, we may choose $\vecv_1$ and
$\vecv_2$ so that $\| \vecv_1 \| = \| \vecv_2 \| = 1$.  We denote by $W_1$ and
$W_2$ the corresponding  eigenspaces of $V$.    Let $\mathbf{y}_0$ be the fixed
vector $\startm 1 \\ 1 \finishm \in V$. We define the following subsets of $L$:
\begin{align*} S_1 & =  \left\{  \pm \sum_{i=1}^m a_iA^i\mathbf{y}_0
:  a_i \in \{ 0, 1, 2 \}, \ a_m \not = 0 \right\} \\ S_2 & = \left\{
 \pm \sum_{j=1}^n b_jA^{-j}\mathbf{y}_0 : b_j \in \{ 0, 1, 2 \} ,\
b_n \not = 0 \right\}. \end{align*} The relationship between the sets $S_1$ and
$S_2$ and the eigenspaces $W_1$ and $W_2$ is described in the following lemma.


\begin{lemma} \label{syndetic} For $k = 1,2$, each point of the eigenspace
$W_k$ is within a bounded distance of some point of the set $S_k$. \end{lemma}

\begin{proof} Let  \[ \mathbf{v} = \pm \sum_{i=1}^m a_iA^i\mathbf{y}_0 \] be a
point of $S_1$.  Since $\mathbf{v}_1$ and $\mathbf{v}_2$ span $\mathbb{R}^2$,
we have $\mathbf{y}_0 = \alpha \mathbf{v}_1 + \beta \mathbf{v}_2$ for fixed
real numbers $\alpha$ and $\beta$. Then \[ A^i \mathbf{y}_0 = \alpha \lambda^i
\mathbf{v}_1 + \beta \lambda^{-i}  \mathbf{v}_2, \] and so, \[\mathbf{v} = \pm
\left(  \alpha \sum_{i=1}^m a_i \lambda^i  \mathbf{v}_1 +  \beta \sum_{i=1}^m
a_i \lambda^{-i}  \mathbf{v}_2 \right). \]  The sum $\sum_{i=1}^m a_i
\lambda^{-i}$  is bounded.  This  is because $a_i < \lambda$ and $\lambda^{-1}
< 1$, so \[   \sum_{i=1}^m a_i \lambda^{-i}    <   \sum_{i=0}^\infty
\lambda^{-i} =  \frac{1}{1 - \lambda^{-1}}.\]

Let $\mathbf{u}$ belong to $W_1$.   Then $\mathbf{u} = \alpha'\mathbf{v}_1$ for
some $\alpha' \in \R$. Suppose first that $\alpha' \geq 0$, and consider an
element $\mathbf{v}$ in $S_1$ with the form \[ \mathbf{v} = \alpha \sum_{i=1}^m
a_i \lambda^i \mathbf{v}_1 + \beta \sum_{i=1}^m a_i \lambda^{-i} \mathbf{v}_2.
\] Note here that $\alpha > 0$.  The distance from $\mathbf{v}$ to $\mathbf{u}$
is \begin{align*} \|\mathbf{v} - \mathbf{u}\| & =  \left\| \left( \alpha
\sum_{i=1}^m a_i \lambda^i  - \alpha' \right)\mathbf{v}_1 + \left( \beta
\sum_{i=1}^m a_i \lambda^{-i} \right)\mathbf{v}_2 \right\| \\ & \leq
\alpha\left| \sum_{i=1}^m a_i \lambda^i - \frac{\alpha'}{\alpha} \right| +
|\beta|\sum_{i=1}^m a_i \lambda^{-i} .\end{align*} Since $\alpha$ and $\beta$
are constant, and  $\sum_{i=1}^m a_i \lambda^{-i}$ is bounded, it now suffices
to prove that for all $a \geq 0$ there exists an integer $m \geq 1$, and a set
of integers $\{ a_i \}_{i=1}^m$, with $a_i \in \{0,1,2\}$ and $a_m \not = 0$,
such that \begin{equation} \label{syndet_eqn} \left| \sum_{i=1}^m a_i \lambda^i
- a \right| \leq \lambda.\end{equation}

If $0 \leq a < 1$ then $m = 1$ and $a_1 = 1$ will do. For $a \geq 1$, the proof
is by induction on $k$, where \[ \lambda^k \leq a < \lambda^{k+1}. \] When $k
=0$, we have $1 \leq a < \lambda$, so take $m = 1$ and  $a_1=1$.  Now assume
inductively that, for all $l$ such that $0 \leq l < k$ and all $a$ such that
$\lambda^l \leq a < \lambda^{l+1}$, there exist integers $\{ a_i \}_{i=1}^l $,
with $a_i \in \{ 0,1,2\}$ and $a_l \not = 0$, such that~(\ref{syndet_eqn})
holds. Suppose that $\lambda^k \leq a < \lambda^{k+1}$. Then, since $2 <
\lambda < 3$, either  $\lambda^k \leq a < 2\lambda^k$, or  $2\lambda^k \leq a <
3\lambda^k$.  Put $a_k = 1$ in the first case  and $a_k = 2$ in the second
case.  Then \[ a - a_k\lambda^k < \lambda^k, \] so we have \[ \lambda^l \leq a
- a_k\lambda^k < \lambda^{l+1} \] for some $l < k$.  By the inductive
assumption, there exists a set of integers $\{ a_i \}_{i=1}^l$, with $a_i \in
\{0,1,2\}$ and $a_l \not = 0$,  such that \[ \left| \sum_{i=1}^l a_i\lambda^i -
(a - a_k\lambda^k)\right| \leq \lambda. \] So \[ \left| \sum_{i=1}^l
a_i\lambda^i + a_k\lambda^k  - a \right| \leq \lambda, \] and if we put $a_i =
0$ for $l < i < k$, we obtain a set $\{ a_i \}_{i=1}^k$, with $a_i \in \{0,1,2\}$
and $a_k \not = 0$, satisfying~(\ref{syndet_eqn}).  Thus $\| \mathbf{v} -
\mathbf{u} \|$ is bounded.

When $\alpha' < 0$ the proof is almost the same; consider \[ \mathbf{v} =
-\alpha\sum_{i=1}^m a_i \lambda^i  \mathbf{v}_1 - \beta\sum_{i=1}^m a_i
\lambda^{-1} \mathbf{v}_2. \]  To show that every point of $W_2$ is within a
bounded distance of some point of $S_2$, the proof is similar. \end{proof}


Let $S$ be the set $S_1 + S_2$.  Note that $S$ is a subset of the integer
lattice $L$.

\begin{corollary}\label{SinL}  Each point of $L$ is within a bounded distance of
some point in $S$.\end{corollary}

\begin{proof} The eigenvectors $\vecv_1$ and $\vecv_2$ are linearly
independent, so $\R^2 = W_1 + W_2$.  Each point of $L$, then, can be
written as $\mathbf{w}_1 + \mathbf{w}_2$ for some $\mathbf{w}_1 \in W_1$ and
$\mathbf{w}_2 \in W_2$.  By Lemma~\ref{syndetic}, $\vecw_1$ is within a bounded
distance of some point of $S_1$, and $\vecw_2$ is within a bounded distance of
some point of $S_2$.  Thus $\mathbf{w}_1 + \mathbf{w}_2$ is within a bounded
distance of some point of $S_1 + S_2$.    \end{proof}


We now establish a lower bound on the norm of vectors in $S$. Let
$\mathbf{v}$ be in $S$.  Then $\mathbf{v}$ may be expressed in the form
\begin{equation} \label{formv} \mathbf{v} = \pm_1 \sum_{i=1}^m a_i A^i
\mathbf{y}_0 \pm_2 \sum_{j=1}^n b_j A^{-j} \mathbf{y}_0, \end{equation} where
$\pm_1$ and $\pm_2$ are independent of each other. We first prove a general
result about linearly independent vectors.


\begin{lemma} \label{combbound}  Let $\mathbf{w}_1$ and $\mathbf{w}_2$ be
linearly independent vectors in the vector space $\mathbb{R}^m$, where $m \geq 2$.
Then there exists an $M > 0$ such that, for all $\mu_1$, $\mu_2 \in
\mathbb{R}$, \[ \| \mu_1 \mathbf{w}_1 + \mu_2 \mathbf{w}_2 \| \geq M \max \{ \|
\mu_1 \mathbf{w}_1 \|, \| \mu_2 \mathbf{w}_2 \| \}. \] \end{lemma}

\begin{proof} The two-dimensional vector space spanned by $\mathbf{w}_1$ and
$\mathbf{w}_2$ may be identified with the complex plane $\mathbb{C}$. Indeed we
may, without loss of generality, identify $\mathbf{w}_1$ with the point
$1e^{i0} = 1$, and $\mathbf{w}_2$ with the point $re^{i\theta}$ where $r > 0$
and $0 < \theta < \pi$.  Let $M = \sin \theta$. Then, if $\mu_1 = 0$, there is
nothing more to prove. Otherwise, we may assume without loss of generality that
$\mu_1 > 0$, and so we are trying to prove that \begin{equation}
\label{start_ineq} |\mu_1 + \mu_2re^{i\theta}| \geq M \max\{\mu_1,|\mu_2|r\}
\end{equation} for all $\mu_1 > 0$ and all $\mu_2 \in \mathbb{R}$.  Divide
through this inequality by $\mu_1$ and put $R = \mu_2r/\mu_1$.  To
establish~(\ref{start_ineq}), it now suffices to prove that for all $R \in
\mathbb{R}$, \[  | 1 + Re^{i\theta}| \geq M \max \{1,|R|\}. \] Now, $ |1 +
Re^{i\theta}|^2  = 1 + 2R\cos \theta + R^2$. We have \begin{align*} (\cos
\theta + R)^2 & \geq 0 \\ \Rightarrow \cos^2\theta + 2R\cos \theta + R^2 & \geq
0 \\ \Rightarrow 1 + 2R\cos \theta + R^2 & \geq \sin^2 \theta.\end{align*}
Also,\begin{align*} (1 + R \cos \theta )^2 & \geq 0 \\ \Rightarrow 1  + 2R\cos
\theta + R^2\cos^2 \theta & \geq 0 \\ \Rightarrow 1 + 2R\cos \theta + R^2 &
\geq R^2 \sin^2 \theta.\end{align*} It follows that \[ 1 + 2R \cos \theta + R^2
\geq M^2 \max \{ 1, R^2\}. \] Take square roots of both sides to complete the
proof.\end{proof}

Note that the value of $M$ in Lemma~\ref{combbound} depends only on the angle
between the vectors $\mathbf{w}_1$ and $\mathbf{w}_2$. When we use
Lemma~\ref{combbound} in the next proof, the value of $M$ will thus depend only
on the angle between the fixed eigenvectors  $\mathbf{v}_1$ and
$\mathbf{v}_2$.  So $M$, as far as it is used here, is a constant.


\begin{lemma} \label{lengthbounds} There exists a constant $c_1 > 0$ such that,
for any $\mathbf{v}$ in $S$, if $k$ is the maximum of $m$ and $n$ where
$\vecv$ has  the form~(\ref{formv}), then $c_1 \lambda^k \leq \| \mathbf{v}
\|$. \end{lemma}

\begin{proof} As in Lemma~\ref{syndetic},  $\mathbf{y}_0 = \alpha \mathbf{v}_1
+ \beta \mathbf{v}_2$ for fixed $\alpha$ and $\beta$.  Then, we get
\begin{align*} \mathbf{v} & =  \pm_1  \sum_{i=1}^m a_i \left(
\alpha\lambda^i\mathbf{v}_1 + \beta\lambda^{-i}\mathbf{v}_2 \right) \pm_2
\sum_{j=1}^n b_j \left( \alpha\lambda^{-j} \mathbf{v}_1 + \beta\lambda^j
\mathbf{v}_2\right) \\ & =  \left(\pm_1 \sum_{i=1}^m a_i \lambda^i \pm_2
\sum_{j=1}^n b_j \lambda^{-j}\right)\alpha \mathbf{v}_1 + \left(\pm_1
\sum_{i=1}^m a_i \lambda^{-i} \pm_2 \sum_{j=1}^n b_j \lambda^j\right)\beta
\mathbf{v}_2. \end{align*} Let \[\mu_1 = \left(\pm_1 \sum_{i=1}^m a_i \lambda^i
\pm_2 \sum_{j=1}^n b_j \lambda^{-j}\right)\alpha,\ \mbox{  and  } \ \mu_2 =
\left(\pm_1 \sum_{i=1}^m a_i \lambda^{-i} \pm_2 \sum_{j=1}^n b_j
\lambda^j\right)\beta.\] By Lemma~\ref{combbound}, there exists a constant $M >
0$ such that \[ \| \mathbf{v} \|  \geq  M \max \{ \| \mu_1\mathbf{v}_1\|,
\|\mu_2 \mathbf{v}_2 \| \} = M \max \{ |\mu_1|, |\mu_2|\}.  \] Suppose first
that $m \geq n$, and consider  \[ \frac{|\mu_1|}{\alpha} = \left| \pm_1
\sum_{i=1}^m a_i \lambda^i \pm_2 \sum_{j=1}^n b_j \lambda^{-j}\right|. \]
Since $a_m \not = 0$, $\sum_{i=1}^m a_i \lambda^i \geq \lambda^m$.  For
the other sum, we have \[ 0 < \sum_{j=1}^n b_j \lambda^{-j} <
\sum_{j=0}^\infty \lambda^{-j} = \frac{1}{1 - \lambda^{-1}} =
\frac{\lambda}{\lambda - 1} < 2. \] It follows that \begin{align*} \left| \pm_1
\sum_{i=1}^m a_i \lambda^i \pm_2 \sum_{j=1}^n b_j \lambda^{-j}\right| & \geq
\left| \left| \sum_{i=1}^m a_i \lambda^i \right| - \left| \sum_{j=1}^n b_j
\lambda^{-j}\right| \right| \\ & = \sum_{i=1}^m a_i \lambda^i -  \sum_{j=1}^n
b_j \lambda^{-j}\\ & \geq \lambda^m - 2 \\ & \geq \frac{1}{5}\lambda^m.
\end{align*} Thus, $\| \vecv \| \geq M |\mu_1| \geq c_1 \lambda^m$ where
$c_1 > 0$ is constant.  In the
case $n > m$, the proof is similar, with a lower bound being established for
$|\mu_2|$. \end{proof}


Corollary~\ref{SinL} and Lemma~\ref{lengthbounds} will be used in the conclusion of the
proof of Proposition~\ref{Eij_is_U1}.  Before this, we return to the
interpretation of $V$ as a subgroup of $SL(3,\mathbb{Z})$.  If $\mathbf{v} \in
S$ is considered as an element of $SL(3,\mathbb{Z})$, then its word length is
bounded as follows.


\begin{lemma} \label{wordbound} There exists a constant $c_2 > 0$ such that,
for any $\mathbf{v}$ in $S$, if $k$ is the maximum of $m$ and $n$ where $\vecv$
has the form~(\ref{formv}), then $\mathbf{v}$ may be written as a word in the
elements $A$ and $Y_0$ of $SL(3,\mathbb{Z})$ of length at most $c_2 k$.
\end{lemma}

\begin{proof} The expression \[\mathbf{v} = \pm_1 \sum_{i=1}^m a_i A^i
\mathbf{y}_0 \pm_2 \sum_{j=1}^n b_j A^{-j} \mathbf{y}_0\] may be rewritten as
\begin{align*} \mathbf{v}  =   & \pm_1 A(a_1\mathbf{y}_0 + A(a_2\mathbf{y}_0 +
\cdots + A(a_{m-1}\mathbf{y}_0 + A a_m \mathbf{y}_0)\cdots)) \\ & \mbox{} \pm_2
A^{-1}(b_1\mathbf{y}_0 + A^{-1}(b_2 \mathbf{y}_0 + \cdots +
A^{-1}(b_{n-1}\mathbf{y}_0 + A^{-1} b_n \mathbf{y}_0)\cdots)). \end{align*}
Using the correspondence between $\mathbb{R}^2$ and $SL(3,\mathbb{Z})$, we may
translate this expression for $\mathbf{v}$ into a word written multiplicatively
in the elements \[ A =  \startm 2 & 1 & 0 \\ 1 & 1 & 0 \\ 0 & 0 & 1 \finishm
\hspace{5mm}\mbox{and} \hspace{5mm}
Y_0 = \startm 1 & 0 & 1 \\ 0 & 1 & 1 \\ 0 & 0 & 1 \finishm \] of
$SL(3,\mathbb{Z})$.   Terms of the form $a_i\mathbf{y}_0$ have length at most
2, since $0 \leq a_i \leq 2$.  There are $m$ terms of this form. Each
multiplication by $A$ in $\mathbb{R}^2$ corresponds to conjugation by $A$ in
$SL(3,\mathbb{Z})$, so we add 2 to the length for each of the $m$
multiplications by $A$. Similar calculations for the second half of the
expression show that $\mathbf{v}$ may be written as a word of length at most $
2m + 2m + 2n + 2n  =  4(m+n) \leq  4(2k) = C_2 k.$ \end{proof}


We are now in a position to conclude the proof of Proposition~\ref{Eij_is_U1}.
Let $\Sigma$ be a finite set of generators of $SL(3, \mathbb{Z})$ which
includes the elements $A$, $Y_0$, $E_{13}(1)$ and $E_{23}(1)$, and let
$d_\Sigma$ be the word distance function induced by the set $\Sigma$. Let
$\mathbf{v}$ be in the set $S$. Using logarithm laws to rearrange the inequality $c_1
\lambda^k \leq \| \mathbf{v} \| $ from Lemma~\ref{lengthbounds}, we obtain \[
k  \leq  \frac{\log \| \mathbf{v} \| - \log c_1}{\log \lambda}.\] By
Lemma~\ref{wordbound} and this inequality, the word distance function $d_\Sigma$ satisfies \[
d_\Sigma(1,\mathbf{v})  \leq  c_2k  \leq  \frac{c_2}{\log \lambda}(\log \|
\mathbf{v} \| - \log c_1) \leq B \log \| \vecv \|, \] for some constant $B>0$.

Next, by Corollary~\ref{SinL}, there exists some constant $b_1 > 0$ such that,
for all $\vecw \in L$, we can find a vector $\mathbf{v} \in S$ so that
$\|\mathbf{w} - \mathbf{v}\| \leq b_1$.  Let $\mathbf{w}$ be any element of
$L$, and write $\vecv'  = \vecw - \vecv$, so that $\|\mathbf{v}'\| \leq b_1$.
We now prove that \[d_\Sigma(1,\mathbf{w}) \leq B' \log (\| \mathbf{w}\|+1), \]
for some constant $B' > 0$.    Since $\mathbf{w}$ and $\mathbf{v}$ are elements
of the lattice $L$, so is $\mathbf{v}'$.  We may then translate the vector
$\mathbf{w} = \vecv + \vecv'\in\R^2$ into a word, written multiplicatively, in
terms of elements of $\slthreez$.  We use the elements $A$ and $Y_0$ for
$\mathbf{v}$, and $E_{13}(1)$ and $E_{23}(1)$ for $\mathbf{v}'$. Since the
value of $\|\mathbf{v}'\|$ is bounded, this multiplicative expression for
$\mathbf{v}'$ is bounded in length, by say $b_2 > 0$. Thus \[
d_\Sigma(1,\mathbf{w})  \leq  d_\Sigma(1,\mathbf{v}) + d_\Sigma(1,\mathbf{v}')
\leq  B \log \| \mathbf{v} \| + b_2. \] By the triangle inequality, $\|
\mathbf{v} \|  =  \| \mathbf{w} - \mathbf{v}' \| \leq  \| \mathbf{w} \| + \|
\mathbf{v}' \|$,  so \begin{align*} B \log \| \mathbf{v} \|  + b_2  &\leq   B
\log (\| \mathbf{w} \| + \| \mathbf{v}' \| ) + b_2  \\ & \leq  B\log (\|
\mathbf{w} \| + b_1) + b_2 \\ & \leq B'\log(\|\vecw\|+1),\end{align*} for some constant $B' >
0$.

Let $n$ be any integer. Then the matrix $(E_{13}(1))^n$ is in the lattice $L$.
It corresponds to the vector $\vecw = \startm n \\ 0 \finishm$, which has norm
equal to $|n|$.  So, for some $C>0$, and all integers $n$,  \[ d_W(1,
(E_{13}(1))^n) \leq C \log (|n|+1).  \] Thus, $E_{13}(1)$ is  U1. \end{proof}

\begin{propn} \label{product2} There exist constants $c_3, c_4 > 0$ so that any
$\delta \in SL(2,\mathbb{Z})$ may be written as a word of the form $\delta =
r_1 r_2 \cdots r_m$, where \begin{enumerate}\item each $r_j$, for $1 \leq j \leq
m$, is either \matrixv, or $\startm 1 & n_j \\ 0 & 1 \finishm$ for some non-zero
integer $n_j$, \item we have \[ c_3
d_U(p_0, \delta p_0) \leq \sum_{j=1}^m f(r_j) \leq c_4 d_U(p_0, \delta p_0), \]
where \[ f(r) = \left\{ \begin{array}{ll} 1 & \mbox{if} \ r  = \matrixv \\[4mm]
\log (|n| + 1) & \mbox{if} \ r = \startm 1 & n \\ 0 & 1 \finishm, \end{array}
\right. \] the point $p_0$ is  $2i$ in $U^2$ the upper half-plane, and
$d_U(\cdot,\cdot)$ is the hyperbolic distance function on $U^2$.
\end{enumerate} \end{propn}

\begin{proof} Let $T$ be the hyperbolic triangle with vertices at $\pm
\frac{1}{2} + \frac{\sqrt{3}}{2}i$ and $\infty$.  Let $u(n)$ denote the matrix
$\startm 1 & n \\ 0 & 1 \finishm$, where $n \in \Z$, and let $v$ be the matrix
\matrixv.

Given $\delta \in \sltwoz$, we consider the geodesic segment $[p_0, \delta
p_0]$ connecting the points $p_0$ and $\delta p_0$ in $U^2$. This geodesic
segment passes through a sequence of tiles $\left\{ \theta_k T
\right\}_{k=0}^n$, such that $\theta_0 = 1$ and $\theta_n = \delta$, and for $1
\leq k \leq n$, tile $\theta_{k-1} T$ is adjacent to tile $\theta_k T$. Let
$s_k = \theta^{-1}_{k-1}\theta_k$. Then the product of the elements $s_k$ is
\[s_1s_2 \cdots s_n = \theta^{-1}_0\theta_1\theta^{-1}_1\theta_2 \cdots
\theta^{-1}_{n-1}\theta_n = 1.\theta_n = \delta.\] Since successive tiles in
the sequence $\left\{ \theta_k T \right\}_{k=0}^n$ are adjacent, we have, for
$1 \leq k \leq n$, \[ s_k = \theta^{-1}_{k-1}\theta_k \in \left\{ u(1), u(-1),
v \right\}. \] Notice also that since $v^2$ is the identity in $P\sldz$,  there
are no two consecutive matrices $s_{k-1}$ and $s_k$ which are both equal to
$v$.

We now, by reducing the product of the $s_k$, construct a new product
$r_1r_2\cdots r_m$ which is also equal to $\delta$, a corresponding subsequence
of tiles $\{ \theta_j T\}_{j=0}^m$, and a set of points $\{ p_j \}_{j=0}^m$ in
$U^2$.  The basic idea is to multiply together all adjacent matrices of the
form $u(1)$ or $u(-1)$. Let $s_i$ be the first matrix in the product
$s_1s_2\cdots s_n$ which has the form $s_i = u(1)$ or $s_i = u(-1)$.  For all
$k$ such that $1 \leq k < i$, let $r_k = s_k$, and let the point $p_k$ be $p_k
= \theta_k  p_0$. There is an integer $l > 0$ such that $s_i = u(n_i), s_{i+1}
= u(n_{i+1}), \ldots, s_{i+l} = u(n_{i+l})$, but $s_{i + l + 1}= v$.  Let $r_i$
be the matrix $r_i = u(n_i')$, where $n_i' = n_i + n_{i+1} + \cdots +
n_{i+l}$.  Then, \[ r_i  = s_i s_{i+1} \cdots s_{i+l}  = \theta_{i-1}^{-1}
\theta_i \theta_i^{-1} \theta_{i+1} \cdots \theta_{i + l - 1}^{-1} \theta_{i+l}
= \theta_{i-1}^{-1}\theta_{i+l}. \] Let $p_i$ be the point $\theta_{i+l}  p_0$.
Relabel $\theta_{i+l}$ as $\theta_i$, so that $r_i = \theta_{i-1}^{-1}
\theta_i$ and $p_i = \theta_i  p_0$. We now have a product \[ r_1 \cdots r_i
s_{i+l+1} \cdots s_n = \delta, \] a sequence of tiles $\{ \theta_j
T\}_{j=0}^i$, and a set of points $\{ p_j \}_{j=0}^i$, such that, for $1 \leq j
\leq i$, $r_j = \theta_{j-1}^{-1} \theta_j$ and $p_j = \theta_j p_0$.
Iterate the above reduction process on the product $s_{i+l+1} \cdots s_n$ until
a product \[ r_1 r_2 \cdots r_m = \delta, \] a sequence $\{ \theta_j
T\}_{j=0}^m$ and a set $\{ p_j \}_{j=0}^m$ is obtained, where, for $1 \leq j
\leq m$, $r_j = \theta_{j-1}^{-1}\theta_j$ and  $p_j = \theta_j p_0$.
In this product, no two consecutive matrices $r_{j-1}$ and $r_j$ have the
respective forms $ r_{j-1} = u(n_{j-1})$ and $r_j =
u( n_j)$, where $n_{j-1}$ and $n_j$ are non-zero integers.

We now consider the path $\alpha$ in $U^2$ which is the following union of geodesic
segments: \[ \alpha = [p_0, p_1] \cup [p_1, p_2] \cup \cdots \cup
[p_{m-1},p_m]. \] The length of this path is
\begin{equation}\label{plength}\length(\alpha) = \sum_{j=1}^m d_U(p_0, r_j
p_0),\end{equation} since
\[ \sum_{j=1}^m d_U(p_{j-1},p_j)  = \sum_{j=1}^m d_U(\theta_{j-1}
p_0,\theta_j p_0)  = \sum_{j=1}^m d_U(p_0, \theta_{j-1}^{-1}\theta_j   p_0)
 = \sum_{j=1}^m d_U(p_0, r_j  p_0). \]   The following lemma
shows that, roughly speaking, the length of $\alpha$ is not too different from
the length of the geodesic joining the points $p_0$ and $\delta p_0$.

\begin{lemma} \label{K1lemma}There exist constants $K_1, K_2 > 0$ such that \[ K_1
d_U(p_0,\delta p_0) \leq \length(\alpha) \leq K_2
d_U(p_0,\delta p_0). \] \end{lemma}

\begin{proof} Since the geodesic segment $[p_0, \delta p_0]$ is the shortest
path from $p_0$ to $\delta p_0$, we have immediately that \[ d_U(p_0, \delta
p_0) \leq \length(\alpha). \]

To prove the other required inequality, for $0 \leq j \leq m$ let $q_j$ be the
point on the geodesic segment $[p_0, \delta p_0]$ which is closest to the point
$p_j$. Note that  $q_0 = p_0$ and $q_m = p_m$, and also that the point $q_j$
need not be in the same tile of the tesselation as $p_j$.   We claim that, for
$1 \leq j \leq m$,    \begin{equation}\label{qorder} d_U(p_0,q_{j-1}) <
d_U(p_0,q_j). \end{equation} This claim implies that the geodesic segments
$[q_{j-1},q_j]$ are disjoint except for their endpoints, and so
\begin{equation}\label{qlength} d_U(p_0, \delta p_0) =  \sum_{j=1}^m
d_U(q_{j-1},q_j).\end{equation} In proving this claim, we will establish that
the sequence $\{q_j\}_{j=0}^m$ is well-defined.

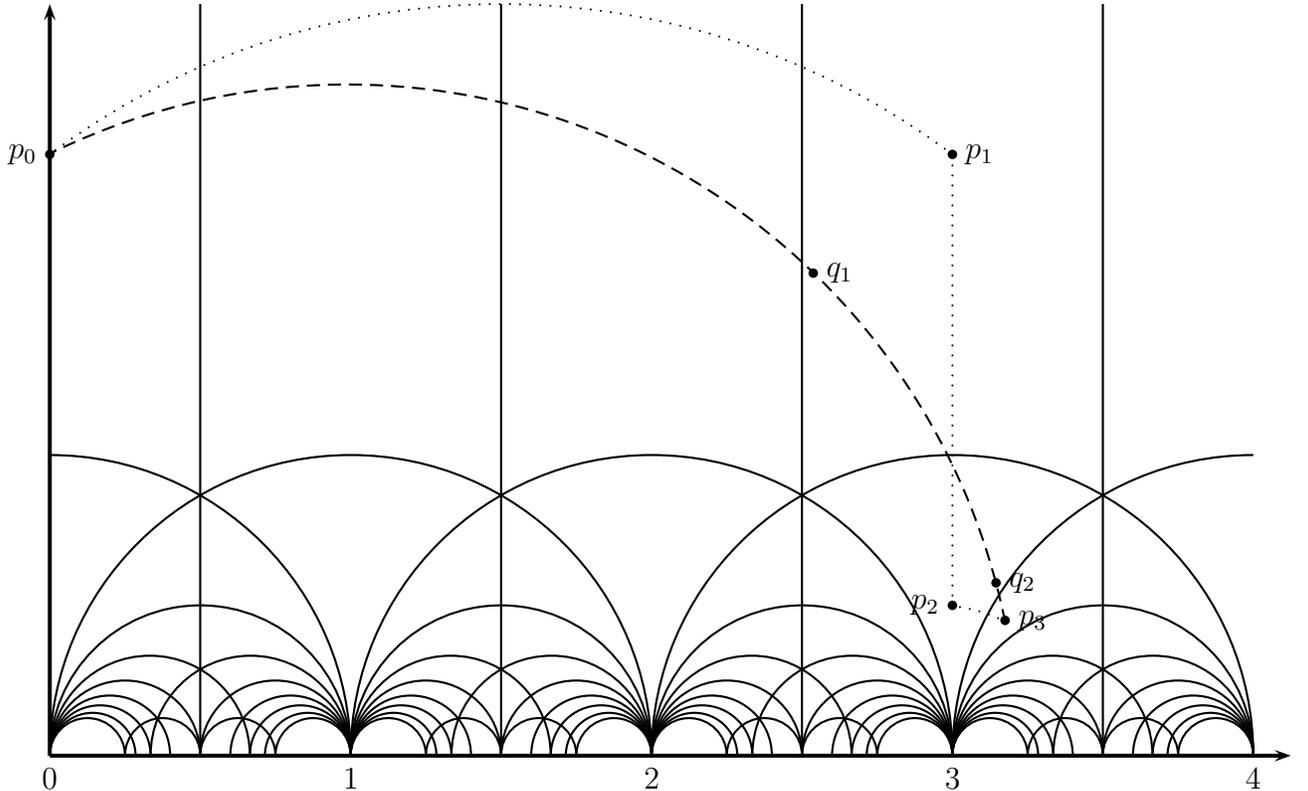
\begin{center}
\begin{pspicture}(0,0)(16,10)
\psline[linewidth=1.5pt]{->}(0,0)(16.5,0)
\psline[linewidth=1.5pt]{->}(0,0)(0,10)
\psline(2,0)(2,10)
\psline(6,0)(6,10)
\psline(10,0)(10,10)
\psline(14,0)(14,10)
\uput[d](0,0){0}
\uput[d](4,0){1}
\uput[d](8,0){2}
\uput[d](12,0){3}
\uput[d](16,0){4}
\psarc(0,0){4}{0}{90}
\psarc(4,0){4}{0}{180}
\psarc(8,0){4}{0}{180}
\psarc(12,0){4}{0}{180}
\psarc(16,0){4}{90}{180}
\psarc(2,0){2}{0}{180}
\psarc(6,0){2}{0}{180}
\psarc(10,0){2}{0}{180}
\psarc(14,0){2}{0}{180}
\psarc(1,0){1}{0}{180}
\psarc(3,0){1}{0}{180}
\psarc(5,0){1}{0}{180}
\psarc(7,0){1}{0}{180}
\psarc(9,0){1}{0}{180}
\psarc(11,0){1}{0}{180}
\psarc(13,0){1}{0}{180}
\psarc(15,0){1}{0}{180}
\psarc(1.33,0){1.33}{0}{180}
\psarc(5.33,0){1.33}{0}{180}
\psarc(9.33,0){1.33}{0}{180}
\psarc(13.33,0){1.33}{0}{180}
\psarc(2.67,0){1.33}{0}{180}
\psarc(6.67,0){1.33}{0}{180}
\psarc(10.67,0){1.33}{0}{180}
\psarc(14.67,0){1.33}{0}{180}
\psarc(0.8,0){0.8}{0}{180}
\psarc(4.8,0){0.8}{0}{180}
\psarc(8.8,0){0.8}{0}{180}
\psarc(12.8,0){0.8}{0}{180}
\psarc(3.2,0){0.8}{0}{180}
\psarc(7.2,0){0.8}{0}{180}
\psarc(11.2,0){0.8}{0}{180}
\psarc(15.2,0){0.8}{0}{180}
\psarc(0.67,0){0.67}{0}{180}
\psarc(4.67,0){0.67}{0}{180}
\psarc(8.67,0){0.67}{0}{180}
\psarc(12.67,0){0.67}{0}{180}
\psarc(3.33,0){0.67}{0}{180}
\psarc(7.33,0){0.67}{0}{180}
\psarc(11.33,0){0.67}{0}{180}
\psarc(15.33,0){0.67}{0}{180}
\psarc(0.57,0){0.57}{0}{180}
\psarc(4.57,0){0.57}{0}{180}
\psarc(8.57,0){0.57}{0}{180}
\psarc(12.57,0){0.57}{0}{180}
\psarc(3.43,0){0.57}{0}{180}
\psarc(7.43,0){0.57}{0}{180}
\psarc(11.43,0){0.57}{0}{180}
\psarc(15.43,0){0.57}{0}{180}
\psarc(0.5,0){0.5}{0}{180}
\psarc(1.5,0){0.5}{0}{180}
\psarc(2.5,0){0.5}{0}{180}
\psarc(3.5,0){0.5}{0}{180}
\psarc(4.5,0){0.5}{0}{180}
\psarc(5.5,0){0.5}{0}{180}
\psarc(6.5,0){0.5}{0}{180}
\psarc(7.5,0){0.5}{0}{180}
\psarc(8.5,0){0.5}{0}{180}
\psarc(9.5,0){0.5}{0}{180}
\psarc(10.5,0){0.5}{0}{180}
\psarc(11.5,0){0.5}{0}{180}
\psarc(12.5,0){0.5}{0}{180}
\psarc(13.5,0){0.5}{0}{180}
\psarc(14.5,0){0.5}{0}{180}
\psarc(15.5,0){0.5}{0}{180}
\psarc[linestyle=dashed](3.95,0){8.93}{11.5}{116.5}
\uput[l](0,8){$p_0$}
\psdot*(0,8)
\uput[r](12,8){$p_1$}
\psdot*(12,8)
\uput[l](12,2){$p_2$}
\psdot*(12,2)
\uput[r](12.7,1.8){$p_3$}
\psdot*(12.7,1.8)
\uput[r](10.15,6.42){$q_1$}
\psdot*(10.15,6.42)
\uput[r](12.58,2.3){$q_2$}
\psdot*(12.58,2.3)
\psarc[linestyle=dotted](6,0){10}{53}{127}
\psline[linestyle=dotted](12,8)(12,2)
\psarc[linestyle=dotted](11.8,0){2}{62}{85.5}
\end{pspicture}
\end{center}

\medskip

Let $L$ be the geodesic which contains the geodesic segment $[p_0,\delta
p_0]$.  As $\theta_{j-1}^{-1}$ is an isometry,~(\ref{qorder}) holds
if and only if $d_U(\theta_{j-1}^{-1}p_0,\theta_{j-1}^{-1}q_{j-1}) <
d_U(\theta_{j-1}^{-1}p_0,\theta_{j-1}^{-1}q_j)$.  Since $q_{j-1}$ is the
closest point on $L$ to $p_{j-1}$, the point $\theta_{j-1}^{-1}q_{j-1}$ is the
closest point on the geodesic $\theta_{j-1}^{-1}L$ to $\theta_{j-1}^{-1}p_{j-1}
= p_0$.    Similarly, $\theta_{j-1}^{-1}q_j$ is the closest point on the
geodesic $\theta_{j-1}^{-1}L$ to the point $\theta_{j-1}^{-1}p_j$.  Now, the
geodesic $L$ intersects the tiles $\theta_j T$ and $\theta_{j-1} T$, and the
tiles $\theta_j T$ and $\theta_{j-1} T$ are adjacent.  So, the geodesic
$\theta_{j-1}^{-1}L$ intersects the tiles $\theta_{j-1}^{-1}\theta_j T$ and
$T$, and these tiles are also adjacent.   By symmetry in the imaginary axis, we
need only consider the cases $\theta_{j-1}^{-1}\theta_j T = u(1)T$ and
$\theta_{j-1}^{-1}\theta_j T = vT$.  Thus the point $\theta_{j-1}^{-1}p_j$ is either,
respectively, $1 + 2i$, or $i/2$. To simplify notation, we now write $L$ for
$\theta_{j-1}^{-1}L$, $p$ for $p_0$, $p'$ for $\theta_{j-1}^{-1}p_j$, $q$ for
$\theta_{j-1}^{-1}p_{j-1}$ and $q'$ for $\theta_{j-1}^{-1}q_j$.

The first case is when $p' = 1+2i$.  Here, the geodesic $L$ is a semicircle
orthogonal to the real axis, which intersects the tiles $T$ and $u(1)T$.  We may
parametrise $z \in L$ by \[ z = a + re^{i\phi}, \] where $a \in \R$, $r > 0$
and $\phi \in (0,\pi)$.  Since $L$ meets $T$ and $u(1)T$,
\begin{equation}\label{qorderboundonr} r \geq \left|a - \left(\frac{1}{2} +
\frac{\sqrt{3}}{2}i\right)\right| \hspace{5mm}\Rightarrow\hspace{5mm} r^2 \geq
a^2 - a + 1. \end{equation} The distance from the point $n+2i$, where $n \in
\Z$, to a point $z = a + re^{i\phi} \in L$ is  \[ d_U(n+2i,z) =
\cosh^{-1}\left( 1 + \frac{(n-a - r\cos\phi)^2 + (2 -
r\sin\phi)^2}{4r\sin\phi}\right). \] For fixed $n$, $a$ and $r$, this function
has a unique minimum when \begin{equation}\label{minangle} \cos\phi =
\frac{2(n-a)r}{(n-a)^2 + r^2 + 4}. \end{equation} Thus, the points $q$ and $q'$
are the unique closest points on $L$ to $p$ and $p'$ respectively. So, in
this first case, the choice of the $q_j$ is well-defined.  If $q = a +
re^{i\phi}$ and $q' = a + re^{i\phi'}$, to prove~(\ref{qorder}) it now suffices
to prove that $\phi'<\phi$. By~(\ref{minangle}), \[ \cos \phi =
\frac{-2ar}{a^2 + r^2 + 4} \hspace{5mm}\mbox{and}\hspace{5mm}\cos \phi' =
\frac{2(1-a)r}{(1-a)^2 + r^2 + 4}. \] If $0 \leq a \leq 1$, then by considering
the signs of $\cos\phi'$ and
$\cos\phi$, it follows that $\phi' < \phi$.  If $a > 1$ or $a < 0$
then $\cos\phi$ and $\cos\phi'$ have the same sign.  Since cosine is
decreasing on $(0,\pi)$, it suffices to prove that  $ \cos\phi < \cos\phi'
$.  The inequality
\[
\frac{-2ar}{a^2 + r^2 + 4} < \frac{2(1-a)r}{(1-a)^2 + r^2 + 4}
\]
holds if and only if $a^2 < a + r^2 + 4$.  By~(\ref{qorderboundonr}), we have
\[
a + r^2 + 4  > a + r^2 - 1 \geq a^2,
\]
and so $\cos\phi < \cos\phi'$.  This completes the proof of the case $p' = 1
+ 2i$.

The second case is when $p' = i/2$.  Here, $L$ intersects the tiles $T$ and
$vT$, so is either a vertical line, or a semicircle orthogonal to the real
axis. If $L$ is a vertical line, to prove~(\ref{qorder}) it suffices to prove
that $\im(q') < \im(q)$.  We may parametrise $z \in L$ by $z = a + iy$ for some
$a \in [-1/2,1/2]$.  Then the unique closest point on $L$ to $p = 2i$ is $q = a
+ i\sqrt{a^2 + 4}$, and the unique closest point on $L$ to $p' = i/2$ is $q' =
a + i(\sqrt{4a^2 + 1}/2)$.  Therefore the $q_j$ are well-defined here.  It is
straightforward to prove that $\sqrt{4a^2 + 1}/2 < \sqrt{a^2 + 4}$.  If $L$ is
a semicircle, we again parametrise $z \in L$ by $z = a + re^{i\phi}$.  Without
loss of generality, $a > 0$ (if $a = 0$ then $L$ cannot intersect both $T$ and
$vT$).  Again, the closest points to $p$ and $p'$ are unique; we find that the unique minimum
distances from $p$ and $p'$ are achieved when, respectively, \[ \cos \phi = \frac{-2ar}{a^2 +
r^2 + 4},\hspace{5mm}\mbox{and}\hspace{5mm}\cos \phi' = \frac{-8ar}{4a^2 + 4r^2
+ 1}.  \] So $q = a +
re^{i\phi}$ and $q' = a + re^{i\phi'}$.  To establish~(\ref{qorder}), it suffices to prove that $\phi' >
\phi$. Since $\cos\phi < 0$ and $\cos\phi' < 0$, we wish to show that $\cos\phi
> \cos\phi'$; this is again a straightforward inequality.

We have now proved~(\ref{qorder}), and so~(\ref{qlength}) holds.   The plan for
the rest of the proof of Lemma~\ref{K1lemma} is as follows.  After considering
two special cases, we show that when $r_j = u(n_j)$, there exists a constant
$K_2'$ such that $d_U(p_{j-1},p_j) \leq K_2' d_U(q_{j-1},q_j)$. This involves
parametrisation of a geodesic which is a semicircle with centre $a$ and radius
$r$.  The case  $j = 1$ is treated first, then for $j > 1$ we find the distance
between $q_{j-1}$ and $q_j$ in terms of $a$, $r$ and $n_j$, and thus show that
there exists a constant $K$ so that $\cosh^{-1}(1 + n_j^2 / K) \leq
d_U(q_{j-1},q_j)$.  From the explicit formula for $d_U$, it follows that
$d_U(p_{j-1},p_j) \leq K_2' d_U(q_{j-1},q_j)$.  Then, we consider the $r_j$
which are equal to $v$.  We show that \[ d_U(p_{j-1},p_j) + d_U(p_j,p_{j+1})
\leq 4K_2' (d_U(q_{j-1},q_j) + d_U(q_j,q_{j+1})). \] Adding together these
kinds of inequalities completes the proof.

In the special case where $\delta$ is the identity, Lemma~\ref{K1lemma} is
trivial.  The other special case is when $\delta = v$.  Here, we have
$d_U(p_0, p_1) = d_U(q_0, q_1)$, so the path $\alpha$ is identical to the
geodesic segment $[p_0,\delta p_0]$.  Provided $K_2 \geq 1$, there is nothing
more to prove in this case.  So if, later, we arrive at a value of $K_2$ which
is less than 1, we will take $K_2 = 1$ instead.

For all other $\delta \in \sltwoz$, at least one of the $r_j$ must have the
form $u(n_j)$, where $n_j$ is a non-zero integer.  We now show that when $r_j =
u(n_j)$, there exists a constant $K_2'$ such that \begin{equation}\label{case1}
d_U(p_{j-1},p_j) \leq K_2' d_U(q_{j-1},q_j). \end{equation} We have \[
d_U(p_{j-1},p_j)  = d_U(\theta_{j-1} p_0, \theta_j  p_0)  = d_U(p_0,
\theta_{j-1}^{-1}\theta_j p_0) = d_U(p_0,r_j p_0).\] Also, \[ d_U(q_{j-1},q_j)
= d_U(\theta_{j-1}^{-1}  q_{j-1},\theta_{j-1}^{-1}  q_j). \]  Since in both
cases we have applied the isometry $\theta_{j-1}^{-1}$, we may now, without
loss of generality, compare the distances between the points $p_0$ and $r_j
p_0$, and $\theta_{j-1}^{-1} q_{j-1}$ and $\theta_{j-1}^{-1}  q_j$. By symmetry
in the imaginary axis, we may in addition assume that $n_j \geq 1$.

We are considering geodesics $L$ which contain the image of $[p_0, \delta p_0]$
under $\theta_{j-1}^{-1}$, and which, since $p_0 \in T$ and $r_j p_0 \in
u(n_j)T$,  intersect the tiles $T$ and $u(n_j)T$.   We may, as above,
parametrise $z \in L$ by $ z = a + re^{i\phi}$,  where $a \in \R$, $r > 0$
and $0 < \phi < \pi$.

If $j = 1$, then the isometry $\theta_{j-1}^{-1}$ is just the identity.  Thus,
as $p_0 = q_0 = 2i$, the geodesic $L$ passes through the point $2i$ in the
tile $T$. Now, $L$ either intersects the arc of the circle $|z| = 1$ between
the points $z = \frac{1}{2} + \frac{\sqrt{3}}{2}i$ and $z = -\frac{1}{2} +
\frac{\sqrt{3}}{2}i$, or $L$ intersects the ray $R$ given by \[ R = \left\{ z
\in U^2 : \re(z) = -\frac{1}{2}\ \ \mbox{and } \im(z) \geq \frac{\sqrt{3}}{2}
\right\}. \] The case where $L$ intersects the arc may be dealt with as below,
where $\theta_{j-1}^{-1}$ is not the identity.  We now  suppose $L$ intersects
the ray $R$ in the point $w$. Then, the point $a$, which is the centre of the
semicircle $L$, is where the perpendicular bisector of the points $w$ and $2i$
meets the real axis. The value of $a$ thus has a maximum, when $w =
-\frac{1}{2} + \frac{\sqrt{3}}{2}i$. It follows that there are only finitely
many positive integers $n_1$ such that $p_1 = n_1 + 2i$. Let $N$ be the maximum
possible value of $n_1$. By the reductions carried out above to form the
product $r_1 r_2 \cdots r_m$, as the angle $\phi$ increases, the geodesic $L$
must pass straight from a tile \emph{not} of the form $u(m_1)T$, where $m_1 \in
\Z$, into the tile $u(n_1)T$.   Now, for $1 \leq n_1 \leq N$, the smallest
value of $d_U(2i,q_1)$ occurs when $L$ passes through the vertex $\left(n_1 -
\frac{1}{2}\right) + \frac{\sqrt{3}}{2}i$ of the tile $u(n_1)T$.  Since there are only
finitely many cases to consider, we can thus find a $K_2'$ sufficiently large
so that, for each instance $1 \leq n_1 \leq N$, \[ d_U(p_0,p_1) = d_U(2i,n_1 +
2i) \leq K_2' d_U(2i,q_1) = d_U(q_0,q_1). \]

For all $j > 1$, to simplify notation, write $n = n_j$, $p = p_0$, $p' = n +
p_0$, $q = \theta_{j-1}^{-1}  q_{j-1}$ and $q' = \theta_{j-1}^{-1} q_j$.  We
first establish bounds on the values of $a$ and $r$ in terms of $n$.  By the
reductions carried out above to form the product $r_1 r_2 \cdots r_m$, as the
angle $\phi$ increases, the geodesic $L$ must pass straight from a tile
\emph{not} of the form $u(m_1)T$ into $u(n)T$, and later on from $T$ straight
into a tile \emph{not} of the form $u(m_2)T$, where $m_1, m_2 \in \mathbb{Z}$.

\begin{center}
\begin{pspicture}(-2,0)(10,6)
\psline[linewidth=1.5pt]{<->}(-2,0)(10,0) \psline[linewidth=1.5pt]{->}(0,0)(0,6)
\psline(8,0)(8,0.1) \psline(3.7,0)(3.7,0.1)
\psline(-1,1.73)(-1,6)
\psline(1,1.73)(1,6) \psline(7,1.73)(7,6) \psline(9,1.73)(9,6)
\psarc(0,0){2}{60}{120} \psarc(8,0){2}{60}{120}
\psarc[linestyle=dashed](3.7,0){4.5}{0}{180} \uput[d](0,0){0} \uput[d](8,0){$n$}
\uput[d](3.7,0){$a$} \uput[u](5.7,4.1){$L$}
\psdot*(0,4) \uput[l](0,4){$p$}
\psdot*(8,4) \uput[l](8,4){$p'$}
\psdot*(1,1.73)
\psdot*(-1,1.73)\psdot*(9,1.73)\psdot*(7,1.73)
\uput[r](1,1.73){$\frac{1}{2} + \frac{\sqrt{3}}{2}i$}
\uput[l](-1,1.73){$-\frac{1}{2} + \frac{\sqrt{3}}{2}i$}
\uput[r](9,1.73){$n + \frac{1}{2} + \frac{\sqrt{3}}{2}i$}
\uput[l](7,1.73){$n - \frac{1}{2} + \frac{\sqrt{3}}{2}i$}
\end{pspicture}
\end{center}

\medskip

\noindent Thus, the radius of $L$ is bounded by the
inequalities \begin{equation} \label{rbound2} \left| a - \left( n -  1/2 +
i\sqrt{3}/2 \right) \right| \leq r \leq \left| a - \left(n + 1/2 +i\sqrt{3}/2
\right) \right| \end{equation} and \begin{equation} \label{rbound1} \left| a -
\left( 1/2 + i\sqrt{3}/2 \right) \right| \leq r \leq \left| a - \left( - 1/2 +
i\sqrt{3}/2 \right) \right|. \end{equation}
 By squaring, expanding and combining~(\ref{rbound2}) and~(\ref{rbound1}),
we get \begin{equation} \label{abound} \frac{n-1}{2} \leq a \leq \frac{n+1}{2}.
\end{equation} Then, using~(\ref{rbound1}) and~(\ref{abound}), we obtain
\begin{equation} \label{rbound3} \frac{n^2 - 4n + 7}{4} \leq r^2 \leq \frac{n^2
+ 4n + 7}{4}. \end{equation}

Now, we find the hyperbolic distance between $q$ and $q'$ in terms of $a$, $r$
and $n$. Using~(\ref{minangle}), the closest point to  $p=2i$ on the geodesic $L$
is \[ q  = a + \frac{-2ar^2}{a^2 + r^2 + 4} + ir \frac{\sqrt{(a^2 + r^2 + 4 )^2
- 4a^2r^2}}{a^2 + r^2 + 4}, \] and the  closest point to $p' =
n+2i$ on the geodesic $L$ is \[ q'  = a + \frac{2(n-a)r^2}{(n-a)^2 + r^2 + 4}
+ ir \frac{\sqrt{((n-a)^2 + r^2 + 4 )^2 - 4(n-a)^2r^2}}{(n-a)^2 + r^2 + 4}. \]
Then, the distance between $q$ and $q'$ is, for fixed $n$, \[ d_U(q,q') = \cosh^{-1}\left( 1 +
\frac{F(a,r)}{G(a,r)}\right), \] where $F:\mathbb{R}^2 \rightarrow \mathbb{R}$
is the function \begin{align*} F(a,r)  = & \left|q - q'\right|^2 \\  = & \left(
\re(q) - \re(q')\right)^2 + \left( \im(q) - \im(q')\right)^2\\  = & \left(
\frac{2(n-a)r^2}{(n-a)^2 + r^2 + 4} + \frac{2ar^2}{a^2 + r^2 + 4} \right)^2 +
\\ & r^2\left( \frac{\sqrt{(a^2 + r^2 + 4)^2 - 4a^2r^2}}{a^2 + r^2 + 4} -
\frac{\sqrt{((n-a)^2 + r^2 + 4)^2 - 4r^2(n-a)^2 }}{(n-a)^2 + r^2 + 4}
\right)^2,\end{align*} and $G:\R^2 \rightarrow \R$ is \begin{align*} G(a,r) &
=  2\im(q)\im(q') \\ & =  2r^2\frac{\sqrt{((a^2 + r^2 + 4 )^2 - 4a^2r^2 )(
((n-a)^2 + r^2 + 4)^2 - 4r^2(n-a)^2 )}}{(a^2 + r^2 + 4 )((n-a)^2 + r^2 + 4)}.
\end{align*}

From the picture of the geodesic $L$ above,  it is clear that $\im(q)$ may equal
$\im(q')$.  To show that $F(a,r)$ is bounded below, we thus consider the
difference between $\re(q)$ and $\re(q')$.  We then show that $G(a,r)$ is
bounded above.  These facts together will prove that for all $n \geq 1$, and
all $a$ and $r$ satisfying~(\ref{abound}) and~(\ref{rbound3}) respectively,
\begin{equation} \label{fgbound}  \frac{n^2}{K} \leq \frac{F(a,r)}{G(a,r)},
\end{equation} where $K > 0$ is a constant. This implies a lower bound  on
$d_U(q,q')$.

From~(\ref{abound}), we have $a \leq (n+1)/2$, and $n - a \leq (n+1)/2$. Using
these bounds and~(\ref{rbound3}), \begin{align*} \frac{2(n-a)r^2}{(n-a)^2 + r^2
+ 4} + \frac{2ar^2}{a^2 + r^2 + 4} & \geq \frac{2(n-a)r^2 + 2ar^2}{(n+1)^2/4 +
(n^2 + 4n + 7)/4 + 4} \\ & = \frac{2nr^2}{n^2/2 + 3n/2 + 6} \\ & \geq
\frac{n(n^2 - 4n + 7)}{n^2 + 3n + 12}. \end{align*} The right-hand side of the
final line goes to $n$ as $n \rightarrow \infty$, so \[ F(a,r)  \geq  \left(
\frac{2(n-a)r^2}{(n-a)^2 + r^2 + 4} + \frac{2ar^2 }{a^2 + r^2 + 4} \right)^2
\geq \frac{n^2}{k}, \]  for some positive constant $k$. For $G$, by a similar
process using~(\ref{abound}) and~(\ref{rbound3}) repeatedly, we obtain \[
G(a,r) \leq \frac{(n^2 + 4n + 7)/2}{(n^2/2 - 3n/2 + 6)^2}(25n^2/4 + 33n/2 +
137/4). \] The function of $n$ on the right-hand side converges to a constant
as $n \rightarrow \infty$, so $G(a,r)$ is bounded above, by say $K/k$.  The
lower bound on $F(a,r)/G(a,r)$ at~(\ref{fgbound}) follows. Therefore, the
distance between the points  $q$ and $q'$ is bounded below: \[ \cosh^{-1}\left(
1 + \frac{n^2}{K} \right) \leq \cosh^{-1}\left( 1 + \frac{F(a,r)}{G(a,r)}
\right)   = d_U(q,q') . \]

We can find a constant $K_2' > 0$ such that for all $n \geq 1$, \[ \cosh^{-1}\left( 1 +
\frac{n^2}{8} \right) \leq K_2'
\cosh^{-1}\left( 1 + \frac{n^2}{K} \right)  . \]  Therefore,
$d_U(p,p') \leq K_2'd_U(q,q')$. We have established~(\ref{case1}).

We next consider the $r_j$ of the form $r_j = v$, where $j$ is strictly less
than $m$. Rather than comparing the distances
$d_U(q_{j-1},q_j)$ and $d_U(p_{j-1},p_j)$, we show instead that
\begin{equation}\label{risv}  d_U(p_{j-1},p_j) + d_U(p_j, p_{j+1}) \leq 4K_2'
\left( d_U(q_{j-1},q_j) + d_U(q_j, q_{j+1}) \right).\end{equation} When $r_j =
v$, we have \[ d_U(p_{j-1},p_j)  = d_U( p_0, r_j  p_0)  = d_U(2i, i/2). \] By
the construction of the product $r_1 r_2 \cdots r_m$, and the assumption that
$j < m$, if $r_j = v$ then $r_{j+1} = u(n_{j+1})$ for some non-zero integer
$n_{j+1}$. So, by~(\ref{case1}), we have \[ d_U(p_j,p_{j+1})\leq
K_2'd_U(q_j,q_{j+1}). \] Then, the following  proves~(\ref{risv}):
\begin{align*}d_U(p_{j-1},p_j) + d_U(p_j, p_{j+1}) & = d_U(2i,i/2) + d_U(p_j,
p_{j+1}) \\ & = \log 4 + d_U(p_j, p_{j+1}) \\ & \leq 3 \cosh^{-1}(1 + 1/8)  +
d_U(p_j, p_{j+1}) \\ & \leq 3 \cosh^{-1}(1 + n_{j+1}^2/8)  + d_U(p_j, p_{j+1})
\\ & = 3d_U(p_0, p_0 + n_{j+1}) + d_U(p_j, p_{j+1}) \\ & = 3d_U(p_0, r_{j+1}
p_0) + d_U(p_j, p_{j+1}) \\ & = 3d_U(p_j,p_{j+1}) + d_U(p_j, p_{j+1}) \\ & =
4d_U(p_j,p_{j+1}) \\ & \leq 4K_2' d_U(q_j,q_{j+1}) \\ & \leq 4K_2' \left(
d_U(q_{j-1},q_j) + d_U(q_j, q_{j+1}) \right).\end{align*}

The final case is when $j = m$ and $r_m = v$.  Then $r_{m-1}$ must be
$u(n_{m-1})$ for some non-zero integer $n_{m-1}$.
It can be proved, as for~(\ref{risv}), that \[  d_U(p_{m-2},p_{m-1}) +
d_U(p_{m-1},p_m) \leq 4K_2' \left(
d_U(q_{m-2},q_{m-1}) + d_U(q_{m-1}, q_{m}) \right). \]

Let the constant $K_2$ be $K_2 = 2 \times 4K_2'$. (Multiplication by 2 is
needed to cater for the possibility of counting $d_U(q_{m-2},q_{m-1})$ twice.)
Then, \[  \sum_{j=1}^{m} d_U (p_{j-1},p_j) \leq K_2 \sum_{j=1}^m
d_U(q_{j-1},q_j), \] and so by~(\ref{plength}) and~(\ref{qlength}), $\length(\alpha)
\leq K_2 d_U(p_0,\delta p_0)$. \end{proof}

The next lemma shows that the length of the path $\alpha$ is also roughly the
same as the sums of the values of $f$ at each of the terms $r_j$.

\begin{lemma} \label{K3lemma} There exist constants $K_3, K_4 > 0$ such that
\[K_3 \length(\alpha) \leq \sum_{j=1}^m f(r_j) \leq K_4 \length(\alpha).\]
\end{lemma}

\begin{proof} If $r_j = v$, then $d_U(p_0, r_j  p_0) = \log 4$ and $f(r_j)
= 1$.  If $r_j = u(n_j)$ for $n_j \not = 0$, then
$d_U(p_0, r_j  p_0) = \cosh^{-1}\left( 1 + n_j^2/8 \right)$ and $f(r_j) =
\log(|n_j| + 1)$.  Let the set $J$ be \[ J = \left\{ j \in\Z: 1 \leq j \leq m
\ \mbox{and }
r_j = v \right\}, \] and denote its
cardinality by $|J|$. Then \[ \length(\alpha)  = \sum_{j=1}^m d_U(p_0,r_j  p_0)
=  |J|\log 4 + \sum_{j \not \in J} \cosh^{-1}\left( 1 + n_j^2/8\right),\] and \[
\sum_{j=1}^m f(r_j)  = |J| + \sum_{j \not \in J} \log( |n_j| + 1).\]

Now, for all $n_j \not = 0$, \[  \cosh^{-1}(1
+ n_j^2/8) \leq 3\log(|n_j| + 1) \] and, also, \[   \log(|n_j| + 1) \leq
3 \cosh^{-1}(1 + n_j^2/8). \] It follows that \[ \frac{1}{3}\length(\alpha) \leq
\sum_{j=1}^m f(r_j) \leq 3 \length(\alpha).\]  \end{proof}

We may now conclude the proof of Proposition~\ref{product2}.  By
Lemmas~\ref{K1lemma} and~\ref{K3lemma}, there exist constants $K_1$, $K_2$,
$K_3$ and $K_4$ such that \[ K_3 K_1
d_U(p_0, \delta p_0)  \leq K_3 \length(\alpha)  \leq \sum_{j=1}^m f(r_j)  \leq
K_4 \length(\alpha)  \leq K_4K_2 d_U(p_0, \delta  p_0). \] \end{proof}


\begin{lemma} \label{dU_lemma} Let $p_0$ be the point $2i$ in $U^2$.  Then
there is a constant $K_5 > 0$ such that for any $\delta \in \sltwoz$,
\[d_U(p_0,\delta p_0) \leq K_5 \log \|\delta\|.\] \end{lemma}

\begin{proof} Let $\delta$ be in \sltwoz.  As shown in Lemma~\ref{sl2decomp},
there exist matrices $k_1$ and $k_2$ in \sotwor, and a matrix $a$ in \sltwor\
of the form $\startm s & 0 \\ 0 & s^{-1} \finishm$ where $s \geq 1$, such that
$\delta = k_1 a k_2$. Then, by Lemma~\ref{diagnorm}, $\|\delta\| = \| a\| = s$.
Since the group $\sotwor$ stabilises $i$, we have $d_U(i,\delta i) = d_U(i,
ai) = d_U(i, s^2i) =  2 \log s$.    Thus, \begin{align*} d_U(p_0,\delta p_0) & \leq d_U(2i,i) +
d_U(i, \delta i) +  d_U(\delta  i, \delta  (2i)) \\ & = 2d_U(2i,i)+
d_U(i,\delta  i) \\ & = 2 \log 2 + 2 \log s  \\ & = \log 4 + 2 \log \| \delta \| \\ & \leq K_5 \log \|
\delta \|, \end{align*} for some constant $K_5>0$. \end{proof}


\begin{corollary} \label{word_log} There exists a constant $K_6 > 0$ such that,  for all
$s$ and $t$ where $1 \leq s \not = t \leq d$, and all $\delta \in
SL^{s,t}(2,\mathbb{Z})\subseteq \sldz$,  \[ d_W(1,\delta) \leq K_6 \log \| \delta \|. \]
\end{corollary}

\begin{proof} Let $r_1 r_2 \cdots r_m$ be a word for $\delta$ as in
Proposition~\ref{product2}, where each $r_j$ belongs to
$SL^{s,t}(2,\mathbb{Z})$.  Let $\Sigma$ be a fixed finite set of generators of
$\sldz$, and let $d_\Sigma$ be the word distance function on $\sldz$ induced by
$\Sigma$.

First, we consider the terms $r_j$ which are of the form $\startm 1 & n_j \\ 0
& 1 \finishm$ for some non-zero integer $n_j$. Such elements may be written as
$r_j = (E_{st}(1))^{n_j}$. By Proposition~\ref{Eij_is_U1}, there exists a
constant $C_{st}$ such that \[ d_\Sigma(1,
r_j) \leq C_{st} \log (|n_j| + 1) = C_{st}f(r_j), \] where $f$ is the function
defined in the statement of Proposition~\ref{product2}.  Let $C > 0$ be the
maximum of the constants $C_{st}$ over the finitely many values of $s$ and $t$
such that  $1 \leq s \not = t \leq d$.  Then $d_\Sigma(1,r_j) \leq C f(r_j)$.

If the element $r_j$ is of the form $r_j = \matrixv$, then $r_j$  may be
written as a shortest word of fixed length $L_{st}$ in our generators $\Sigma$. For such
$r_j$, we have $f(r_j) = 1$.  Let $L>0$ be the maximum of the constants
$L_{st}$ over the finitely many values of $s$ and $t$ such that  $1 \leq s \not
= t \leq d$. Then $d_\Sigma(1,r_j) \leq L f(r_j)$.

Let $M$ be the maximum of $L$ and $C$. By Proposition~\ref{product2} and
Lemma~\ref{dU_lemma}, the word length of $\delta$ with respect to $\Sigma$  is
bounded as follows: \[ d_\Sigma(1,\delta)  \leq \sum_{j=1}^m d_\Sigma(1,r_j)
\leq M \sum_{j=1}^m f(r_j) \leq Mc_4 d_U(p_0, \delta p_0) \leq Mc_4 K_5 \log \|
\delta \|.\] Thus there exists a constant $K_6>0$ such that
$d_W(1,\delta) \leq K_6 \log \|\delta\|$. \end{proof}

\subsection{Bounds on the operator norm of a matrix}\label{opnorm}

\begin{lemma} \label{norm_ineqs}  There exist constants $c_5$, $c_6 > 0$ such
that, for any $d \times d$ real matrix $A = (a_{ij})$, \[c_5 \max_{i,j}
|a_{ij}| \leq \| A \| \leq c_6 \max_{i,j} |a_{ij}| . \] \end{lemma}

\begin{proof}  Suppose the entry of $A$ which has maximum absolute value occurs
in column $k$.  Then \[ \| A \|  = \sup_{\|\vecx\| = 1}\| A \mathbf{x} \|  \geq
\| A \vece_k \|  = \left( \sum_{i=1}^d |a_{ik}|^2 \right)^{1/2}  \geq \left(
\left( \max_i |a_{ik}| \right)^2 \right)^{1/2} = \max_{i,j} |a_{ij}| .\] For
the other inequality, let $\mathbf{x}$ be any vector such that $\| \mathbf{x}
\| = 1$, and write $\mathbf{y} = A \mathbf{x}$.  For $1 \leq j \leq d$, we have
$|x_j| \leq 1$, so \begin{align*} \| A \mathbf{x} \| & = \| \mathbf{y} \| \\ &
=  \left( \sum_{i=1}^d |y_i|^2 \right)^{1/2} \\ & \leq \sqrt{d} \max_i |y_i|
\\ & = \sqrt{d} \max_i \left| \sum_{j=1}^d a_{ij}x_j \right| \\ & \leq \sqrt{d}
d \max_{i,j} |a_{ij}x_j| \\ & \leq d^{3/2} \max_{i,j}|a_{ij}|. \end{align*}
Thus, \[ \| A \|  = \sup_{\| \mathbf{x} \| = 1}\| A \mathbf{x}\|  \leq d^{3/2}
\max_{i,j} |a_{ij}|.\] Note that we may take the constant $c_5$ to be $1$.
\end{proof}

\subsection{Conclusion of the proof of
Theorem~\ref{normthm}}\label{finishnormproof}

Let $\gamma$ be in $\sldz$. By Corollary~\ref{product1}, $\gamma$ may be
written as a word \[ \gamma = \delta_1 \delta_2 \cdots \delta_{d^2}, \] where
each $\delta_i \in SL^{s,t}(\mathbb{Z})$ for some $1 \leq s \not = t \leq d$.
Moreover, the entries of these $\delta_i$ are bounded by a fixed polynomial $P$ in the
entries of $\gamma$; we will write $P(\gamma)$ for
$P(\gamma_{11},\gamma_{12},\ldots,\gamma_{dd})$.  Then for $1 \leq i \leq d^2$,
the $(k,l)$ entry of the matrix $\delta_i$ satisfies \begin{equation}
\label{poly_bound}\left| \left( \delta_i \right)_{kl} \right| \leq |P
(\gamma)|. \end{equation} Let $\delta$ be the matrix in the expression
$\delta_1 \delta_2 \cdots \delta_{d^2}$ which has maximum norm. That is,
$\delta = \delta_i$ for some $1 \leq i \leq d^2$, and $\| \delta \| \geq \|
\delta_i \|$ for all $1 \leq i \leq d^2$. Then, by Corollary~\ref{word_log}, the
word distance from $1$ to $\gamma$ is bounded as follows: \[ d_W(1,\gamma) \leq
\sum_{i=1}^{d^2} d_W(1,\delta_i) \leq \sum_{i=1}^{d^2} K_6 \log \| \delta_i \|
\leq K_7 \log \| \delta \|,  \] where $K_7$ is some positive constant.   Using
Lemma~\ref{norm_ineqs}, and (\ref{poly_bound}), and writing $\delta =
(\delta_{kl})$, we then obtain \[ K_7 \log \| \delta \|  \leq K_7 \log \left(c_6
\max_{k,l} |\delta_{kl}|\right)  \leq K_7 \log \left(c_6 |P(\gamma)|\right). \] Let the degree
of the polynomial $P$ be $\deg(P)$, and let the sum of the the absolute values
of the coefficients of $P$ be $S$. Then for all $\gamma$, \[ |P(\gamma)|  \leq
S\left(\max_{i,j}|\gamma_{ij}|\right)^{\deg(P)}, \] so \begin{align*} K_7 \log(C_6
|P(\gamma)|) & \leq K_7
\log\left(c_6S(\max_{i,j}|\gamma_{ij}|)^{\degree(P)}\right) \\ &
\leq K_7 \log\left(c_6S\max_{i,j}|\gamma_{ij}|\right)^{\degree(P)} \\ & = K_7 \degree(P)
\log \left(c_6S\max_{i,j}|\gamma_{ij}|\right) \\ & \leq K \log
\left(\max_{i,j}|\gamma_{ij}|\right)\end{align*}  for some constant $K > 0$. Finally,
using Lemma~\ref{norm_ineqs} again, \[ K \log \left(\max_{i,j} | \gamma_{ij} |
\right)
\leq K \log\| \gamma \|. \] Thus \[ d_W(1,\gamma) \leq K \log \| \gamma \|, \]
and the proof of Theorem~\ref{normthm} is complete.


\clearpage
\addcontentsline{toc}{chapter}{References}
\bibliographystyle{plain} 
\bibliography{allsources}
\nocite{nor1:gcekr}
\nocite{bosmeh1:imsh}
\nocite{hod1:mrlc}
\nocite{boy1:hm}
\nocite{gray1:int}
\nocite{gray2:gfi}
\nocite{gray3:is}
\nocite{kli1:mtamt}
\nocite{mil1:hg}
\nocite{mil2:eptg}
\nocite{pla1:m}
\nocite{poi3:sh}
\nocite{pro1:ce}
\nocite{rat1:fhm}
\nocite{ros1:hneg}
\nocite{sti2:int}
\nocite{tor1:pgrp}
\nocite{wal1:nemr}

\end{document}